\documentclass[11pt]{amsart}

\usepackage[makeroom]{cancel}

\newif\ifdebug
\debugfalse

\newif \iffig
\figtrue

\newif \iftable
\tablefalse

\usepackage{sfilip_package_settings}
\usepackage{sfilip_abbreviations}
\usepackage{sfilip_thm_style_long}

\title[Cyclotomic family of thin monodromy in $\operatorname{Sp}_4(\mathbb{R})$]{A cyclotomic family of thin hypergeometric monodromy groups in $\operatorname{Sp}_4(\mathbb{R})$}

\hypersetup{
	pdftitle={A cyclotomic family of thin hypergeometric monodromy groups in Sp4(R)},	
	pdfauthor={Simion Filip, Charles Fougeron},	
	pdfsubject={hypergeometric equations, monodromy, thin groups, ping-pong},		
	pdfkeywords={MSC class:
33C05, 32S40, 32G20, 20F65, 37D20},	
	pdfnewwindow=true,		
	linkcolor=darkblue,		
	citecolor=darkred,		
	filecolor=darkblue,		
	urlcolor=darkblue,		
	pdfborder={0 0 0},
	breaklinks=true
}

\thanks{Revised \textsc{\today}, 
MSC class:
33C05, 32S40, 32G20, 20F65, 37D20}

\date{May 2021}

\author{
	Simion Filip
}
\address{
	\parbox{0.5\textwidth}{
		Department of Mathematics\\
		University of Chicago\\
		5734 S University Ave\\
		Chicago, IL 60637}
}
\email{{sfilip@math.uchicago.edu}}

\author{
	Charles Fougeron
}
\address{
	\parbox{0.5\textwidth}{
		IRIF - Université de Paris\\
		8 place Nemours\\
		Bureau 3302\\
		75013 Paris}
}
\email{{charles.fougeron@math.cnrs.fr }}
\begin{document}

\begin{abstract}
	We exhibit an infinite family of discrete subgroups of $\Sp_4(\bR)$ which have a number of remarkable properties.
	Our results are established by showing that each group plays ping-pong on an appropriate set of cones.
	The groups arise as the monodromy of hypergeometric differential equations with parameters $\left(\tfrac{N-3}{2N},\tfrac{N-1}{2N}, \tfrac{N+1}{2N}, \tfrac{N+3}{2N}\right)$ at infinity and maximal unipotent monodromy at zero, for any integer $N\geq 4$.

	Additionally, we relate the cones used for ping-pong in $\bR^4$ with crooked surfaces, which we then use to exhibit domains of discontinuity for the monodromy groups in the Lagrangian Grassmannian.
\end{abstract}

%
\maketitle
%
\noindent \hrulefill
\tableofcontents
\noindent \hrulefill
\ifdebug
   \listoffixmes
\fi


\section{Introduction}
	\label{sec:introduction}

The monodromy of hypergeometric differential equations has been actively studied for a long time.
A historical overview going back to the 19th century can be found in the book of Gray \cite{Gray2008_Linear-differential-equations-and-group-theory-from-Riemann}, and the more recent developments relevant to the present text started with the work of Beukers--Heckman \cite{BeukersHeckman1989_Monodromy-for-the-hypergeometric-function-nFn-1} who analyzed the basic features of the monodromy groups of hypergeometric equations on $\bP^1(\bC)\setminus \{0,1,\infty\}$.
In particular, they described the Zariski closure of the discrete groups which arise.

A more refined question about the monodromy group is what is its relation to the ambient arithmetic lattice.
The most interesting case is when the monodromy representation identifies the  fundamental group of the base (in an orbifold sense) with the corresponding arithmetic lattice, and this leads to uniformization of algebraic manifolds by domains.
A representative example is the congruence subgroup $\Gamma(2)$ of $\SL_2(\bZ)$.
In a different direction, the representation can surject (with large kernel) onto a finite index subgroup of the arithmetic group.
Results of this nature have been obtained recently by Singh--Venkataramana \cite{SinghVenkataramana2014_Arithmeticity-of-certain-symplectic-hypergeometric-groups} (see also Detinko--Flannery--Hulpke \cite{DetinkoFlanneryHulpke2018_Zariski-density-and-computing-in-arithmetic-groups}).
Finally, the image of the representation\footnote{So far, all known thin examples appear also to be (essentially) injective.} can be an infinite index subgroup of the lattice, which is called ``thin''.
This is the case of interest to us.

An infinite family of thin monodromy groups has been obtained by Fuchs, Meiri, and Sarnak \cite{FuchsMeiriSarnak2014_Hyperbolic-monodromy-groups-for-the-hypergeometric-equation-and-Cartan-involutions}.
These are discrete subgroups of the indefinite orthogonal group $\SO_{1,n}(\bR)$, finitely many for each $n$, and arbitrarily large $n$.
By different methods, Brav and Thomas \cite{BravThomas2014_Thin-monodromy-in-Sp4} exhibited 7 parameters for which the monodromy group is thin in $\Sp_4(\bZ)\subset \Sp_4(\bR)$.

In this text, we extend the methods of Brav--Thomas and exhibit an infinite family of parameters for which the monodromy yields discrete subgroups in $\Sp_4(\bR)$.
These are moreover thin, when viewed inside appropriately chosen lattices.
Note that as our ambient dimension is fixed at $4$, the matrix entries of the groups will necessarily lie in number fields of increasing size.
This family of parameters initially emerged from numerical experiments on Lyapunov exponents of hypergeometric differential equations \cite{Fougeron19}.

\subsubsection*{Parameters for hypergeometric equations}
	\label{sssec:parameters_for_hypergeometric_equations}
We will consider rank $4$ hypergeometric groups with maximal unipotent monodromy at zero (see \autoref{sec:background_on_hypergeometric_groups} for more background on hypergeometric equations).
This leads to the differential operator
\[
	D^4 - z
	(D+\alpha_1)
	(D+\alpha_2)
	(D+\alpha_3)
	(D+\alpha_4) \quad D = z\partial_z
\]
which has regular singular points at $0,1,\infty \in \bP^1(\bC)$.
The $\alpha$-parameters that we consider are $\left(\frac{N-3}{2N}, \frac{N-1}{2N}, \frac{N+1}{2N}, \frac{N+3}{2N}\right)$ for $N\geq 4$.
We let $\Gamma_N\subset \Sp_4(\bR)$ be the monodromy group of the equation, and $P\Gamma_N\subset \PSp_4(\bR)$ its image in the projectivized symplectic group.
Let $R$ be the monodromy at infinity and $T$ the monodromy around $1$.

\begin{theoremintro}[Thin, discrete monodromy]
	\label{thmintro:thin_discrete_monodromy}
	The projective monodromy group is isomorphic, as an abstract group, to:
	\[
		P\Gamma_N \isom \ip{R,T \, | \, R^N = 1 \,}
	\]
	Furthermore, it is a discrete subgroup of $\PSp_4(\bR)$.

	The full monodromy group $\Gamma_N$ is isomorphic to $P\Gamma_N$ if $N$ is odd, and is a $\bZ/2$ central extension if $N$ is even.
	It is also a discrete subgroup of $\Sp_4(\bR)$.

	Furthermore, denote by $Y_N$ the orbifold $\bP^1\setminus \{0,1\}$ with an orbifold point of order $N$ at infinity.
	Then the monodromy representation is an isomorphism of $\pi_1^{orb}(Y_N)$ and $P\Gamma_N$.
\end{theoremintro}

\begin{proof}
	In \autoref{ssec:reflection_structure} we construct a group generated by three reflections $\tilde{\Gamma}_N$ that contains $\Gamma_N$ with index $2$.
	The reflections are denoted $A,B,C$ and $\Gamma_N$ is mapped to $\tilde{\Gamma}_N$ via $R\mapsto BC$ and $T\mapsto AB$.
	In \autoref{thm:ping_pong_property_of_cones} we show that $\tilde{\Gamma}_N$ acts on a set of $2N$ cones in $\bP(\bR^4)$ in a (generalized) ping-pong manner.
	It follows that for its image in the projective group we have:
	\[
		P\tilde{\Gamma}_N\isom \ip{A,B,C \, | \, A^2=B^2=C^2=1, (BC)^N=1}.
	\]
	Then $\tilde{\Gamma}_N$ is either a $\bZ/2$-extension of $P\tilde{\Gamma}_N$ or isomorphic to it, according to whether $R^{N}=-\id$ ($N$ is even) or $R^N=\id$ ($N$ is odd).

	Discreteness follows from \autoref{thmintro:domain_of_discontinuity} below which shows that the groups have nonempty domains of discontinuity in the Lagrangian Grassmannian.
\end{proof}

\subsubsection*{log-Anosov property and further consequences}
Our method of proof, namely giving cones on which the group plays ping-pong, has many further consequences which are developed in \cite{sfilip_prep}.
Namely, the monodromy groups $\Gamma_N$ are examples of Anosov representations introduced by Labourie \cite{Labourie2006_Anosov-flows-surface-groups-and-curves-in-projective-space} (see also \cite{GuichardWienhard2012_Anosov-representations:-domains-of-discontinuity-and-applications} for further developments of the notion, as well as \cite{KapovichLeebPorti2018_Dynamics-on-flag-manifolds:-domains-of-proper-discontinuity-and-cocompactness,GueritaudGuichardKassel2017_Anosov-representations-and-proper-actions}), except that the definition needs to be adapted in order to allow for unipotents.
Such an extension has been provided by Kapovich and Leeb \cite{KapovichLeeb_Relativizing-characterizations-of-Anosov-subgroups-I}, Zhu
\cite{Zhu2019_Relatively-dominated-representations}, and in \cite{sfilip_prep}.

Let $G_N$ be the reflection group generated by the triangle in hyperbolic space with a point of angle $\pi/N$ in $\bH^2$ and two points at infinity.
The content of \autoref{thmintro:thin_discrete_monodromy} is that $G_N$ is mapped isomorphically onto a subgroup $P\tilde{\Gamma}_N\subset \PSp_4(\bR)$.
In fact, the method of proof also implies:
\begin{enumerate}
	\item There exists a $G_N$-equivariant continuous (\Holder) map
	\[
		\xi\colon \partial \bH^2 \to \bP(\bR^4)
	\]
	This is illustrated in Figure 1 below.
	\item There exists a nonempty open set $\Omega\subset \LGr(\bR^4)$ in the Lagrangian Grassmannian on which $P\tilde{\Gamma}_N$ acts properly discontinuously.
	\item The formula for the sum of Lyapunov exponents from \cite{EskinKontsevichMoller2018_Lower-bounds-for-Lyapunov-exponents-of-flat-bundles-on-curves} holds.
	\item The group $P\tilde{\Gamma}_N$ acts properly discontinuously on $\LGr^{1,1}(\bC^4)$, the Grassmannian of Lagrangians on which the indefinite hermitian pairing of signature $(2,2)$ restricts to signature $(1,1)$.
\end{enumerate}
The third point on Lyapunov exponents was exactly the property observed numerically and conjectured in \cite{Fougeron19} which motivated the study of this family of parameters.
Note that in weight $2$, for variations of Hodge structure of K3 type, the formula for the sum of Lyapunov exponents was established in \cite{Filip2018_Families-of-K3-surfaces-and-Lyapunov-exponents}.

The reader familiar with Hodge theory will recognize that $\LGr^{1,1}(\bC^4)$ is the target of a forgetful map from the Griffiths period domain of Hodge structures with Hodge numbers $(1,1,1,1)$, given by forgetting the first term of the Hodge filtration.
This is in contrast to Siegel space, which consists of Lagrangians for which the restricted hermitian pairing has signature $(2,0)$.
The action of a discrete group in $\Sp_4(\bR)$ has no apriori reason to act properly on $\LGr^{1,1}(\bC^4)$, even though it always acts properly on Siegel space.
It is also established in \cite{sfilip_prep} that the quotient of the domain of discontinuity $\Omega\subset \LGr(\bR^4)$ by $P\tilde{\Gamma}_N$ can be identified with a circle bundle over the orbifold $\bH^2/G_N$, constructed using Hodge theory.

\subsubsection*{Crooked surfaces}
Drumm \cite{Drumm1992_Fundamental-polyhedra-for-Margulis-space-times} introduced crooked surfaces to construct fundamental domains for discrete groups acting on Minkowski space.
These have found further applications in Lorenzian geometry, see e.g. \cite{DancigerGueritaudKassel2016_Geometry-and-topology-of-complete-Lorentz-spacetimes-of-constant-curvature} for some recent applications.
Let us note that in contrast to hyperbolic or euclidean spaces, where totally geodesic hyperplanes are natural and effective tools for constructing fundamental domains of group actions, in higher rank situations such obvious choices are not available.
Crooked surfaces have been effective in constructing fundamental domains in $3$-dimensional homogeneous spaces with Lorentz metrics (or conformal classes thereof).

In \autoref{sec:symplectic_geometry_and_causality} we connect the cones that are used to prove \autoref{thmintro:thin_discrete_monodromy} to crooked surfaces.
As it turns out, many properties of crooked surfaces can be conveniently expressed using cones, via the dictionary relating symplectic geometry in $\bR^4$ to the causal geometry of the projectivized null vectors in $\bR^{2,3}$, the latter being just the Lagrangian Grassmannian of $\bR^4$.
Most importantly for us, the criteria establishing disjointness of crooked surfaces developed in \cite{BurelleCharetteFrancoeur2021_Einstein-tori-and-crooked-surfaces} are concisely expressed by the containment of cones established during the ping-pong argument.
We need to further extend their criterion to allow the crooked surfaces to touch, see \autoref{ssec:disjointness_of_crooked_surfaces} for details.
In particular, we can explicitly analyze the action of our groups on $\LGr(\bR^4)$ and obtain:

\begin{theoremintro}[Domain of discontinuity]
	\label{thmintro:domain_of_discontinuity}
	For each $N\geq 4$ there exists a nonempty open set $\Omega_N\subset \LGr(\bR^4)$ on which $\Gamma_N$ acts properly discontinuously.
\end{theoremintro}
See \autoref{thm:proper_discontinuity} and the discussion preceding it.

\begin{figure}[htbp!]
	\centering
	\includegraphics[width=0.326\linewidth]{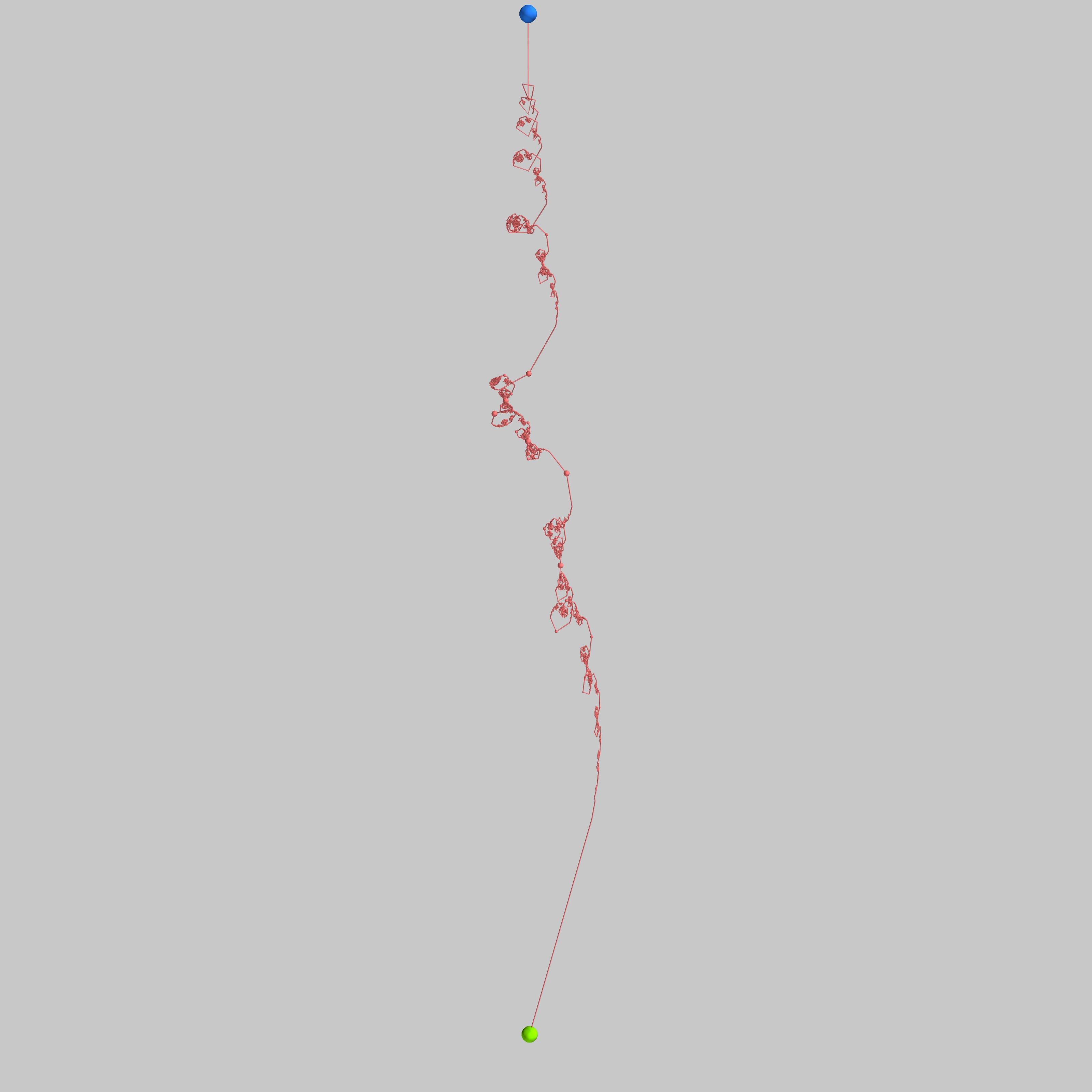}
	\includegraphics[width=0.326\linewidth]{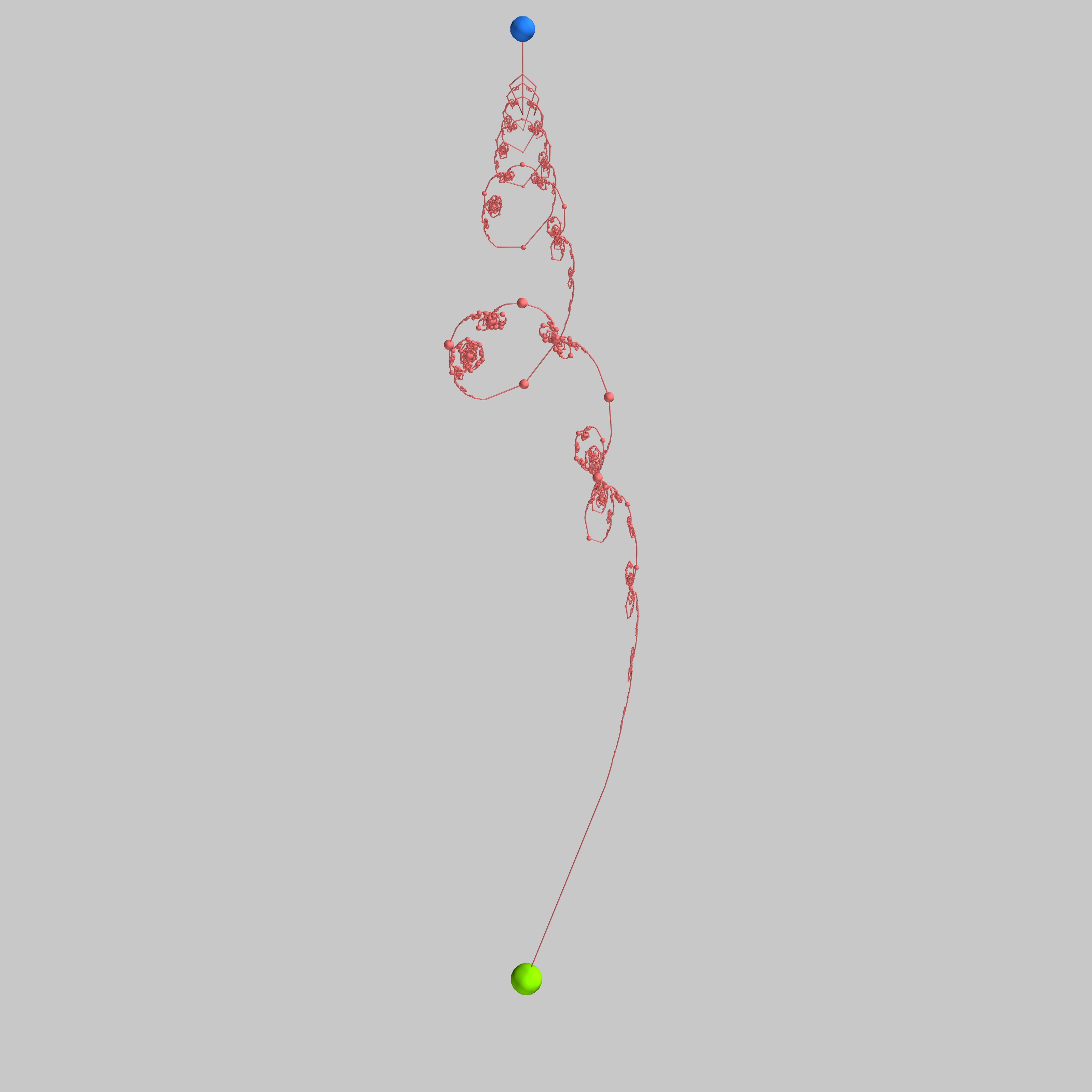}
	\includegraphics[width=0.326\linewidth]{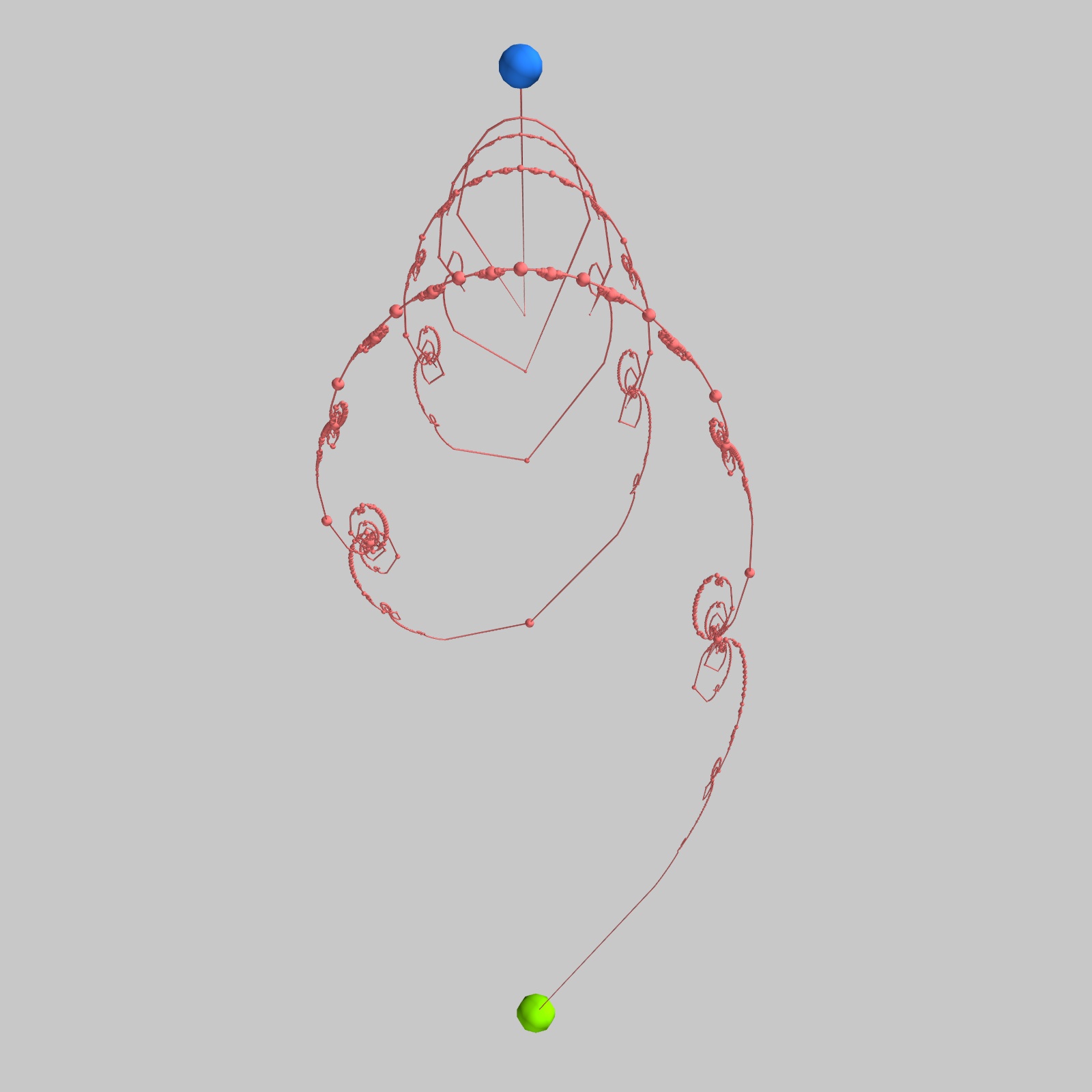}\\
	\includegraphics[width=0.326\linewidth]{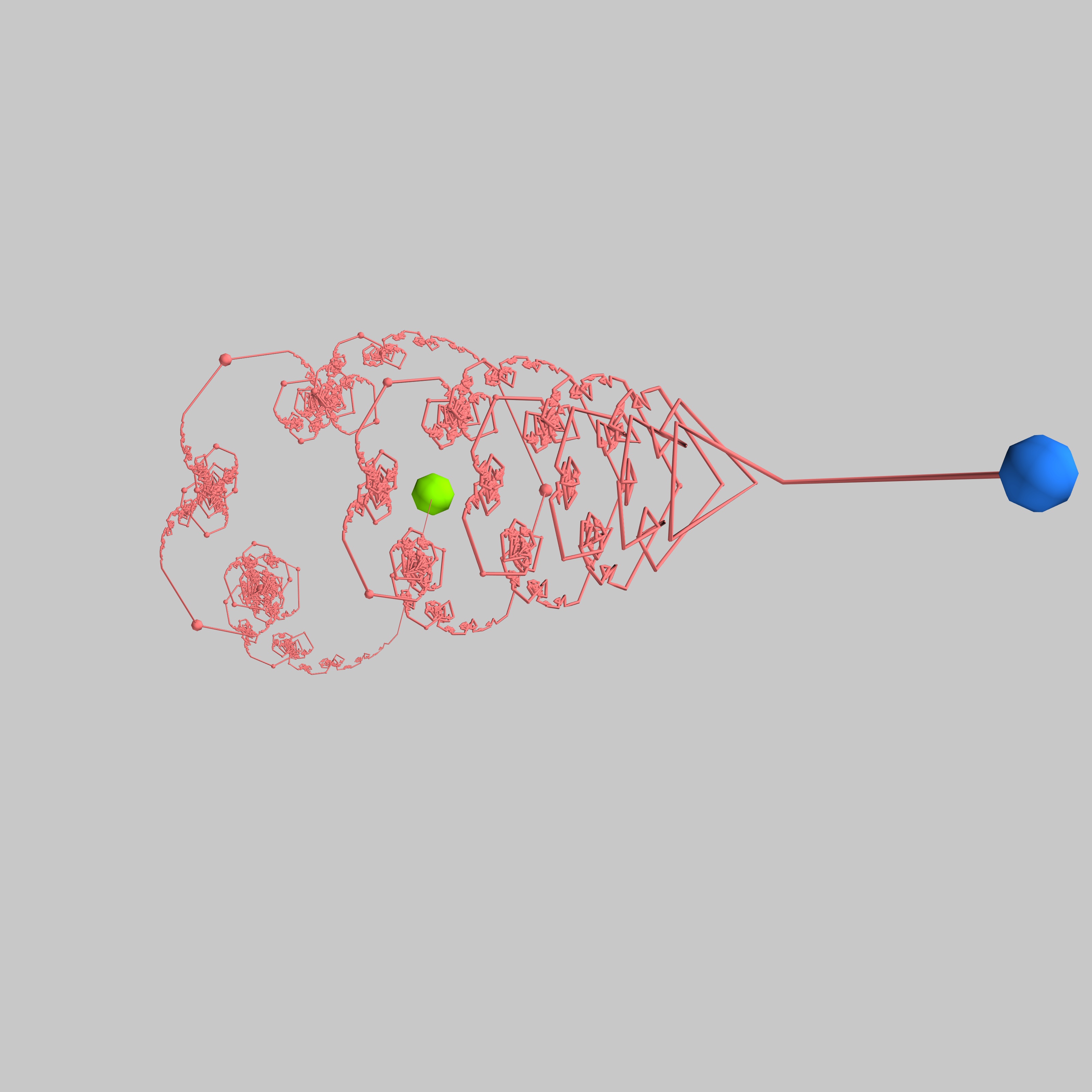}
	\includegraphics[width=0.326\linewidth]{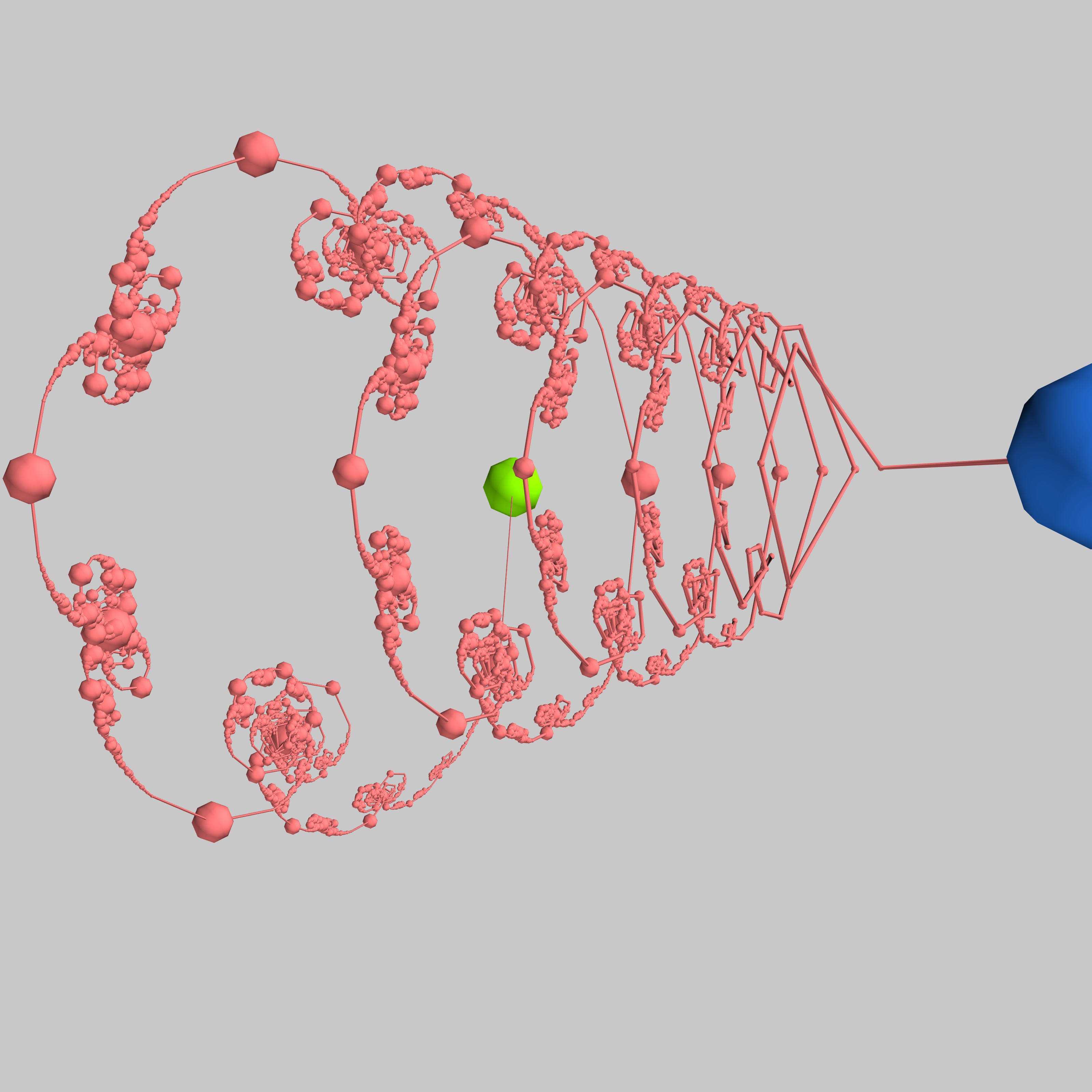}
	\includegraphics[width=0.326\linewidth]{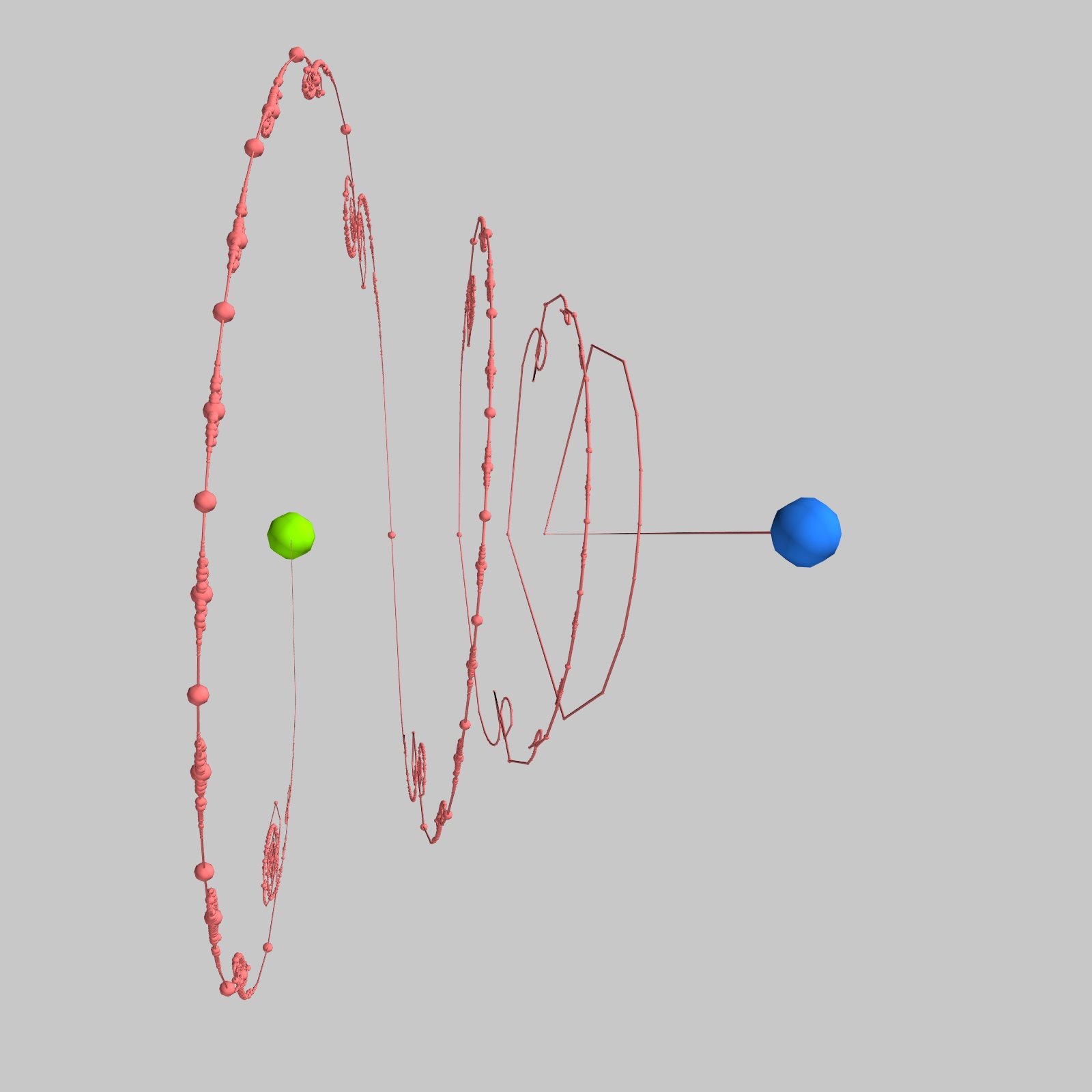}\\
	\includegraphics[width=0.326\linewidth]{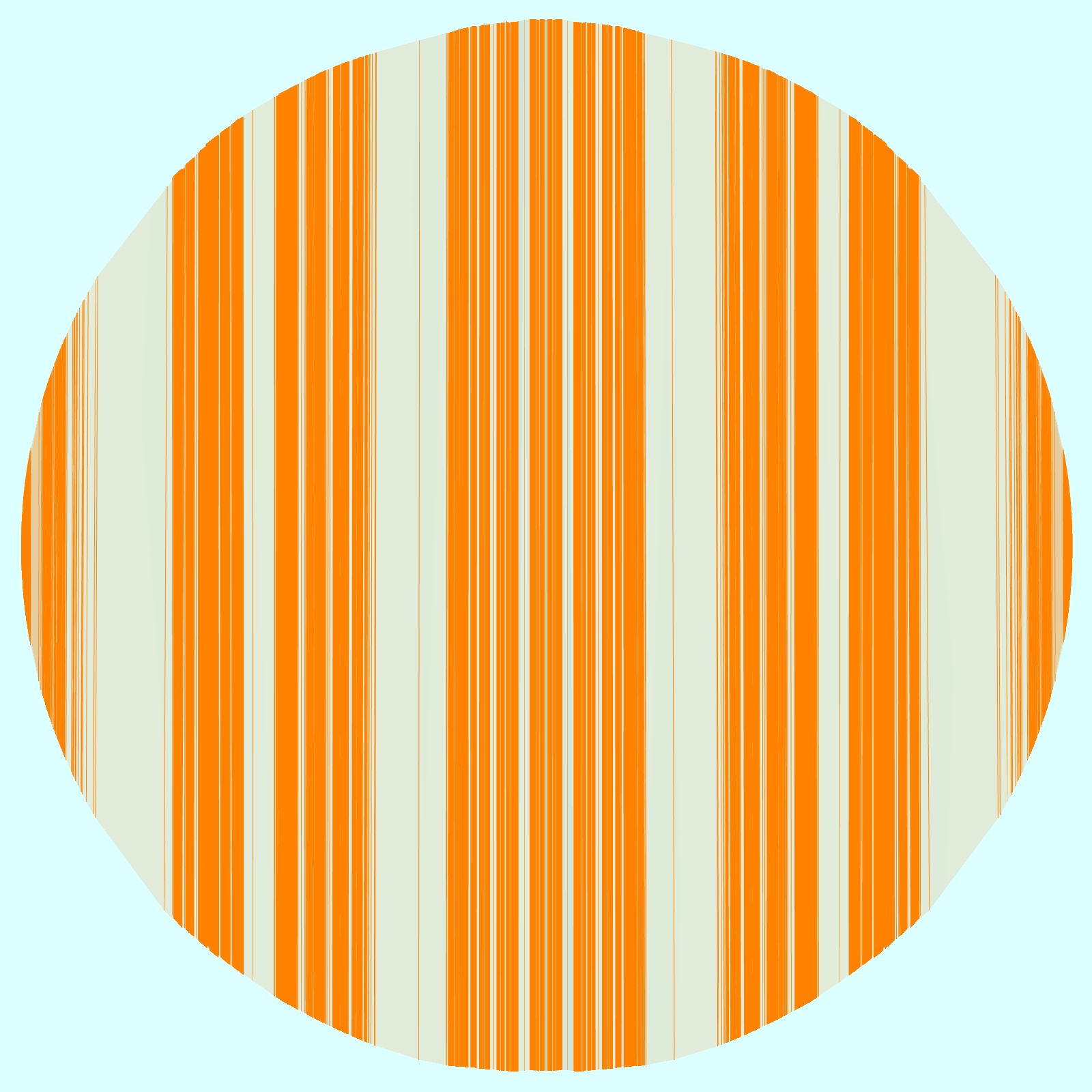}
	\includegraphics[width=0.326\linewidth]{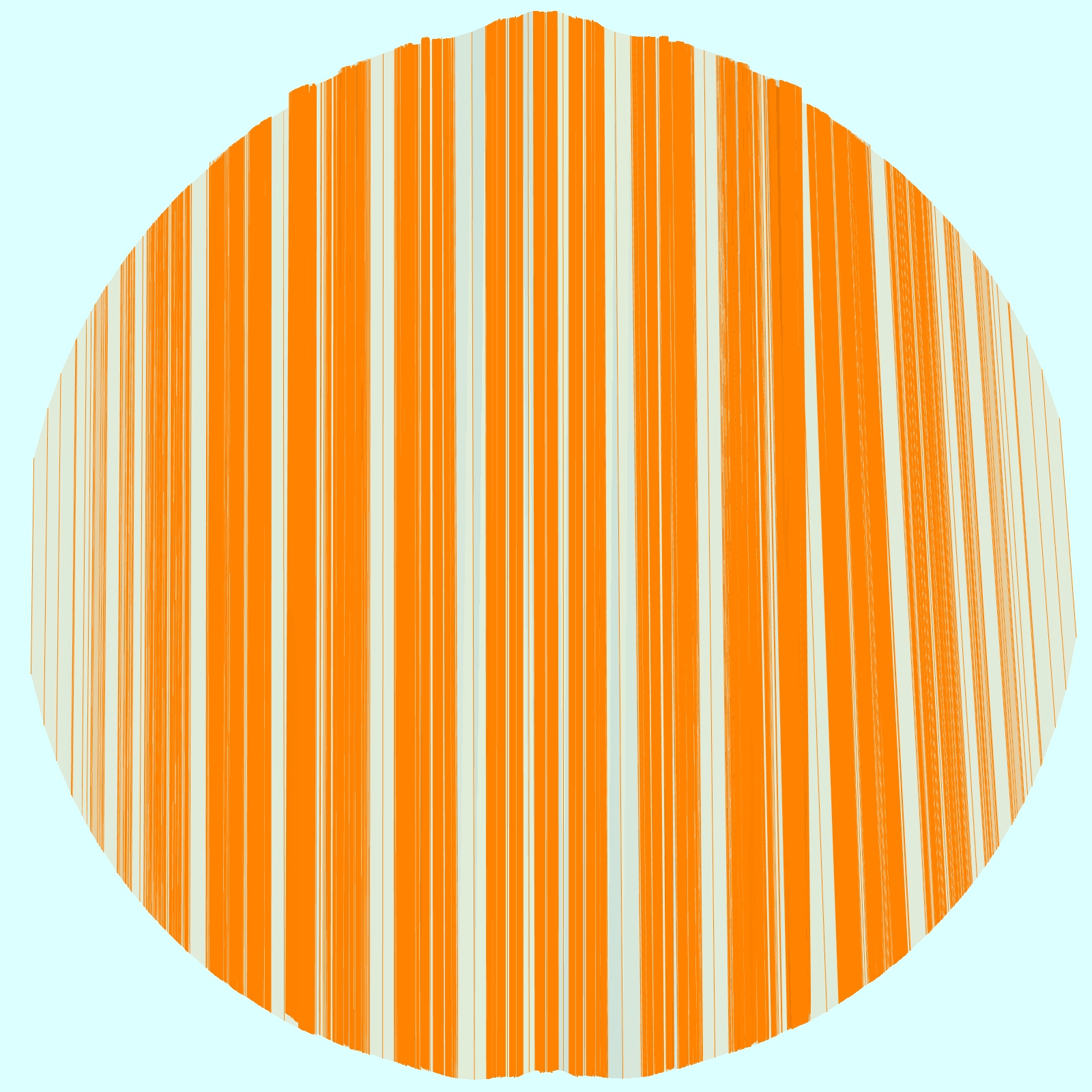}
	\includegraphics[width=0.326\linewidth]{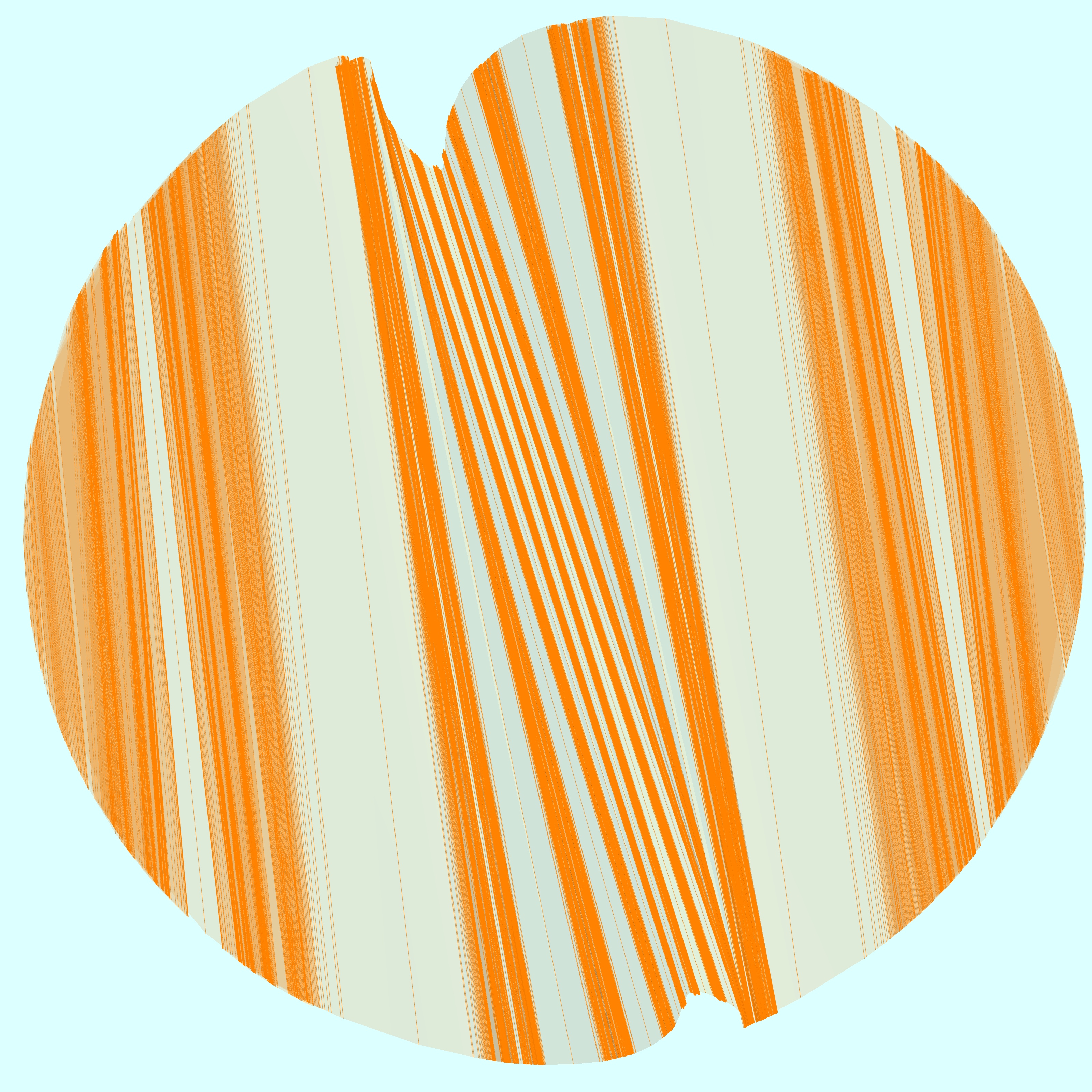}\\
	\includegraphics[width=0.326\linewidth]{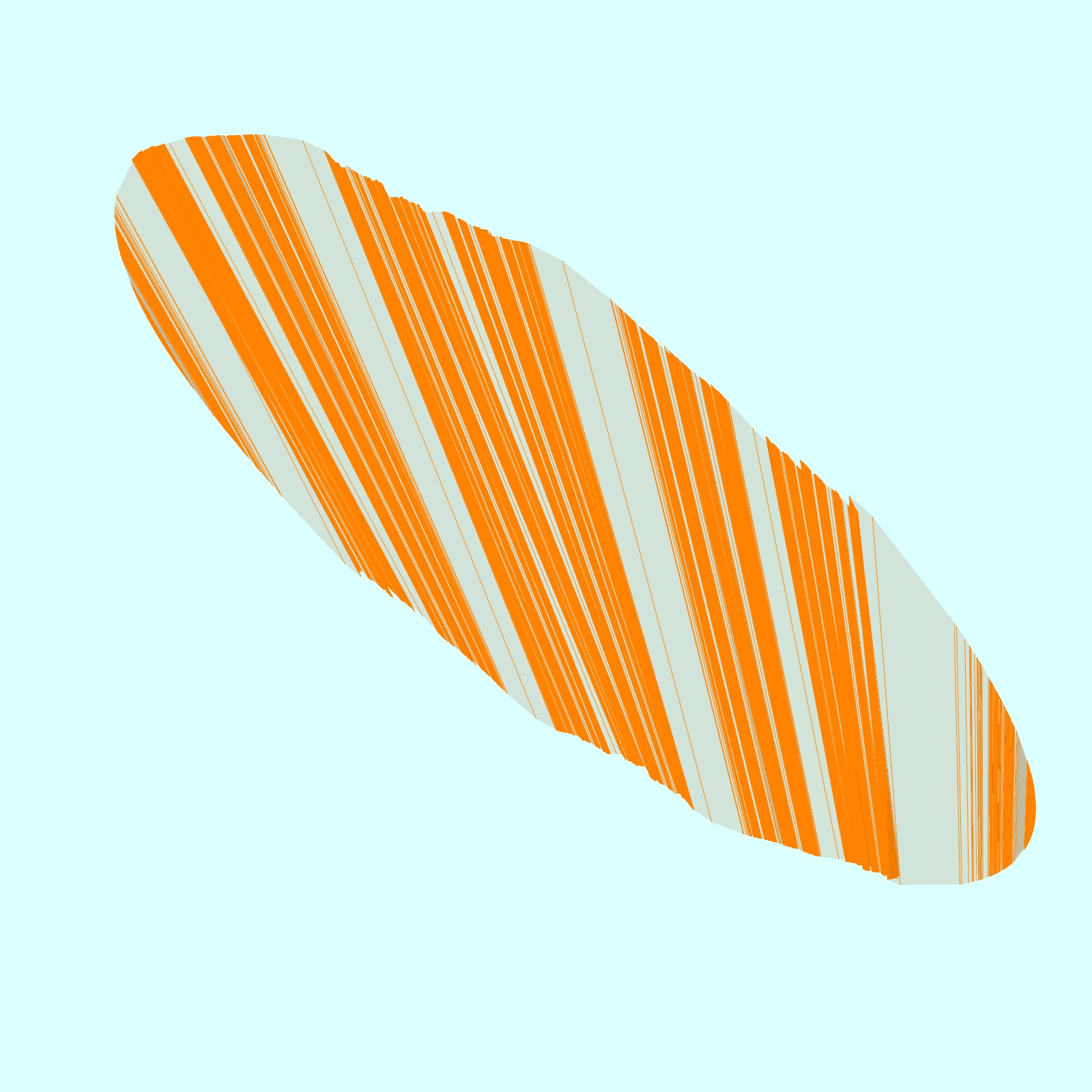}
	\includegraphics[width=0.326\linewidth]{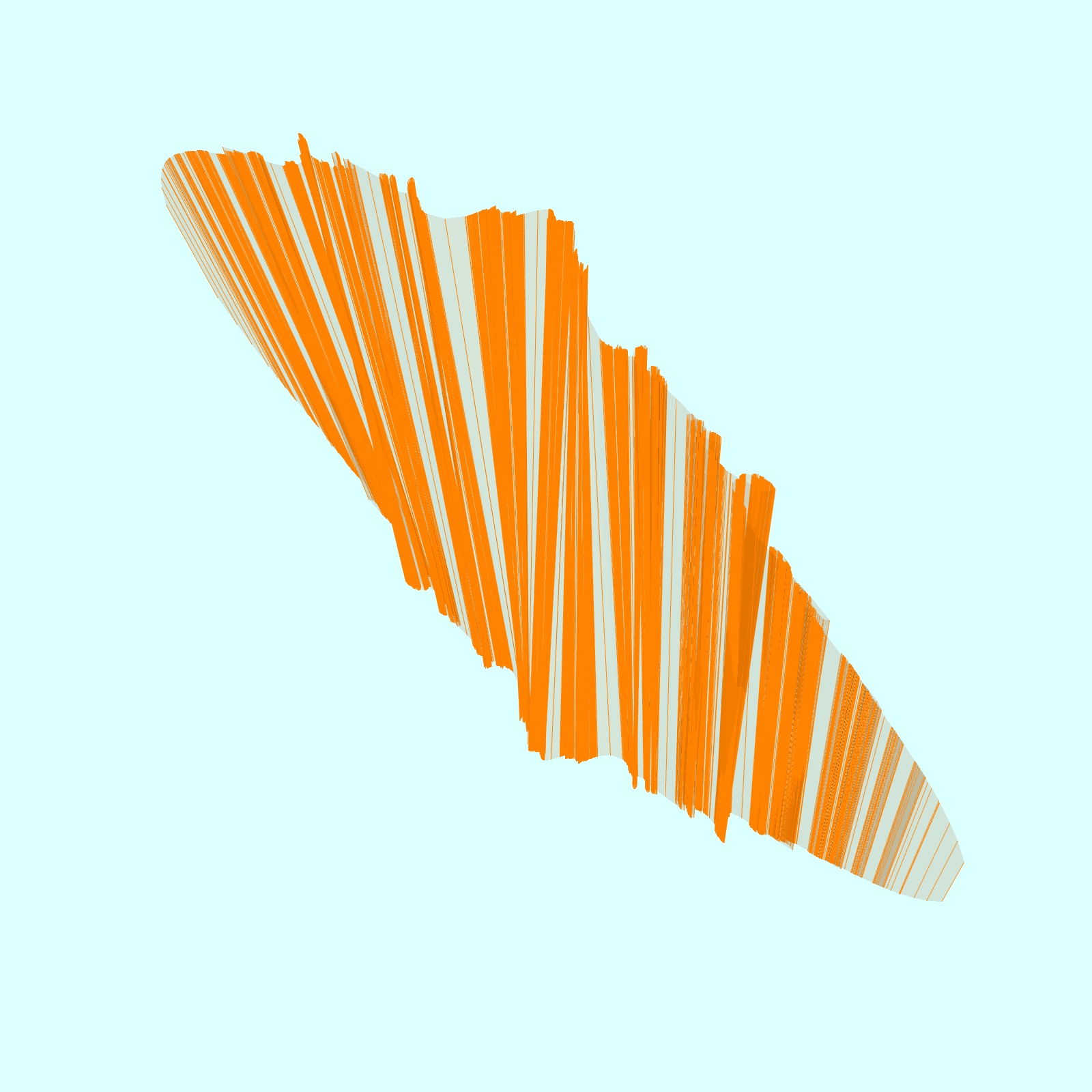}
	\includegraphics[width=0.326\linewidth]{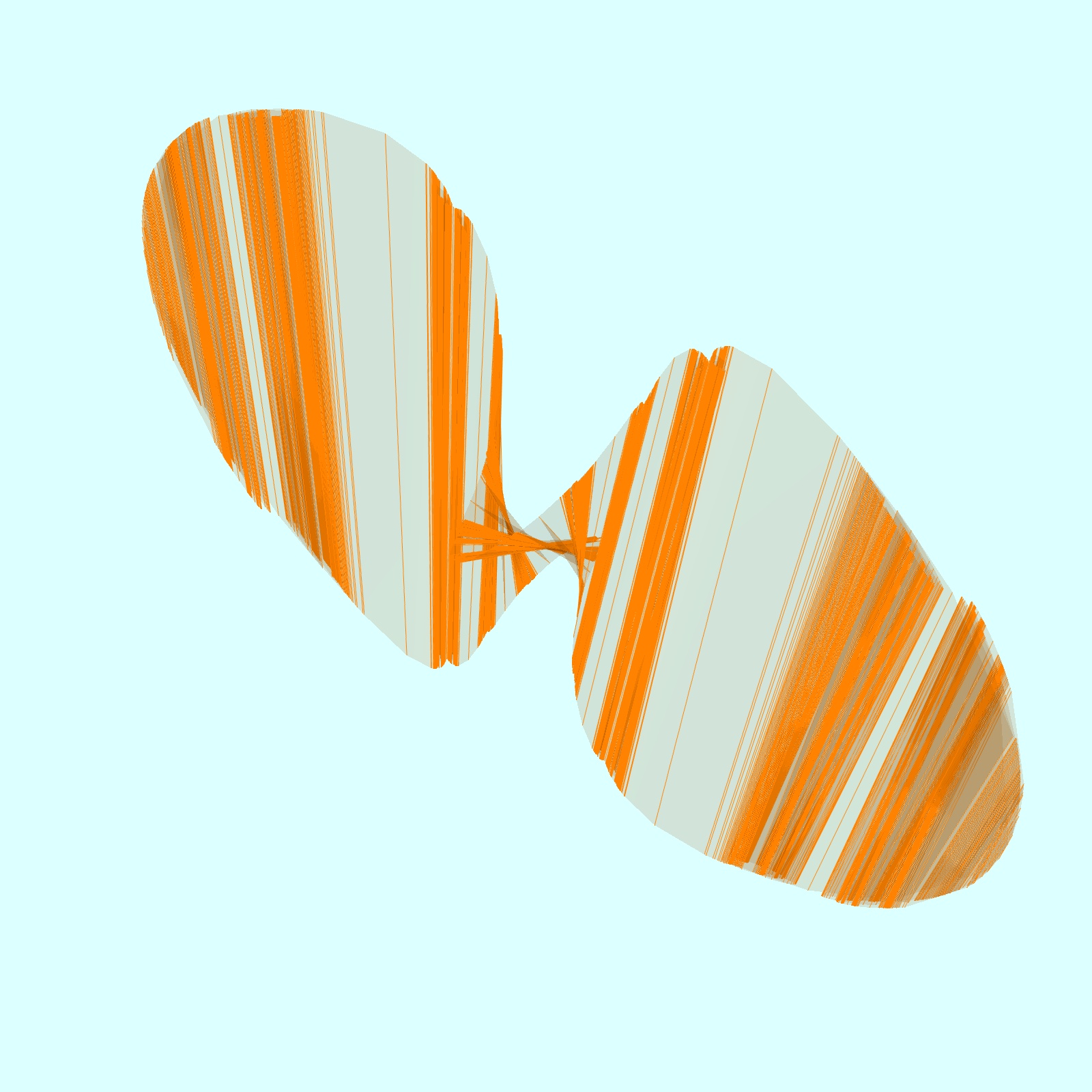}
	\caption*{\textsc{Figure 1.}
	}
	\label{fig:limit_curve}
\end{figure}

\subsubsection*{On Figure 1}
We include some numerical simulations demonstrating the limit curves and surfaces.
The limit surface is the complement of the domain of discontinuity $\Omega_N$ from \autoref{thmintro:domain_of_discontinuity}.
The parameter $N$ is fixed in each column and is equal to $N=4,5$ and $11$ respectively.
Each row gives a view of the limit curve in $\bP(\bR^4)$, and limit surface in $\LGr(\bR^4)$, from roughly the same position.
Only the part of the limit curve between the points stabilized by the MUM, resp. rank $1$ unipotent, is displayed.
The limit surface is intersected with a Euclidean sphere and displayed in a chart of $\LGr(\bR^4)$ which is conformally equivalent to Minkowski space $\bR^{2,1}$ (the chart is given as the complement of the nullcone of one Lagrangian).

\subsubsection*{On thinness}
	According to the customary definition, see e.g. the survey of Sarnak \cite{Sarnak_Notes-on-thin-matrix-groups}, a thin group is one which is of infinite index in an arithmetic lattice.
	Let us explain how this arises in our context.

	Let $\bQ(\zeta_{2N})$ be an extension of $\bQ$ obtained by adjoining a $2N$-th root of unity, $\bQ(\zeta_{2N})^+$ its totally real subfield, and $\cO_{2N}^+$ the corresponding ring of integers.
	Plainly from the definitions of hypergeometric groups (see e.g. \cite{BeukersHeckman1989_Monodromy-for-the-hypergeometric-function-nFn-1}) it follows that the monodromy matrices can be expresses with entries in $\cO_{2N}^+$.
	However, the multiplicatively invertible residue classes $(\bZ/2N)^{\times}$ have a Galois action on the roots of unity, and we can identify those with $\left(\tfrac{1}{2N}\bZ\right)/\bZ$, or rationals in $[0,1)$ with denominator $2N$.
	Then a subgroup will stabilize our given four-tuple $\alpha^{(0)}_{\bullet}=\left(\tfrac{N-3}{2N},\tfrac{N-1}{2N}, \tfrac{N+1}{2N}, \tfrac{N+3}{2N}\right)$.
	We also have the orbit of our four-tuple under this multiplicative action, with representatives (all mod $1$) $\alpha_{\bullet}^{(i)}$, say a total of $k$ distinct representatives.

	This defines a subfield $K_N\subset \bQ(\zeta_{2N})^+$ such that the Galois group of $\bQ(\zeta_{2N})/K_N$ is precisely the stabilizer in $(\bZ/2N)^\times$ of our original four-tuple (note that $-1$ stabilizes our four-tuple so the subfield is totally real).
	If we denote by $\cO_{K_N}$ the ring of integers in $K_N$ then our monodromy group $\Gamma_N$ embeds in $\Sp_4(\cO_{K_N})$.
	This last group is an arithmetic lattice in $\Sp_4(\bR)^k$ (a product of $k$ copies of $\Sp_4(\bR)$) where $k$ is the number of four-tuples obtained by the multiplicative action on our original one.
	The projection of $\Gamma_N\into \Sp_4(\cO_{K_N})\into \Sp_4(\bR)^k$ to any of the $\Sp_4(\bR)$-factors yields the Galois-conjugate local systems of our original one.

	The monodromy group $\Gamma_N$ is visibly discrete in the product $\Sp_4(\bR)^k$ since it is contained in the discrete lattice there.
	But \autoref{thmintro:thin_discrete_monodromy} implies that $\Gamma_N$ is in fact discrete when projected to one of the factors.
	This is similar to the classical constructions of Deligne--Mostow \cite{DeligneMostow1986_Monodromy-of-hypergeometric-functions-and-nonlattice-integral-monodromy} of non-arithmetic lattices in $\SU(1,n)$, with the difference that our group $\Sp_4(\bR)$ is of higher rank.

	Note that the Galois-conjugate monodromy representations yield groups which are abstractly isomorphic to the original one.
	However, the discreteness part of \autoref{thmintro:thin_discrete_monodromy} has no reason to extend to the Galois-conjugate local systems (and we suspect it does not hold in general).

	Let us remark that Veech groups, which arise in \Teichmuller dynamics, also yield lattices in $\SL_2(\bR)$ but also yield thin groups in Hilbert modular groups, which are arithmetic lattices inside products of $\SL_2(\bR)$.
	The relation between these lattices and Hodge theory was investigated by M\"{o}ller \cite{Moller2006_Variations-of-Hodge-structures-of-a-Teichmuller-curve}, and in higher rank in \Teichmuller dynamics in \cite{Filip_Semisimplicity-and-rigidity-of-the-Kontsevich-Zorich-cocycle}.
	See also McMullen's recent investigation
	\cite{McMullen2020_Billiards-heights-and-the-arithmetic-of-nonarithmetic-groups} in this direction, and Zorich's survey \cite{Zorich2006_Flat-surfaces} for further background in \Teichmuller dynamics.

	Let us finally remark that the $\bR$-Zariski density of $\Gamma_N$ inside $\Sp_4(\bR)^k$ follows from the combination of the results of Beukers--Heckman \cite{BeukersHeckman1989_Monodromy-for-the-hypergeometric-function-nFn-1}, which establish Zariski density in each factor separately, and Goursat's lemma in group theory (combined with the fact that $\Sp_4$ is simple).

\subsubsection*{Numerical experiments}
	Our work started from observations on the numerical behavior of the action of monodromy groups on the cones that can be accessed at \url{https://gitlab.com/fougeroc/ping-pong}.
	We have also used symbolic computation tools from SageMath \cite{sagemath} and our final worksheet can be found at \url{https://gitlab.com/fougeroc/notebook-cyclotomic-family}.
	This can be handy, but not logically necessary, for the reader who wants to follow our computations.

\subsubsection*{Related work, and generalizations}
	\label{sssec:related_work}
After a first version of this text was released, we learned from Fanny Kassel that together with Jean-Philippe Burelle, they have a forthcoming paper \cite{Burelle_Kassel} which contains and generalizes some of the contents of our \autoref{sec:symplectic_geometry_and_causality} on crooked surfaces.
Specifically, the interpretation of crooked surfaces in the Einstein universe $\Ein^{1,2}\simeq \LGr(\bR^4)$ in terms of projective simplices in $\bP(\bR^4)$, as well as the simpler interpretation of the disjointness criterion from \cite{BurelleCharetteFrancoeur2021_Einstein-tori-and-crooked-surfaces}, is also contained in their work and was known to them in 2018.
Their work also contains an extension to higher dimensions.
We arrived independently and unaware of their work at the results of \autoref{sec:symplectic_geometry_and_causality}.

\subsubsection*{Acknowledgments}
	This work was supported by the Agence Nationale de la Recherche through the project Codys (ANR 18-CE40-0007).
	This material is based upon work supported by the National Science Foundation under Grant No. DMS-2005470 (SF) and DMS-1638352 (at the IAS).

	This research was partially conducted during the period the first-named author served as a Clay Research Fellow.
	SF also gratefully acknowledges support from the Institute for Advanced Study.
	Part of this work was conducted while the authors were in residence at the Mathematical Sciences Research Institute in Berkeley, California, during the Fall 2019 semester, with support from U.S. National Science Foundation grants DMS-1107452, 1107263, 1107367 ``RNMS: Geometric Structures and Representation Varieties'' (the GEAR Network), as well as Grant No. DMS-1440140 (MSRI).



\section{Background on hypergeometric groups}
	\label{sec:background_on_hypergeometric_groups}

\paragraph{Outline of section}
In \autoref{ssec:notation} we recall some basic definitions regarding hypergeometric differential equations and their monodromy.
Next, in \autoref{ssec:reflection_structure} we extend with index $2$ the monodromy group to make it generated by reflections.


\subsection{Notation}
	\label{ssec:notation}

We recall here some standard facts on hypergeometric groups.
See \cite{BeukersHeckman1989_Monodromy-for-the-hypergeometric-function-nFn-1} or \cite{Yoshida1997_Hypergeometric-functions-my-love} for further background.

\subsubsection{Setup}
	\label{sssec:setup_notation}
Fix two $n$-tuples of reals $\{\alpha_i\}_{i=1}^{n},\{\beta_i\}_{i=1}^{n}$ subject to the normalizations $\alpha_i\in[0,1)$ and $\beta_i\in (0,1]$.
Note that most classical normalizations, which involve explicit hypergeometric functions, take $\beta_n=1$, and in our case we will take $\beta_i=1,\forall i$ to ensure maximal unipotent monodromy at $0$ (another popular normalization and different expressions for the differential operators are related to ours by $\beta_i\mapsto 1-\beta_i$).
Let also $a_i:=\exp(\twopii \alpha_i)$ and $b_i = \exp(\twopii \beta_i)$ be (unit) complex numbers.

\subsubsection{Differential operator and monodromy}
	\label{sssec:differential_operator_and_monodromy}
Consider the differential operator
\[
	D_{\alpha,\beta}:=\prod_{i=1}^n(D+\beta_i-1) - z\prod_{i=1}^n (D+\alpha_i) \quad D:=z\partial_z
\]
In $\bP^1(\bC)\setminus \{0,1,\infty\}$ its solutions form a local system $\bV(\alpha,\beta)$ of rank $n$, which we will call the hypergeometric local system.
Let $g_0,g_1,g_{\infty}$ be the monodromy matrices of this local system, along paths as described in \autoref{fig:paths_monodromy}.
Then their conjugacy classes are determined by the following conditions on the characteristic polynomials and ranks:
\begin{align*}
	\det(t-g_{\infty})&=\prod_{i}(t-a_i)\\
	\det(t-g_{0}^{-1})&=\prod_{i}(t-b_i)\\
	\rk(g_1 - \id) & = 1 \quad \det(g_1)=\exp\left(\twopii \sum \left(\beta_i-\alpha_i\right)\right)\\
\end{align*}
with the convention that whenever there are repeated roots, there is only one Jordan block.

\begin{figure}[htbp!]
	\centering
	\includegraphics[width=0.7\linewidth]{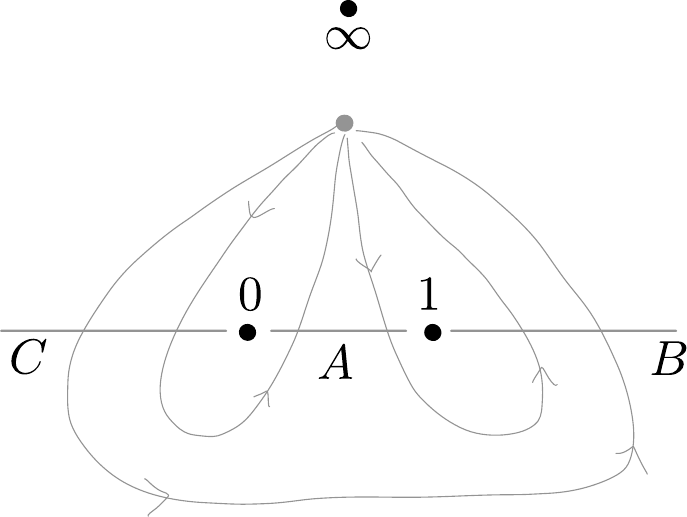}
	\caption{The paths along which we parallel-transport solutions}
	\label{fig:paths_monodromy}
\end{figure}

\subsubsection{Rigidity of the local system}
	\label{sssec:rigidity_of_the_local_system}
Assuming that $\alpha_i- \beta_j\notin \bZ$ for any $i,j$, the local system $\bV(\alpha,\beta)$ is irreducible.
Furthermore, any local system on $\bP^1(\bC)\setminus \{0,1,\infty\}$ which has the same conjugacy classes of monodromy matrices around the missing points is isomorphic to the hypergeometric local system.
In particular, to verify that an explicit representation of the free group on two letters yields a hypergeometric local system, it suffices to consider the corresponding conjugacy classes of the monodromy matrices around the removed points.

\subsubsection{Thin cyclotomic parameters}
	\label{sssec:thin_cyclotomic_parameters}
We will consider the family of hypergeometric groups with parameters
\begin{align*}
	\beta_{\bullet}: & \quad (1,1,1,1)\\
	\alpha_{\bullet}: & \quad
	\left(\frac{N-3}{2N},\frac{N-1}{2N},
		\frac{N+1}{2N},\frac{N+3}{2N}
		\right) \quad N\geq 3
\end{align*}
Note that we have the linear equation $-\alpha_1+3\alpha_2=1$.
When working with rotation matrices, we will make use of the parameters:
\begin{align}
	\label{eqn:mu_i_definition}
	\begin{split}
	\mu_1 & = 2\pi \alpha_1 = \frac{N-3}{N}\pi\\
	\mu_2 & = 2\pi \alpha_2 = \frac{N-1}{N}\pi
	\end{split}
\end{align}



\subsection{Reflection structure}
	\label{ssec:reflection_structure}

As stated in \autoref{sssec:rigidity_of_the_local_system}, in order to verify that a certain representation is the monodromy of a hypergeometric group, it suffices to consider the conjugacy classes of the corresponding matrices.
In this section, we will enlarge (with index $2$) our hypergeometric groups to groups generated by reflections.
This structure arises because our parameters are real, hence the hypergeometric differential equation has a complex conjugation symmetry and its solutions can be Schwarz-reflected across the real axis.

\subsubsection{Abbreviations}
	\label{sssec:abbreviations_c_i_def}
To keep formula sizes manageable, we will use the following abbreviations:
\begin{align}
	\label{eqn:ci_si_notation}
	\begin{split}
		c_1 & := \cos(\mu_1)	\\
		c_2 & := \cos(\mu_2)
	\end{split}
	\begin{split}
		s_1 & := \sin(\mu_1 )	\\
		s_2 & := \sin(\mu_2 )
	\end{split}
\end{align}
where the parameters $\mu_1,\mu_2$ are introduced in \autoref{eqn:mu_i_definition}.
Their specific numerical values will not be relevant until we reach the calculations with rotated vectors in \autoref{sec:computations_with_rotated_vectors}.

It will also be convenient to introduce the shorthands:
\begin{align}
	\label{eqn:r_i_definition}
	\begin{split}
	r_1 & := \frac {2(c_1-1)^2} {s_1(c_1-c_2)}
	\end{split}
	\begin{split}
	r_2 & := \frac {2(c_2-1)^2} {s_2(c_2-c_1)}
	\end{split}
\end{align}

\subsubsection{The reflection matrices}
	\label{sssec:the_reflection_matrices}
With these preparations, define:
\[
A =
\begin{bmatrix}
	-1   & 0 & 0    & 0 \\
	-r_1 & 1 & -r_1 & 0 \\
	0    & 0 & -1   & 0 \\
	-r_2 & 0 & -r_2 & 1
\end{bmatrix}
B =
\begin{bmatrix}
-1 & 0 & 0 & 0 \\
0 & 1 & 0 & 0 \\
0 & 0 & -1 & 0 \\
0 & 0 & 0 & 1
\end{bmatrix}
C =
\begin{bmatrix}
-c_{1} & s_{1} & 0 & 0 \\
 s_{1} & c_{1} & 0 & 0 \\
0 & 0 & -c_{2} & s_{2} \\
0 & 0 &  s_{2} & c_{2}
\end{bmatrix}
\]

Define also a symplectic pairing on $\bR^4$ by the following matrix:
\begin{align}
	\label{eqn:symplectic_pairing_J_def}
	J =
	\begin{bmatrix}
		0    & r_2 & 0    & 0 \\
		-r_2 & 0   & 0    & 0 \\
		0    & 0   & 0    & r_1 \\
		0    & 0   & -r_1 & 0
	\end{bmatrix}
\end{align}

\subsubsection{Properties of the reflection matrices}
	\label{sssec:properties_of_the_reflection_matrices}
For $M\in \{A,B,C\}$, we have that
\[
	 M^t J M = -J
\]
or in other words, the above matrices satisfy $\ip{Mv,Mw}=-\ip{v,w}$ for vectors $v,w\in \bR^4$ and $\ip{v,w}:=v^t J w$ the symplectic pairing.
It also follows from the formulas that
\[
	A^2 = B^2 = C^2 = 1.
\]
Let us verify that if we define the monodromy matrices of a local system using the above reflections, as described in \autoref{fig:paths_monodromy}, we obtain a hypergeometric group with parameters specified in \autoref{sssec:thin_cyclotomic_parameters}.

It is immediate that the matrix $BC$ is block-diagonal consisting of rotation matrices by angles $\mu_1,\mu_2$, so the conjugacy class at infinity is correct.
It is immediate also that the matrix $BA$ of monodromy around $1$ is such that $BA-\id = B(A-B)$ is of rank $1$.
The only necessary calculation is that $AC$ is a maximally unipotent matrix.

One could check it by a tedious and explicit calculation from the above formulas.
A shortcut in computations it to use the vectors generating the cone $\cC_0$ defined by \autoref{eqn:fundamental_cone_definition}, see also \autoref{eqn:cone_def_abstract} for which vectors are $v_i$.
Then two readily verified properties yield the result.
First, one checks that each column vector is an eigenvector of $A$, with eigenvalue $(-1)^{i+1}$ for $v_i$.
Next, one verifies that $C$ satisfies $Cv_i = (-1)^{i+1} v_i + \sum_{j>i} c_{i}^j v_j$, i.e. $C$ respects the filtration induced by the vectors $v_i$.
It then follows that $CA$ is a maximally unipotent matrix preserving the filtration induced by the cone vectors.




\section{Cones and ping-pong}
	\label{sec:cones_and_ping_pong}

\paragraph{Outline of section}
We describe in \autoref{ssec:triangle_reflection_groups} the hyperbolic triangle reflection groups that give the fundamental group of the orbifold which is the basis for our analysis.
Next, in \autoref{ssec:the_cones} we describe the abstract properties of the cones that are used for the ping-pong argument.
Based on these abstract properties we explain in \autoref{ssec:proof_of_the_ping_pong_property} how to reduce the proof of the ping-pong property to certain explicit calculations.
Finally in \autoref{ssec:explicit_cones_and_properties} we give the explicit formula for the cone and verify that it has the properties that we used.
This reduces the calculations to an explicit analysis in \autoref{sec:computations_with_rotated_vectors}.


\subsection{Triangle reflection groups}
	\label{ssec:triangle_reflection_groups}

\subsubsection{Setup}
	\label{sssec:setup_triangle_reflection_groups}
Fix an integer $N\geq 4$.
We will be interested in the group given by the generators and relations:
\[
	G_{N}:=\ip{a,b,c\, |\, a^2=b^2=c^2=1, (bc)^{N}=1}
\]
It is transparent that it acts on the hyperbolic plane such that $a,b,c$ are reflections in geodesics, with the geodesics for $b$ and $c$ forming an angle of $\tfrac{\pi}{2N}$ and the geodesic for $a$ going between the (nearest) endpoints of the geodesics for $b$ and $c$.
An illustration in the disc model is provided in \autoref{fig:orbifold_fundamental_domain}.
\begin{figure}[htbp!]
	\centering
	\includegraphics[width=0.7\linewidth]{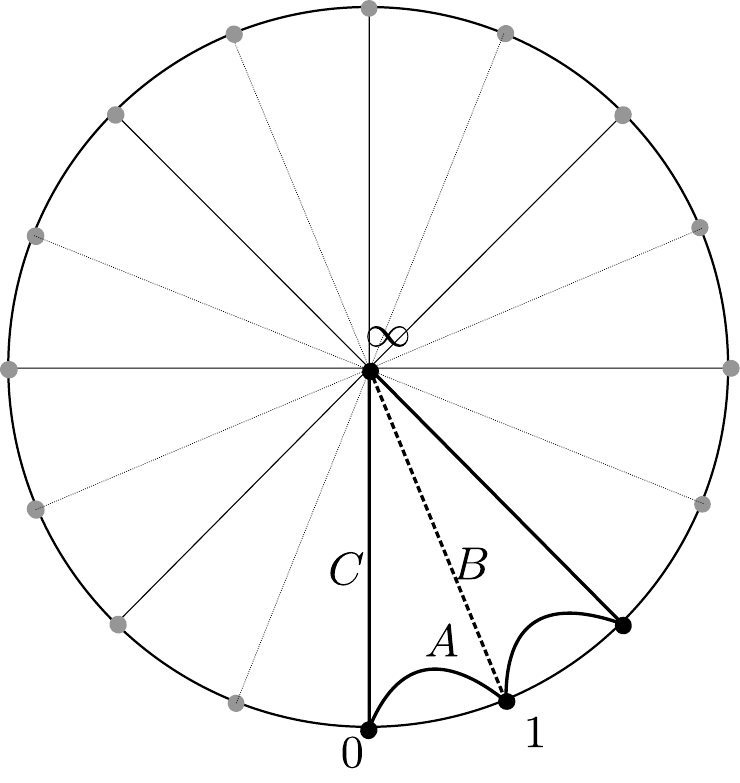}
	\caption{Fundamental domain for the triangle group action.}
	\label{fig:orbifold_fundamental_domain}
\end{figure}

\subsubsection{Linear and projective representation}
	\label{sssec:linear_and_projective_representation}
Recall that our basic angles from \autoref{eqn:mu_i_definition} are:
\[
	(\mu_1,\mu_2) = \left(2\pi\frac{N-3}{2N},2\pi\frac{N-1}{2N}\right).
\]
and that we defined the matrices $A,B,C$ in \autoref{sssec:the_reflection_matrices}.
Define the group generated by them:
\[
	\tilde\Gamma_N := \ip{A,B,C}\subset \GSp_{4}(\bR)
\]
They are in the general symplectic group, i.e. $\ip{gv,gw}=\chi(g)\ip{v,w}$ for a character $\chi \colon \GSp_4\to \bG_m$ and where $\ip{-,-}$ is the symplectic pairing.

Let $R:=BC$ be the block rotation matrix by the corresponding angles $\mu_i$.
We then have the following basic dichotomy:
\begin{description}
	\item[$N$ is odd]  Then the order of $R$ is $N$ and $-\id$ is not in $\Gamma_N$.
	\item[$N$ is even]  Then the order of $R$ is $2N$ and $-\id$ is in $\Gamma_N$.
	Specifically $R^N=-\id$.
\end{description}
Indeed, we have that $gcd(N-1,2N)=gcd(N-1,2)$ which is $2$ or $1$, according to whether $N$ is odd or even.

In both cases we will consider only the projective action of $\Gamma_N$, so let:
\[
	P\tilde\Gamma_N := \text{image of }\tilde\Gamma_N \subset \PGSp_4(\bR).
\]
It follows that independently of the parity of $N$ we have the representation
\begin{align}
	\label{eqn:rho_N_def}
	\begin{split}
	G_N & \xrightarrow{\phantom{1}\rho_N\phantom{1}} P\tilde\Gamma_N \acts \bP(\bR^4)\\
	\{a,b,c\} & \mapsto \{A,B,C\}
	\end{split}
\end{align}
which will be our basic object of study.



\subsection{The cones}
	\label{ssec:the_cones}

The action of the reflections $b,c$ on the boundary of hyperbolic space divides it into $2N$ circular arcs.
We will associate a projective cone in $\bP(\bR^4)$ to one of the arcs, and propagate it to the remaining $2N-1$ arcs using the action of the matrices $B,C$.

\subsubsection{Cones and projective cones}
	\label{sssec:cones_and_projective_cones}
To specify a cone in $\bR^4$ one can either give the vectors spanning it, or specify the equations of its faces.
This, in particular, gives a duality between cones in a vector space and cones in its dual.
Our cones will turn out to be self-dual when the vector space is identified with its dual via the symplectic pairing.
Additional, our cones will be simplicial, i.e. have four faces and four extreme rays.

We will describe a (simplicial) cone $\cC$ by specifying its four spanning vectors, and write $\cC=[v_0| v_1| v_2| v_3]$ where $v_i$ are column vectors in $\bR^4$.
So elements of $\cC$ are of the form $\sum a_i v_i$ where $a_i\geq 0$.
For the ping-pong argument we will consider the image of the cones in $\bP(\bR^4)$, but for calculations we will distinguish between a cone $\cC$ and its negative $-\cC$, spanned by $-v_i$.
Given a cone $\cC$ in $\bR^4$ we denote by $\bP\cC$ its image in $\bP(\bR^4)$.

\subsubsection{The ping-pong cones}
	\label{sssec:the_ping_pong_cones}
Let us postpone the explicit definition of the cone vectors until \autoref{ssec:explicit_cones_and_properties} but use the following notation to describe some important properties.
We start with a cone of the form
\begin{align}
	\label{eqn:cone_def_abstract}
	\cC_0 := [v_0 \,|\, v_1 \,|\, v_2\,|\, v_3]
\end{align}
where $v_i\in \bR^4$ are vectors with the following properties:
\begin{enumerate}
	\item The vectors $v_0, v_2$ are fixed by $B$.
	\item The vector $v_3$ satisfies $Cv_3=-v_3$.
\end{enumerate}
In particular $Bv_3 = B(-Cv_3) = (-R)v_3$.

We can then define the adjacent cone by reflection in $B$:
\begin{align}
	\label{eqn:cone_prime_def_abstract}
	\cC'_0 := B \cC_0 = [v_0 \,|\, Bv_1 \,|\, v_2 \,|\,  (-R)v_3]
\end{align}
All the other cones are obtained by applying the rotation matrix to these basic cones, specifically:
\[
	\cC_k := R^k \cC_0 \quad \cC_k':=R^k\cC'_0 \quad k=1,\ldots, N-1.
\]
Our main result can then be stated as follows:
\begin{theorem}[Ping-pong property of cones]
	\label{thm:ping_pong_property_of_cones}
	Consider the projective cones $\bP\cC_k,\bP\cC'_k$ for $k=0,\ldots, N-1$ defined above.
	\begin{enumerate}
		\item The interiors of distinct cones are disjoint.
		\item For any of the cones $\bP\cC$ except $\cC_0$, we have that $A\cdot \bP\cC\subset \bP\cC_0$ where $A$ is the matrix from \autoref{sssec:the_reflection_matrices}.
	\end{enumerate}

	Therefore
	\[
		P\tilde{\Gamma}_N \isom \ip{a,b,c \, | \, a^2=b^2=c^2=1, (bc)^N=1}
	\]
	as a group.
\end{theorem}



\subsection{Proof of the ping-pong property}
	\label{ssec:proof_of_the_ping_pong_property}

In this section we reduce the proof of \autoref{thm:ping_pong_property_of_cones} to certain positivity properties that will be verified in \autoref{sec:computations_with_rotated_vectors}.

\subsubsection{Disjointness property}
	\label{sssec:disjointness_property}
The dihedral group $\ip{B,C}$ acts freely and transitively on the set of projective cones $\{\bP\cC_i,\bP\cC_i'\}$.
Therefore, to verify disjointness of any pair, it suffices to verify disjointness of $\bP\cC_0$ from any other cone in the list.
Furthermore, because we work projectively, it suffices to show that the cones $(-R)^k \cC_0$ and $(-R)^k\cC_0'$ are disjoint from $\cC_0$ and $-\cC_0$ in $\bR^4$.

\subsubsection{Contraction property}
	\label{sssec:contraction_property}
Continuing to make use of the freedom to work projectively, for contraction it suffices to verify that $(-1)^k \cC_k$ is mapped by $A$ into $\cC_0$, and similarly for $(-1)^k\cC_k'$.
Recall that we have
\begin{align}
	\label{eqn:roated_cones_abstract}
	\begin{split}
	(-1)^k\cC_{k} & = (-R)^k\cC_0 =
	(-R)^k
	\left[
	v_0\, | \,
	v_1\, | \,
	v_2\, | \,
	v_3
	\right]	\\
	(-1)^k\cC_{k}' & = (-R)^k\cC_0' =
	(-R)^k
	\left[
	v_0\, | \,
	Bv_1\, | \,
	v_2\, | \,
	(-R)v_3
	\right]
	\end{split}
\end{align}
So for contraction it suffices to verify that the vectors
\begin{align}
	\label{eqn:vectors_to_test}
	(-R)^k v_0 \quad
	(-R)^k v_1 \quad
	(-R)^k v_2 \quad
	(-R)^k v_3 \quad
	(-R)^k Bv_1
\end{align}
are mapped by $A$ into the (closed) original cone $\cC_0$, for any $k=1,\ldots, N-1$.
Similarly, for disjointness of cones it suffices to verify that the above vectors are themselves disjoint from the (closed) cone and its opposite (recall that we must check disjointness projectively).
Let us emphasize here that we will verify this assertion for the cone $\cC_0$ and vectors viewed in $\bR^4$, not their projective versions.

Note that there is one exceptional case, namely for the vector $v_3$ we also have to consider $(-R)^N v_3$.
However $(-R)^N=-\id$ independently of the parity of $N$, and $Av_3=-v_3$, so the required positivity properties will follow straightforwardly in this endpoint case.

\subsubsection{Verifying inclusion in a cone}
	\label{sssec:verifying_inclusion_in_a_cone}
For a simplicial cone $\cC$ in $\bR^4$, we will use the same letter for the matrix of its columns.
In general, to certify that a column vector $v$ belongs to $\cC$ one must first compute a matrix $\check{\cC}$, which defines the faces of $\cC$, and check that $\check{\cC}v$ has only non-negative entries.
Conversely, if the result has at least one strictly negative entry, the vector is not in the cone.

Our cone $\cC_0$ will have an additional self-duality property under the symplectic pairing.
Specifically, we will find that $\cC_0^t \cdot J\cdot \cC_0$ is anti-diagonal, where $J$ is the matrix of the symplectic pairing.
Let now $S$ be a diagonal matrix with the same signs on the diagonal as the anti-diagonal matrix $\cC_0^t \cdot J \cdot \cC_0$.
We are thus lead to define the matrix:
\begin{align}
	\label{eqn:M_test_matrix}
	M:=S \cdot \cC_0^t\cdot  J
\end{align}
which gives the following certificate on a column vector $v\in \bR^4$.
Consider $Mv$: if all entries are non-negative then $v$ belongs to $\cC_0$, and if at least one entry is strictly negative then it is in the exterior.
Furthermore, if one entry is strictly negative and one is strictly positive, then $v$ is disjoint both from $\cC$ and $-\cC$, so disjoint projectively.

\subsubsection{Contraction implies disjointness}
	\label{sssec:contraction_implies_disjointness}
By the discussion in the preceding paragraphs, our task is reduced to showing that for certain vectors $v$ listed in \autoref{eqn:vectors_to_test}, the vector $MAv$ has all entries positive (to certify contraction) while $Mv$ has two entries of opposite sign (to certify projective disjointness).

Our matrix $M$ will have the further useful property that its rows are eigenvectors of $A$, with eigenvalues $\pm 1$.
Specifically the first and third rows are (right) eigenvectors with eigenvalue $+1$, and the second and fourth rows have eigenvalue $-1$.
So if all entries of $MAv$ are non-negative, and at least three entries are not zero, then $Mv$ has two entries of opposite sign.

Note that the property of $M$ to have rows which are (right) eigenvectors of $A$ is equivalent, by the construction of $M$, to the property that the original cone $\cC_0$ is spanned by eigenvectors of $A$, with the sign pattern of eigenvalues flipped since $A^t J A = - J$.

\subsubsection{Summary}
	\label{sssec:summary_of_proof}
To sum up, we have reduced the proof of \autoref{thm:ping_pong_property_of_cones} to showing that vectors of the form $MA(-R)^kv$ have all entries non-negative (and some strictly positive) for $k=1\ldots N-1$ and $v\in \{v_0,v_1,v_2,v_3,Bv_1\}$.
We next exhibit in \autoref{ssec:explicit_cones_and_properties} the explicit vectors and matrices described above and verify that they satisfy the useful properties we stated.
We then proceed to actually verify the required positivity properties in \autoref{sec:computations_with_rotated_vectors}.



\subsection{Explicit cones and properties}
	\label{ssec:explicit_cones_and_properties}

\subsubsection{Further abbreviations}
	\label{sssec:further_abbreviations}
Besides the abbreviations $c_i,s_i$ for cosines and sines from \autoref{sssec:abbreviations_c_i_def}, the following quantities
\begin{align}
	\label{eqn:cc_i_definition}
	cc_1 := 1 - c_1  && cc_2 := 1-c_2
\end{align}
will prove useful.
Additionally it will prove useful to introduce the quantities:
\begin{align}
	\label{eqn:L_i_def}
	\begin{split}
	L_1 & := \frac{cc_1\cdot cc_2 - 3(cc_2-cc_1)}{-cc_1\cdot cc_2 + 3(cc_1+cc_2)}\\
	L_2 & := \frac{cc_1\cdot cc_2 + 3(cc_2-cc_1)}{-cc_1\cdot cc_2 + 3(cc_1+cc_2)}\\
	\end{split}
\end{align}
These quantities satisfy a number of useful identities which will be discussed below.
The main ones, which characterize the $L_i$ in terms of $cc_i$, are:
\begin{align}
	\label{eqn:L_i_cc_i_identity}
	3(L_1+L_2) = (L_1+1)cc_2 = (L_2+1)cc_2
\end{align}
We will deduce some further properties of these quantities in \autoref{sssec:the_parameters_l_i}

\subsubsection{The cone}
	\label{sssec:the_cone}
With these preparations, here is the cone:
\begin{align}
	\label{eqn:fundamental_cone_definition}
	\cC_0:=
	\begin{bmatrix}
	0                          & -L_2\cdot cc_1      & 0                 & cc_1\\
	-L_1\cdot \frac{cc_1}{s_1} & \frac{-cc_1^2}{s_1} & -\frac{cc_1}{s_1} & -\frac{cc_1^2}{s_1}\\
	0                          & L_1\cdot cc_2       & 0                 & -cc_2\\
	L_2\cdot \frac{cc_2}{s_2}  & \frac{cc_2^2}{s_2}  & \frac{cc_2}{s_2}  & \frac{cc_2^2}{s_2}
	\end{bmatrix}
\end{align}

\subsubsection{Self-duality of the cone}
	\label{sssec:self_duality_of_the_cone}
Recall that we introduced the matrix of the symplectic pairing $J$ in \autoref{eqn:symplectic_pairing_J_def}.
Then it is a direct algebraic verification that:
\[
	\cC_0^t \cdot J \cdot \cC_0 =
	\begin{bmatrix}
		0       & 0      & 0     & -\alpha\\
		0       & 0      & \alpha     & 0\\
		0       & -\alpha     & 0     & 0\\
		\alpha       & 0      & 0     & 0 \\
	\end{bmatrix}
	\quad \alpha:=2\cdot \frac{L_1-L_2}{cc_1-cc_2}\cdot\frac{cc_1^2cc_2^2}{s_1s_2}
\]
The verification of this is immediate, the only non-trivial algebraic manipulation is to show that the fourth entry in the second row, and the second entry in the fourth row, vanish.
This ultimately follows from the identity $(L_1+1)cc_2 = (L_2+1)cc_1$ stated in \autoref{eqn:L_i_cc_i_identity}.
It is important to note that $\alpha>0$, since $L_2>L_1$ and $cc_1>cc_2$, as will be established later in \autoref{sssec:the_parameters_l_i}.

\subsubsection{Contraction matrix}
	\label{sssec:contraction_matrix}
The identity for $\cC_0^tJ\cC_0$ from the previous paragraph implies that the equations for the faces of the cone are given by taking the symplectic pairing against the spanning vectors (with appropriate signs).
So if we denote by $S'$ the diagonal matrix with entries $1,-1,1,-1$ it follows that $S'\cdot \cC_0^t \cdot J$ is the matrix that detects if a vector belongs, or not, to the cone $\cC_0$.
In fact, we could have used instead of $S'$ any diagonal matrix with entries having the same pattern of signs.
Let us therefore consider the matrix :
\[
	M:= \tfrac{1}{\alpha'}\cdot S' \cdot  \cC_0^t \cdot J
	\quad \text{ where }\quad
	\alpha':=\frac{2}{cc_2-cc_1}\cdot\frac{cc_1^2cc_2^2}{s_1s_2}>0
\]
Computing it explicitly yields:
\begin{align}
	\label{eqn:belonging_matrix}
	M =
	\begin{bmatrix}
	   -\frac{L_1}{cc_1} & 0                     & -\frac{L_2}{cc_2} & 0\\
	   1                 & -L_2 \frac{s_1}{cc_1} & 1                 & -L_1 \frac{s_2}{cc_2}\\
	   -\frac{1}{cc_1}   & 0                     & -\frac{1}{cc_2}   & 0\\
	   1                 & \frac{s_1}{cc_1}      & 1                 & \frac{s_2}{cc_2}
	\end{bmatrix}
\end{align}

\subsubsection{Why contraction implies disjointness}
	\label{sssec:why_contraction_implies_disjointness}
Recall that in order to establish \autoref{thm:ping_pong_property_of_cones}, we must verify that for certain vectors $v$ the vector $MAv$ has all entries non-negative (the contraction action of $A$) and that for the same vectors $Mv$ has two entries of opposite sign.

We next observe that the matrices $M$ and $MA$ have the property that their second and fourth rows agree, while the first and third are negative of each other.
Specifically, a direct computation yields:
\begin{align}
	\label{eqn:MA_definition}
	MA =
	\begin{bmatrix}
	   \frac{L_1}{cc_1} & 0                     & \frac{L_2}{cc_2} & 0\\
	   1                 & -L_2 \frac{s_1}{cc_1} & 1                 & -L_1 \frac{s_2}{cc_2}\\
	   \frac{1}{cc_1}   & 0                     & \frac{1}{cc_2}   & 0\\
	   1                 & \frac{s_1}{cc_1}      & 1                 & \frac{s_2}{cc_2}
	\end{bmatrix}
\end{align}
In order to verify this identity, we just multiply the expression for $M$ from \autoref{eqn:M_test_matrix} and for $A$ from \autoref{sssec:the_reflection_matrices}.
The only identity that needs to be used (when computing the first and fourth entries of the second row) is that
\[
	1 = \frac{L_1cc_2 - L_2 cc_1}{cc_1-cc_2}
\]
which follows readily from \autoref{eqn:L_i_cc_i_identity}.

From the formulas for the matrices $M$ and $MA$, it is now clear that if $MAv$ has all entries non-negative, and at least one entry in each of the pairs $\{\text{first,third}\}$ and $\{\text{second,fourth}\}$ is nonzero, then automatically $Mv$ will have two entries of opposite sign.
It follows that it suffices to consider vectors of the form $MAv$ and establish that their entries are non-negative; it will be transparent from the calculations that the needed nonvanishing will also hold.



\section{Computations with rotated vectors}
	\label{sec:computations_with_rotated_vectors}

\paragraph{Outline of section}
In \autoref{ssec:the_vectors_their_rotations_and_contraction} we introduce powers of the rotation matrix which is used to transport the vectors of interest.
Next we introduce some further notation and some background calculations in \autoref{ssec:frequently_used_expressions}.

The bulk of the calculations is performed in the remaining sections.
We tackle the vectors in increasing order of complexity, namely $v_2, v_0, v_3, v_1, Bv_1$.


\subsection{The vectors, their rotations, and contraction}
	\label{ssec:the_vectors_their_rotations_and_contraction}

Recall that the vectors we need to consider are $v_0, v_1, v_2, v_3$ and $Bv_1$, where the $v_i$ span the cone $\cC_0$.
Here are the vectors:
\begin{align}
	\label{eqn:list_of_vectors}
	\left[v_0|v_1|v_2|v_3|Bv_1\right]=
	\begin{bmatrix}
	0                          & -L_2 cc_1      & 0                 & cc_1 & -L_2 cc_1\\
	-L_1 \frac{cc_1}{s_1} & \frac{-cc_1^2}{s_1} & -\frac{cc_1}{s_1} & -\frac{cc_1^2}{s_1} & \frac{cc_1^2}{s_1}
	\\
	0                          & L_1 cc_2       & 0                 & -cc_2 & L_1 cc_2\\
	L_2 \frac{cc_2}{s_2}  & \frac{cc_2^2}{s_2}  & \frac{cc_2}{s_2}  & \frac{cc_2^2}{s_2} & \frac{-cc_2^2}{s_2}
	\end{bmatrix}
\end{align}

\subsubsection{Passing to $\theta_N$-parameter}
	\label{sssec:passing_to_theta_n_parameter}
When facing trigonometric expressions, it will be convenient to express everything in terms of the basic angle
\begin{align}
	\label{eqn:thetaN_def}
	\theta_N = \frac{\pi}{N}\text{ so that }
	\mu_1 = \pi - 3\theta_N \text{ and }
	\mu_2 = \pi - \theta_N .
\end{align}
It is then immediate that:
\begin{align}
	\label{eqn:theta_N_mu_i}
	\begin{split}
	\cos(\mu_1)&=\cos(\pi-3\theta_N)=-\cos(3\theta_N)\\
	\sin(\mu_1)&=\sin(\pi-3\theta_N)=\sin(3\theta_N)\\
	\cos(\mu_2)&=\cos(\pi-\theta_N)=-\cos(\theta_N)\\
	\sin(\mu_2)&=\sin(\pi-\theta_N)=\sin(\theta_N)
	\end{split}
\end{align}
Recall next that the matrix $R=BC$ is block-diagonal, rotating in the first block by $\mu_1$ and in the second by $\mu_2$.
Therefore we have:
\[
	(-R) =
	\begin{bmatrix}
		\cos(3\theta_N)  & \sin(3\theta_N) & 0               & 0\\
		-\sin(3\theta_N) & \cos(3\theta_N) & 0               & 0\\
		0                & 0               & \cos(\theta_N)  & \sin(\theta_N)\\
		0                & 0               & -\sin(\theta_N) & \cos(\theta_N)\\
	\end{bmatrix}
\]
so $(-R)$ is rotation by $-3\theta_N$ in the first block and by $-\theta_N$ in the second block.
To abbreviate further the sines of multiple angles, we will use the notation:
\begin{align}
	\label{eqn:sine_cosine_powers_sp_i_cp_i}
	\begin{split}
		cp_1 & := \cos(3k\theta_N)	\\
		cp_2 & := \cos(k\theta_N)
	\end{split}
	\begin{split}
		sp_1 & := -\sin(3k\theta_N )	\\
		sp_2 & := -\sin(k\theta_N )
	\end{split}
\end{align}
where $cp$ and $sp$ are meant to denote ``cosine power'' and ``sine power''.
These are precisely the entries of $(-R)^k$:
\begin{align}
	\label{eqn:minusRk}
	(-R)^k =
	\begin{bmatrix}
		cp_1  & -sp_1 & 0               & 0\\
		sp_1 & cp_1 & 0               & 0\\
		0                & 0               & cp_2  & -sp_2\\
		0                & 0               & sp_2 & cp_2\\
	\end{bmatrix}
\end{align}
We have chosen to express $(-R)^k$ with the sign choices standard for a counterclockwise rotation matrix, but the reader should keep in mind that the signs of sines are as stated in \autoref{eqn:sine_cosine_powers_sp_i_cp_i}.

Let us finally recall the matrix $MA$ from \autoref{eqn:MA_definition}:

\begin{align}
	\label{eqn:MA_definition_repeated}
	MA =
	\begin{bmatrix}
	   \frac{L_1}{cc_1} & 0                     & \frac{L_2}{cc_2} & 0\\
	   1                 & -L_2 \frac{s_1}{cc_1} & 1                 & -L_1 \frac{s_2}{cc_2}\\
	   \frac{1}{cc_1}   & 0                     & \frac{1}{cc_2}   & 0\\
	   1                 & \frac{s_1}{cc_1}      & 1                 & \frac{s_2}{cc_2}
	\end{bmatrix}
\end{align}
Our task has been reduced to computing $MA (-R)^k v$, for $k=1\ldots N-1$, and for each column vector $v$ in \autoref{eqn:list_of_vectors}.

This will take up the rest of this section, after some preliminary recollections from trigonometry in the next section.



\subsection{Frequently used expressions}
	\label{ssec:frequently_used_expressions}

\subsubsection{Some trigonometric formulas}
	\label{sssec:some_trigonometric_formulas}
Since $\mu_1\equiv 3\mu_2 \mod 2\pi$ it will be useful to make use of angle-tripling formulas:
\begin{align}
\begin{split}
	\label{eqn:angle_tripling}
	\sin(3\theta) & = \sin(\theta)\cdot\left(3 - 4\sin^2 (\theta)\right)
	= \sin(\theta)\cdot (2\cos(2\theta)+1)\\
	\cos(3\theta) & = \cos(\theta)\cdot\left(4\cos^2(\theta) - 3\right)
	 = \cos(\theta) \cdot \left( 2 \cos(2\theta) - 1\right)
\end{split}
\end{align}
Besides considering $\cos(3\theta)/\cos(\theta)$ and similarly for sine, we will frequently also use the following difference of cosines:
\begin{align}
	\begin{split}
	\label{eqn:difference_triple_cosine}
	\cos(3\theta)-\cos(\theta) & = \cos(\theta)\cdot 2 \cdot(\cos(2\theta)-1)\\
	& = -4\cdot \cos(\theta)\cdot \sin^2(\theta)\\
	& = -\sin(\theta)\cdot 2\cdot \sin(2\theta)\\
	\end{split}
	\intertext{and its analogue for sines:}
	\label{eqn:difference_triple_sine}
	\begin{split}
		\hfill \sin(3\theta)-\sin(\theta) = \sin(\theta)\cdot 2 \cdot \cos(2\theta).
	\end{split}
\end{align}

We'll also make use of the standard addition/subtraction formulas:
\begin{align}
	\label{eqn:sine_cosine_addition_subtraction}
	\begin{split}
	\cos(\alpha+\beta) & = \cos(\alpha)\cos(\beta) - \sin(\alpha)\sin(\beta)\\
	\cos(\alpha-\beta) & = \cos(\alpha)\cos(\beta) + \sin(\alpha)\sin(\beta)\\
	\sin(\alpha+\beta) & = \sin(\alpha)\cos(\beta) + \cos(\alpha)\sin(\beta)\\
	\sin(\alpha-\beta) & = \sin(\alpha)\cos(\beta) - \cos(\alpha)\sin(\beta)\\
	\end{split}
\end{align}

\subsubsection{Abbreviations}
	\label{sssec:abbreviations}
Recall that we introduced the algebraic expressions
\begin{align}
	\label{eqn:cci_notation}
	\begin{split}
	cc_1 & := 1 - c_1 \quad \text{ so that }\quad \frac{cc_1}{s_1} = \frac{\sin(\mu_1/2)}{\cos(\mu_1/2)}
	= \frac{\cos(3\theta_N/2)}{\sin(3\theta_N/2)}
	\\
	cc_2 & := 1 - c_2 \quad \text{ so that }\quad \frac{cc_2}{s_2} = \frac{\sin(\mu_2/2)}{\cos(\mu_2/2)}
	= \frac{\cos(\theta_N/2)}{\sin(\theta_N/2)}
	\end{split}
\end{align}
where we used the formulas $1-\cos(\theta)=2\sin^2(\theta/2)$ and $\sin(\theta)=2\sin(\theta/2)\cos(\theta/2)$.

\subsubsection{The parameters $L_i$}
	\label{sssec:the_parameters_l_i}
To shorten notation, we introduced in \autoref{eqn:L_i_def} the constants $L_1,L_2$ which we will mostly use through the identities:
\begin{align}
	\label{eqn:L_i_identities}
	3(L_1+L_2) = (L_1+1)cc_2 = (L_2+1)cc_1 = \tfrac{6\cdot cc_1\cdot cc_2}{-cc_1\cdot cc_2+ 3(cc_1+cc_2)}
\end{align}
which are verified directly from the definitions.

Let us record the basic inequalities which we use frequently :
\[
	cc_2 > cc_1 \quad \text{ and }\quad L_2 > L_1
\]
The first one is verified from the definitions of $\mu_i$ while the second one follows from the first, which we will also frequently write as $cc_2-cc_1>0$.\\
The relations between $L_1$ and $L_2$ imply
\begin{align*}
	(L_1+1) cc_2 = 3 (L_1+L_2) > 6 L_1 \iff cc_2 > (6-cc_2) L_1 \\
	\iff L_1 < \frac {cc_2} {6-cc_2}.
\end{align*}
As $cc_2 < 2$ we obtain $L_1 < \frac{1}{2}$.
This also implies the following
\[
	3(L_1+L_2) = (L_1+1) cc_2 < \frac {3} {2} \cdot 2 = 3.
\]
Hence
\begin{align}
	\label{sum_L}
	L_1+L_2 < 1.
\end{align}


\subsubsection{The difference of cosines}
	\label{sssec:the_difference_of_cosines}
Using the triple-angle formula we find:
\begin{align*}
	c_1 - c_2 &= \cos(3\mu_2)-\cos(\mu_2) = 4\cos(\mu_2)\big[\cos(\mu_2)^2-1\big] \\
	& = - 4 c_2 s_2^2\\
	& = 4(c_2-1)c_2(c_2+1)\\
	& > 0
\end{align*}
The sign $c_1-c_2>0$ follows from the above algebraic expressions and $c_2<0$, or by looking at the explicit values of $\mu_i$.

\subsubsection{The smallest and largest sines and cosines}
	\label{sssec:the_smallest_and_largest_sines_and_cosines}
Our rotation will range over $k=1,\ldots, N-1$.
We have the elementary inequalities
\[
	0< \sin(\mu_2) \leq \abs{\sin \left( k\frac{N-1}{N}\pi \right)} \leq |\cos(\mu_2)|<1
\]

\subsubsection{Frequently occurring differences of ratios}
The following manipulation is used frequently, so we record it once here, using the properties of the ratios $cc_i/s_i$ from \autoref{eqn:cci_notation} as well as the addition formula for sines \autoref{eqn:sine_cosine_addition_subtraction}:
\begin{align}
	\label{sssec:a_frequently_occurring_difference_of_ratios}
	\begin{split}
	\frac{cc_2}{s_2} - \frac{cc_1}{s_1} & =
	\frac{\cos(\theta_N/2)}{\sin(\theta_N/2)} - \frac{\cos(3\theta_N/2)}{\sin(3\theta_N/2)}\\
	& =
	\frac{\sin(\theta_N)}{\sin(\theta_N/2)\sin(3\theta_N/2)}\\
	& = \frac{2\cancel{\sin(\theta_N/2)}\cos(\theta_N/2)}{\cancel{\sin(\theta_N/2)}\sin(3\theta_N/2)}\\
	& = 2\frac{\cos(\theta_N/2)}{\sin(3\theta_N/2)}
	\end{split}
\end{align}

We now give a list of useful inequalities appearing with powers of the rotation.
\begin{align}
	\label{eqn:ratio_sine_powers_difference}
	\begin{split}
		\frac{sp_1}{s_1}
		- \frac{sp_2}{s_2} & =
		\frac{\sin(k\theta_N)}{\sin(3\theta_N)} \cdot \left(\frac {-\sin(3k\theta_N)} {\sin(k\theta_N)} + \frac {\sin(3\theta_N)} {\sin(\theta_N)} \right)\\
		& = 2 \cdot \frac{\sin(k\theta_N)}{\sin(3\theta_N)} \cdot \left( \cos(2\theta_N) - \cos(2k\theta_N)\right)\\
	\end{split}
\end{align}
In any event, this expression is always negative (non-positive) for $k=1,\ldots,N-1$.
It vanishes precisely for $k=1,N-1$.

Similarly
\begin{align}
	\label{eqn:ratio_sine_powers_sum}
	\begin{split}
		\frac{sp_1}{s_1}
		+ \frac{sp_2}{s_2} & =
		- \frac{\sin(k\theta_N)}{\sin(3\theta_N)} \cdot \left(\frac {\sin(3k\theta_N)} {\sin(k\theta_N)} + \frac {\sin(3\theta_N)} {\sin(\theta_N)} \right)\\
		& = -2 \cdot \frac{\sin(k\theta_N)}{\sin(3\theta_N)} \cdot \left( \cos(2k\theta_N) + \cos(2\theta_N) + 1\right)\\
	\end{split}
\end{align}
which is always non-positive.

Next we have the combinations
\[
	\frac{cc_1}{s_1}sp_1 \pm \frac{cc_2}{s_2}sp_2.
\]
Recall that $cc_i = 1 - c_i$ and $c_1=-\cos(3\theta_N),c_2=-\cos(\theta_N)$ and we'll use the basic identity
\[
	1+\cos(\theta) = 2 - 2\sin^2(\theta/2) = 2\cos^2(\theta/2)
\]
to reduce to the consideration of
\begin{multline}
	\label{eqn:ratio_sine_power_half_angle_difference}
	\frac{cc_1}{s_1}sp_1 - \frac{cc_2}{s_2}sp_2 =
	-\frac{\cos(3\theta_N/2)}{\sin(3\theta_N/2)} \cdot \sin(3k\theta_N)
	+ \frac{\cos(\theta_N/2)}{\sin(\theta_N/2)} \cdot \sin(k\theta_N) \\
	= \sin(k\theta_N)\left[
	-\frac{\cos(3\theta_N/2)}{\sin(3\theta_N/2)} \cdot \left(2\cos(2k\theta_N)+1\right) + \frac{\cos(\theta_N/2)}{\sin(\theta_N/2)}
	\right]	\\
	= \sin(k\theta_N)\left[
	-2\cos(2k\theta_N)\frac{\cos(3\theta_N/2)}{\sin(3\theta_N/2)} + \frac{\sin(\theta_N)}{\sin(3\theta_N/2)\sin(\theta_N/2)}
	\right]	\\
	= \frac{2\sin(k\theta_N)}{\sin(3\theta_N/2)}\left[
	\cos\left(\frac{\theta_N}{2}\right)-
	\cos(2k\theta_N)\cos\left(\frac{3\theta_N}{2}\right)
	\right]
\end{multline}
This expression is manifestly non-negative, since the value $\cos(\theta_N/2)$ is larger that $\cos(3\theta_N/2)$ for all $N \ge 4$.

Let's do the same but with a sum:
\begin{multline}
	\label{eqn:ratio_sine_power_half_angle_sum}
	\frac{cc_1}{s_1}sp_1 + \frac{cc_2}{s_2}sp_2 =
	-\frac{\cos(3\theta_N/2)}{\sin(3\theta_N/2)} \cdot \sin(3k\theta_N)
	-\frac{\cos(\theta_N/2)}{\sin(\theta_N/2)} \cdot \sin(k\theta_N) \\
	= \sin(k\theta_N)\left[
	-2\cos(2k\theta_N)\frac{\cos(3\theta_N/2)}{\sin(3\theta_N/2)} - \frac{\sin(2\theta_N)}{\sin(3\theta_N/2)\sin(\theta_N/2)}
	\right]	\\
	= -\frac{2\sin(k\theta_N)}{\sin(3\theta_N/2)}\big[
	\cos(2k\theta_N)\cos(3\theta_N/2) + 2\cos(\theta_N)\cos(\theta_N/2)
	\big]
\end{multline}
This is non-positive, in fact the term $\cos(\theta_N)\cos(\theta_N/2)$ is larger in absolute value than the other one $\cos(2k\theta_N)(\cos(3\theta_N/2))$.\\

We give in the following two propositions simplified expression for some combinations of terms that appear several times in our computations.
\begin{proposition}
	\label{cc_s}
	The following formulas hold:
	\begin{align*}
		\pm  cp_1 - \frac {cc_1} {s_1} sp_1 &= \frac {\sin(3(k\mp\frac {1} {2}) \theta_N )} {\sin(3 \theta_N / 2)}\\
		\pm  cp_2 - \frac {cc_2} {s_2} sp_2 &= \frac {\sin((k\mp\frac {1} {2}) \theta_N )} {\sin(\theta_N / 2)}\\
		\pm (cp_1 - cp_2) - \frac {cc_1} {s_1} sp_1 + \frac {cc_2} {s_2} sp_2 &=  - 4 \cdot \frac {\sin((k\pm\frac {1} {2}) \theta_N)} {\sin(3\theta_N/2)} \cdot \sin\left(k \theta_N\right) \cdot \sin\left((k \pm 1) \theta_N\right)
	\end{align*}
\end{proposition}

\begin{proof}
On one hand we have
\begin{align*}
	\pm cp_1 - \frac {cc_1} {s_1} sp_1 &= \mp \cos(3k \theta_N) + \frac {\cos(3 \theta_N / 2)} {\sin(3 \theta_N / 2)} \sin(3k \theta_N) \\
				       &= \frac {\sin(3(k\mp\frac {1} {2}) \theta_N )} {\sin(3 \theta_N / 2)}
\end{align*}
and similarly for the second equality.

We now consider:
\begin{align*}
	\frac {\sin(3(k\mp\frac {1} {2}) \theta_N )} {\sin(3\theta_N / 2)} - \frac {\sin((k\mp\frac {1} {2}) \theta_N )} {\sin(\theta_N /2)}
	& = \frac {\sin((k\mp\frac {1} {2}) \theta_N)} {\sin(3\theta_N/2)} \left(\frac {\sin(3(k\mp\frac {1} {2})\theta_N)} {\sin((k\mp\frac {1} {2})\theta_N)} - \frac {\sin(3\theta_N/2)} {\sin(\theta_N/2)} \right)\\
	& = 2 \cdot \frac {\sin((k\mp\frac {1} {2}) \theta_N)} {\sin(3\theta_N/2)} \left( \cos((2k \mp 1) \theta_N) - \cos(\theta_N) \right)\\
	& = - 4 \cdot \frac {\sin((k\mp\frac {1} {2}) \theta_N)} {\sin(3\theta_N/2)} \cdot \sin\left(k \theta_N\right) \cdot \sin\left((k \mp 1) \theta_N\right)
\end{align*}
\end{proof}

The second proposition is for a similar expression where we invert the fractions.
\begin{proposition}
	\label{s_cc}
	The following formulas hold:
	\begin{align*}
		cp_1 \pm  \frac {s_1} {cc_1} sp_1 &= \frac {\cos(3(k\pm\frac {1} {2}) \theta_N )} {\cos(3 \theta_N / 2)}\\
		cp_2 \pm  \frac {s_2} {cc_2} sp_2 &= \frac {\cos((k\pm\frac {1} {2}) \theta_N )} {\cos(\theta_N / 2)}\\
		cp_1 - cp_2 \pm\left(\frac{s_1}{cc_1} sp_1 - \frac{s_2}{cc_2} sp_2\right) &= - 4 \cdot \frac {\cos((k\pm\frac {1} {2}) \theta_N)} {\cos(3\theta_N/2)} \cdot \sin\left(k \theta_N\right) \cdot \sin\left((k \pm 1) \theta_N\right)
	\end{align*}
\end{proposition}

\begin{proof}
On one hand we have
\begin{align*}
	cp_1 \pm \frac {s_1} {cc_1} sp_1 &= \cos(3k \theta_N) \mp \frac {\sin(3 \theta_N / 2)} {\cos(3 \theta_N / 2)} \sin(3k \theta_N) \\
				       &= \frac {\cos(3(k\pm\frac {1} {2}) \theta_N )} {\cos(3 \theta_N / 2)}
\end{align*}
and similarly for the second equality.

We now consider
\begin{align*}
	\frac {\cos(3(k\pm\frac {1} {2}) \theta_N )} {\cos(3\theta_N / 2)} - \frac {\cos((k\pm\frac {1} {2}) \theta_N )} {\cos(\theta_N /2)}
	& = \frac {\cos((k\pm\frac {1} {2}) \theta_N)} {\cos(3\theta_N/2)} \left(\frac {\cos(3(k\pm\frac {1} {2})\theta_N)} {\cos((k\pm\frac {1} {2})\theta_N)} - \frac {\cos(3\theta_N/2)} {\cos(\theta_N/2)} \right)\\
	& = 2 \cdot \frac {\cos((k\pm\frac {1} {2}) \theta_N)} {\cos(3\theta_N/2)} \left( \cos((2k \pm 1) \theta_N) - \cos(\theta_N) \right)\\
	& = - 4 \cdot \frac {\cos((k\pm\frac {1} {2}) \theta_N)} {\cos(3\theta_N/2)} \cdot \sin\left(k \theta_N\right) \cdot \sin\left((k \pm 1) \theta_N\right)
\end{align*}
\end{proof}
We are now ready to proceed to the analysis of vectors.

\subsection{Computing with \texorpdfstring{$v_2$}{v2}}
	\label{ssec:computing_with_v_2}

The vector $v_2$ is the third column of the matrix $M$:
\[
	v_2 =
	\begin{bmatrix}
		0 \\
		-\frac{cc_1}{s_{1}} \\
		0 \\
		\frac{cc_2}{s_{2}}
	\end{bmatrix}
	\text{ so }
		(-R)^k v_2 =
	\begin{bmatrix}
		\frac{{cc_1 } }{s_{1}} sp_1\\
		- \frac{{cc_1} }{s_{1}} cp_1\\
		-\frac{{cc_2} }{s_{2}} sp_2\\
		\frac{{cc_2 } }{s_{2}} cp_2
	\end{bmatrix}
\]
Applying the matrix $MA$ yields:
\begin{align*}
	MA (-R)^k v_2 =
	\begin{bmatrix}
	L_1\cdot \frac{sp_1}{s_1} - L_2\cdot \frac{sp_2}{s_2}\\
	L_2\cdot cp_1 - L_1\cdot cp_2 + cc_1\cdot \frac{sp_1}{s_1} - cc_2\cdot \frac{sp_2}{s_2}\\
	\frac{sp_1}{s_1} - \frac{sp_2}{s_2}\\
     -cp_1 + cp_2 + cc_1\cdot \frac{sp_1}{s_1} - cc_2\cdot \frac{sp_2}{s_2}\\
	\end{bmatrix}
\end{align*}

\subsubsection{The third entry of $MA (-R)^k v_2$}
	\label{sssec:the_third_entry_of_rkv_2}
The third entry of the vector is:
\[
	\frac{sp_1}{s_1} - \frac{sp_2}{s_2}
	=
	\frac{-\sin(3k\theta_N)}{\sin(3\theta_N)}
	-
	\frac{-\sin(k\theta_N)}{\sin(\theta_N)}
\]
We rewrite this as:
\begin{align*}
	\frac{\sin(k\theta_N)}{\sin(3\theta_N)}
	\left(
	\frac{\sin(3\theta_N)}{\sin(\theta_N)}
	-
	\frac{\sin(3k\theta_N)}{\sin(k\theta_N)}
	\right)
\end{align*}
The first factor is clearly positive for our range of $k$, and we rewrite the difference using the formula for sine of the triple angle (\autoref{eqn:angle_tripling}) as:
\begin{align*}
	\left(
	\frac{\sin(3\theta_N)}{\sin(\theta_N)}
	-
	\frac{\sin(3k\theta_N)}{\sin(k\theta_N)}
	\right)
	=
	2\cos(2\theta_N) - 2\cos(2k\theta_N)\geq 0
\end{align*}
and equality holds precisely for $k=1,N-1$.

\subsubsection{The fourth entry of $MA (-R)^k v_2$}
	\label{sssec:the_fourth_entry_of_v_2}
	By \autoref{cc_s}, the 4th entry of the vector is
\[
	-cp_1 + cp_2 + \frac{cc_1 }{s_1}sp_1 - \frac{cc_2}{s_2}sp_2 =
		4 \cdot \frac {\sin((k+\frac {1} {2}) \theta_N)} {\sin(3\theta_N/2)} \cdot \sin\left(k \theta_N\right) \cdot \sin\left((k + 1) \theta_N\right)
\]
	which is non-negative for $1 \le k \le N-1$.

\subsubsection{The second entry of $MA(-R)^kv_2$}
	\label{sssec:the_second_entry_of_r_kv_2}
We have to show that
\[
	L_2 cp_1 - L_1 cp_2 + cc_1 \frac{sp_1}{s_1} - cc_2 \frac{sp_2}{s_2} \geq 0
\]
Observing the similarity to the fourth entry, we divide the claim into two parts.
For $k\geq \frac{N}{2}$ we claim that the above quantity is clearly greater or equal to the fourth entry, hence non-negative.
For $\frac{N}{2}\geq k$, we will replace $L_2cp_1-L_1cp_2$ by $cp_1-cp_2$, as detailed in \autoref{sssec:the_case_n_2_geq_k_geq_1} below, and proceed as with the fourth entry.

\subsubsection{The case $N> k \geq N/2$}
	\label{sssec:the_case_ngeq_k_geq_n_2_second_entry_v2}
To check that the second entry is greater than the fourth entry reduces to the inequality
\[
	(L_2+1)cp_1 - (L_1+1)cp_2 \geq 0
\]
We now rewrite this, taking into account the identities for $L_i$ from \autoref{eqn:L_i_identities}:
\[
	(L_2+1)cp_1 - (L_1+1)cp_2 =
	(L_2+1)cc_1 \frac{cp_1}{cc_1} - (L_1+1)cc_2\frac{cp_2}{cc_2} =
	3\left(L_1+L_2\right)
	\left(
	\frac{cp_1}{cc_1} - \frac{cp_2}{cc_2}
	\right)
\]
The factor $3(L_1+L_2)$ is positive, so we can drop it.
We can also factor out $\frac{-cp_2}{cc_1}$, taking into account that $cp_2=\cos(k\theta_N)\leq 0$ when $N\geq k\geq \tfrac{N}{2}$, to reduce to showing:
\begin{align}
	\label{eqn:cc_i_ratio_difference_c_i}
	\frac{cc_1}{cc_2}-\frac{cp_1}{cp_2} \geq 0
\end{align}
Now we rewrite everything in terms of $\theta_N$ and use the angle-tripling formula for cosines to express the terms as:
\[
	\frac{cc_1}{cc_2}=\frac{1+\cos(3\theta_N)}{1+\cos(\theta_N)}
	\quad
	\frac{cp_1}{cp_2} = 2\cos(2k\theta_N)-1
\]
We next reduce the expressions:
\begin{multline*}
	\frac{1+\cos(3\theta_N)}{1+\cos(\theta_N)} \geq 2\cos(2k\theta_N)-1 \iff \\
	\iff 1+\cos(3\theta_N) \geq 2\cos(2k\theta_N)-1 + 2\cos(\theta_N)\cos(2k\theta_N) - \cos(\theta_N)\\
	\iff 1 + \frac{\cos(3\theta_N)+\cos(\theta_N)}{2} \geq
	\cos(2k\theta_N) + \cos(\theta_N)\cos(2k\theta_N)
\end{multline*}
The left-hand side above is independent of $k$ while the right-hand side is monotonically increasing and achieves its maximum when $k=N-1$ to reduce to
\[
	1 + \frac{\cos(3\theta_N)+\cos(\theta_N)}{2} \geq
	\cos(2\theta_N) + \cos(\theta_N)\cos(2\theta_N)
\]
Finally we use again the angle-tripling formula to find $\cos(3\theta_N)+\cos(\theta_N)=2\cos(\theta_N)\cos(2\theta_N)$ and reduce to
\[
	1 + \cancel{\cos(\theta_N)\cos(2\theta_N)} \geq
	\cos(2\theta_N) + \cancel{\cos(\theta_N)\cos(2\theta_N)}
\]
which clearly holds.

\subsubsection{The case $N/2 \geq k \geq 1$}
	\label{sssec:the_case_n_2_geq_k_geq_1}
In this case we claim that we have
\begin{align*}
	L_2 cp_1 - L_1 cp_2 \geq cp_1 - cp_2.
\end{align*}
Indeed, it is obvious in the case $k=N/2$ and otherwise it is equivalent to
\begin{align*}
	(1-L_1) cp_2 \geq (1-L_2) cp_1\\
	\iff \frac {1-L_1} {1-L_2} \ge \frac {cp_1} {cp_2} = 2 \cos(2k\theta_N) - 1
\end{align*}
which is true since $L_1 < L_2$ implies that
\begin{align*}
	\frac {1-L_1} {1-L_2} > 1 \ge 2 \cos(2k\theta_N) - 1.
\end{align*}

We now proceed as in the analysis of the fourth entry but this time subtracting $cp_2-cp_1$, i.e.
\begin{align*}
	L_2 cp_1 - L_1 cp_2 + cc_1 \frac{sp_1}{s_1} - cc_2 \frac{sp_2}{s_2} \geq
	cp_1 -cp_2 + cc_1 \frac{sp_1}{s_1} - cc_2 \frac{sp_2}{s_2}.
\end{align*}
This last expression can be rewritten using \autoref{cc_s}
	\[
		4 \cdot \frac {\sin((k-\frac {1} {2}) \theta_N)} {\sin(3\theta_N/2)} \cdot \sin\left(k \theta_N\right) \cdot \sin\left((k - 1) \theta_N\right).
	\]
	Which is positive for  $1 \le k \le N-1$.

\subsubsection{The first entry of $MA (-R)^k  v_2$}
	\label{sssec:the_first_entry_of_mtestr_kv_2}
We have to show that the expression
\[
L_1\cdot \frac{sp_1}{s_1} - L_2\cdot \frac{sp_2}{s_2}
\]
is non-negative.
We factor out the term $-sp_2/s_1$ and it suffices to show that the resulting expression
\[
	L_2 \frac{s_1}{s_2} - L_1\cdot \frac{sp_1}{sp_2}
\]
is positive since $s_1>0$ and $-sp_2=\sin(k\theta_N)>0$.
Both ratios $sp_1/sp_2$ and $s_1/s_2$ are ratios of a triple sine over a sine, so we rewrite them using \autoref{eqn:angle_tripling} to find :
\[
	L_2(3 - 4s_2^2) - L_1(3 - 4sp_2^2)
\]
We rewrite this, using a ``polarization'' identity:
\begin{align}
	\label{sssec:axby_identity}
	Ax - By = \frac{A-B}{2}(x+y) + \frac{A+B}{2}(x-y).
\end{align}
After multiplying by $2$, we get
\[
	(L_2 - L_1)\left[6 - 4(s_2^2+sp_2^2)\right] + (L_2+L_1)\left[
	 sp_2^2 - s_2^2
	\right].
\]
Now $L_2-L_1>0$ and $s_2^2=\sin^2(\theta_N)\leq \frac{1}{2}$ since $N\geq 3$, so the first term is clearly positive.
For the second one we observe that $sp_2^2\geq s_2^2$ since this is saying that $\sin(k\theta_N)\geq \sin(\theta_N)$ for $k=1,\ldots N-1$.



\subsection{Computing with \texorpdfstring{$v_0$}{v0}}
	\label{ssec:computing_with_v_0}

\subsubsection{Setup}
	\label{sssec:setup_computing_with_v_0}
Recall that
\[
	v_0 =
	\begin{bmatrix}
	         0    \\
	-L_1 \frac{cc_1}{s_1} \\
	         0   \\
	 L_2 \frac{cc_2}{s_2} \\
	\end{bmatrix}
	\text{ so }
	(-R)^k v_0 =
	\begin{bmatrix}
	 L_1 \frac{cc_1}{s_1} sp_1\\
	-L_1 \frac{cc_1}{s_1} cp_1\\
	-L_2 \frac{cc_2}{s_2} sp_2\\
	 L_2 \frac{cc_2}{s_2} cp_2
	\end{bmatrix}.
\]
So the vector that we have to analyze, namely $MA(-R)^k v_0$ is:
\begin{align}
	\label{eqn:v_0_positivity}
	\begin{bmatrix}
	L_1^2\cdot \frac{sp_1}{s_1} - L_2^2\cdot \frac{sp_2}{s_2}\\
	L_1 L_2\cdot cp_1 - L_1 L_2\cdot cp_2 +
	L_1\cdot cc_1\cdot \frac{sp_1}{s_1} - L_2\cdot cc_2\cdot \frac{sp_2}{s_2}\\
    L_1\cdot  \frac{sp_1}{s_1} - L_2\cdot  \frac{sp_2}{s_2}\\
    -L_1\cdot cp_1 + L_2\cdot cp_2 +
    L_1\cdot cc_1\cdot \frac{sp_1}{s_1} - L_2\cdot cc_2\cdot \frac{sp_2}{s_2}
	\end{bmatrix}
\end{align}

\subsubsection{The third entry of $MA (-R)^k v_0$}
	\label{sssec:the_third_entry_of_rkv_0}
Notice that the third entry is the same as the first entry of $v_2$, which we already checked is positive in \autoref{sssec:the_first_entry_of_mtestr_kv_2}.

\subsubsection{The first entry of $MA (-R)^k  v_0$}
\label{sssec:the_first_entry_of_mtestr_kv_0}
The first entry is similar to the one of $v_2$ dealt with in \autoref{sssec:the_first_entry_of_mtestr_kv_2}, we use the same method to show positivity.

First factor out the positive term $-sp_2 / s_1$ and use the angle tripling formula \autoref{eqn:angle_tripling} to get
\[
	L_2^2(3 - 4s_2^2) - L_1^2(3 - 4sp_2^2)
\]
Using the polarization identity \autoref{sssec:axby_identity} after multiplying by $2$, we get
\[
	(L_2^2 - L_1^2)\left[6 - 4(s_2^2+sp_2^2)\right] + (L_2^2+L_1^2)\left[
	 sp_2^2 - s_2^2.
	\right]
\]
Now $L_2^2-L_1^2>0$ and $s_2^2=\sin^2(\theta_N)\leq \frac{1}{2}$ since $N\geq 3$, so the first term is positive.
For the second one use again that $sp_2^2\geq s_2^2$ for $k=1,\ldots N-1$.

\subsubsection{The fourth entry of $MA (-R)^k v_0$}
\label{sssec:the_fourth_entry_of_v_0}
Consider first the terms with a factor of $L_1$:
\begin{align*}
	- \left(cp_1  - cc_1 \frac {sp_1} {s_1}\right) &= - \left(\cos(3k\theta_N) + \frac {\cos(3 \theta_N / 2)} {\sin(3 \theta_N / 2)} \sin(3k\theta_N)\right)\\
						       & = - \frac {\sin\left(3(k+\frac{1}{2})\theta_N\right)} {\sin(3 \theta_N / 2)}.
\end{align*}
Similarly consider those with a factor of $L_2$:
\[
	cp_2 - cc_2 \frac {sp_2} {s_2}= \frac {\sin\left((k+\frac{1}{2})\theta_N\right)} {\sin(\theta_N / 2)}.
\]

The fourth entry has thus the following expression
\[
	L_2 \frac {\sin\left((k+\frac{1}{2})\theta_N\right)} {\sin(\theta_N / 2)}
	- L_1 \frac {\sin\left(3(k+\frac{1}{2})\theta_N\right)} {\sin(3\theta_N / 2)}
\]
We factor out the positive term $\sin\left((k+\frac{1}{2})\theta_N\right) / \sin(3 \theta_N / 2)$,
\begin{align*}
	L_2 \frac {\sin\left(3\theta_N / 2\right)} {\sin(\theta_N / 2)}
	- L_1 \frac {\sin\left(3(k+\frac{1}{2})\theta_N\right)} {\sin\left((k+\frac{1}{2})\theta_N\right)} -
	&= L_2 \left(2 \cos(\theta_N) + 1\right) - L_1 \left(2 \cos\left((2k+1) \theta_N\right) + 1\right)\\
	&= (L_2 - L_1) + 2 \left( L_2 \cos(\theta_N) - L_1 \cos\left((2k+1) \theta_N\right)\right).
\end{align*}
As $L_2 > L_1 > 0$ and $\cos(\theta_N) \ge \cos((2k+1) \theta_N)$ for all $k$ this expression is positive.

\subsubsection{The second entry of $MA(-R)^kv_0$}
	\label{sssec:the_second_entry_of_r_kv_0}
Observing the similarity to the fourth entry, we divide the claim into two parts.
For $k\geq \frac{N}{2}$ we claim that the above quantity is greater than or equal to the fourth entry, hence non-negative.
For $\frac{N}{2}\geq k$, we use relations with $L_1, L_2$.

\subsubsection{The case $N> k \geq N/2$}
	\label{sssec:the_case_ngeq_k_geq_n_2_second_entry_v0}
To check that the second entry is greater than the fourth entry reduces to the inequality
\[
	L_1(L_2+1)cp_1 - L_2(L_1+1)cp_2 \geq 0
\]
Using identities of $L_i$ from \autoref{eqn:L_i_identities}:
\begin{align*}
	L_1(L_2+1)cp_1 - L_2(L_1+1)cp_2 &= (L_2+1)cc_1 L_1 \frac{cp_1}{cc_1} - (L_1+1)cc_2 L_2\frac{cp_2}{cc_2} \\
					&= 3\left(L_1+L_2\right) \left( L_1 \frac{cp_1}{cc_1} - L_2 \frac{cp_2}{cc_2} \right).
\end{align*}
The factors $3(L_1+L_2)$ and $\frac{-cp_2}{cc_1}$ are positive when $N\geq k\geq \tfrac{N}{2}$, then we are reduced to showing:
\[
	L_2 \frac{cc_1}{cc_2} - L_1\frac{cp_1}{cp_2} \geq 0
\]
As $L_2 > L_1 > 0$ this is implied by the inequality
\[
	\frac{cc_1}{cc_2}-\frac{cp_1}{cp_2} \geq 0
\]
proved earlier, see \autoref{eqn:cc_i_ratio_difference_c_i}.

\subsubsection{The case $N/2 \geq k \geq 1$}
	\label{sssec:the_case_k_ngeq_n_2_second_entry_v0}
Notice that $L_2 > L_1$, $1-L_1 > 1-L_2$ always, and $cp_2 \ge cp_1$ in this range since $cp_2-cp_1 = 2 \sin(k\theta_N) \sin(2k \theta_N)$.
Multiplying these three inequalities, we get
$$L_2 (1-L_1) cp_2 \ge L_1 (1-L_2) cp_1$$
then
\begin{align*}
	L_1 L_2 cp_1 - L_1 L_2 cp_2 \geq L_1 cp_1 - L_2 cp_2.
\end{align*}
Using this inequality we reduce the positivity of the second entry in this range to the positivity of
\begin{align}
	\label{v_0_second_entry_case_2_term}
	L_1 (cp_1 + cc_1 \frac {sp_1} {s_1}) - L_2 (cp_2 + cc_2 \frac {sp_2} {s_2}).
\end{align}
Using \autoref{cc_s}, we can rewrite this expression as:
\begin{align*}
	L_2 \frac {\sin\left((k-\frac{1}{2})\theta_N\right)} {\sin\left(\theta_N / 2\right)} - L_1 \frac {\sin\left(3(k - \frac {1} {2})\theta_N / 2\right)} {\sin(3 \theta_N / 2)}.
\end{align*}
Factoring out the positive term $\sin((k-\frac {1} {2}) \theta_N) / \sin(3\theta_N / 2)$,
\begin{align*}
	L_2 \frac {\sin\left(3\theta_N / 2\right)} {\sin(\theta_N / 2)} - L_1 \frac {\sin\left(3(k-\frac{1}{2})\theta_N\right)} {\sin\left((k-\frac{1}{2})\theta_N\right)}
	&= L_2 \left(2 \cos(\theta_N) + 1\right) - L_1 \left(2 \cos\left((2k-1) \theta_N\right) + 1\right)\\
	&= (L_2 - L_1) + 2 \left( L_2 \cos(\theta_N) - L_1 \cos\left((2k-1) \theta_N\right)\right).
\end{align*}
As $L_2 > L_1 > 0$ and $\cos(\theta_N) \ge \cos((2k-1) \theta_N)$ for all $k$ this expression is positive.



\subsection{Computing with \texorpdfstring{$v_3$}{v3}}
	\label{ssec:computing_with_v_3}

\subsubsection{Setup}
	\label{sssec:setup_computing_with_v_3}
We have that
\[
	v_3 =
	\begin{bmatrix}
	      cc_1\\
	-\frac{cc_1^2}{s_1}\\
	     -cc_2\\
	\frac{cc_2^2}{s_2}
	\end{bmatrix}
	\text{ so }
	(-R)^k v_3 =
	\begin{bmatrix}
	 cc_1\cdot  cp_1 + \frac{cc_1^2}{s_1}\cdot sp_1\\
	-\frac{cc_1^2}{s_1}\cdot  cp_1 + cc_1\cdot sp_1\\
	-cc_2\cdot  cp_2 - \frac{cc_2^2}{s_2}\cdot sp_2\\
	 \frac{cc_2^2}{s_2}\cdot  cp_2 - cc_2\cdot sp_2
	\end{bmatrix}
\]
Using the identity $$cc_i^2 + s_i^2 = 2 cc_i$$ for $i \in \{0,1\}$ and identity (\ref{eqn:L_i_cc_i_identity}), we then proceed to compute $MA(-R)^k v_3$ to be:
\[
 \begin{bmatrix}
	L_{1} \cdot {cc_1} \cdot \frac{{sp_1}}{s_1} - L_{2} \cdot {cc_2} \cdot \frac{{sp_2}}{s_2} + L_{1} {cp_1} - L_{2} {cp_2}\\
3 \, \left(L_{1} + L_{2}\right) {\left({cp_1} - {cp_2}\right)} - \left(L_{2} + 1\right) {s_1} sp_1 + \left(L_{1} + 1\right){s_2} {sp_2} + 2 \, \left({cc_1} \cdot \frac{{sp_1}}{s_1} - {cc_2} \cdot \frac{{sp_2}}{s_2}\right)\\
{cp_1} - {cp_2} + {cc_1} \cdot \frac{{sp_1}}{s_1} - {cc_2} \cdot \frac{{sp_2}}{{s_2}}\\
2 \, \left({cc_1} \cdot \frac{{sp_1}}{s_1} - {cc_2} \cdot \frac{{sp_2}}{s_2}\right)
 \end{bmatrix}
\]

\subsubsection{The third entry of $MA (-R)^k v_3$}
	\label{sssec:the_third_entry_of_rkv_3}
The third entry
\[
	cp_1-cp_2+cc_1 \frac {sp_1} {s_1} - cc_2 \frac {sp_2} {s_2}
\]
was already proved to be positive in \autoref{sssec:the_case_n_2_geq_k_geq_1}, using \autoref{cc_s}.

\subsubsection{The first entry of $MA (-R)^k  v_3$}
	\label{sssec:the_first_entry_of_mtestr_kv_3}
	The first entry of $v_3$ is the same as \autoref{v_0_second_entry_case_2_term} proved to be positive in \autoref{sssec:the_case_k_ngeq_n_2_second_entry_v0}.

\subsubsection{The fourth entry of $MA (-R)^k v_3$}
	\label{sssec:the_fourth_entry_of_v_3}
	Factoring out 2, we recognize the last two terms of the fourth entry of $v_2$.
	As in \autoref{sssec:the_fourth_entry_of_v_2} we use the triple-angle formula for sines:
\begin{align*}
	\frac{cc_1}{s_1}sp_1 - \frac{cc_2}{s_2}sp_2
	&=
	(-sp_2)
	\left(\frac{cc_2}{s_2} - \frac{cc_1}{s_1} \frac{sp_1}{sp_2}
	\right)\\
	&=
	(-sp_2)
	\left(\frac{cc_2}{s_2} - \frac{cc_1}{s_1} - 2 \frac{cc_1}{s_1} \cos(2k \theta_N)
	\right)
	\\
	\intertext{now using \autoref{sssec:a_frequently_occurring_difference_of_ratios},}
	& = 2(-sp_2)
	\left(\frac{\cos(\theta_N/2)}{\sin(3\theta_N/2)}
	- \frac{\cos(3\theta_N/2)}{\sin(3\theta_N/2)}\cos(2k\theta_N)
	\right)\\
	& = {2\sin(k\theta_N)}\frac{\cos(\theta_N/2)}{\sin(3\theta_N/2)}
	\big(1
		- \frac {\cos(3\theta_N/2)} {\cos(\theta_N/2)} \cos(2k\theta_N)
	\big)\\
	& = {2\sin(k\theta_N)}\frac{\cos(\theta_N/2)}{\sin(3\theta_N/2)}
	\big(1
		- \left(2 \cos(\theta_N) - 1\right)  \cos(2k\theta_N)
	\big)\\
\end{align*}

For all $1 \le k < N$, ${2\sin(k\theta_N)}\frac{\cos(\theta_N/2)}{\sin(3\theta_N/2)} > 0$.
Moreover, as $N > 3$, we have $0 < 2 \cos(\theta_N) -1 < 1$ hence $1 - \left(2 \cos(\theta_N) - 1\right)  \cos(2k\theta_N) > 0$.

\subsubsection{The second entry of $MA(-R)^kv_3$}
	\label{sssec:the_second_entry_of_r_kv_3}
As for the second entry of the previous vectors we divide the claim into two parts.
For $k\geq \frac{N}{2}$ we claim that the above quantity is greater than or equal to the fourth entry, hence non-negative.
For $\frac{N}{2}\geq k$ we use other relations.

\subsubsection{The case $N> k > N/2$}
	\label{sssec:the_case_ngeq_k_geq_n_2_second_entry_v3}

We show in this case that
\[
	3(L_1+L_2)(cp_1-cp_2) - (L_2+1)s_1 sp_1 + (L_1+1)s_2 sp_2 \ge 0.
\]

Using \autoref{eqn:L_i_identities}, we can factor out $3(L_1+L_2)$ to get an equivalent inequality
\[
	cp_1-cp_2 - \frac {s_1}{cc_1} sp_1 + \frac {s_2}{cc_2} sp_2 \ge 0.
\]
By \autoref{s_cc}, this latter expression is equal to
\[
	- 4 \cdot \frac {\cos((k-\frac {1} {2}) \theta_N)} {\cos(3\theta_N/2)} \cdot \sin\left(k \theta_N\right) \cdot \sin\left((k - 1) \theta_N\right).
\]
As $k > \frac{N}{2}$, $\cos((k-\frac{1}{2})\theta_N) < 0$ and moreover $k<N$ then $\sin(k \theta_N)$ and $\sin((k-1)\theta_N)$ are non-negative, thus the inequality is satisfied.

\subsubsection{The case $1 \le k \le N/2$}
	\label{sssec:the_case_k_ngeq_n_2_second_entry_v3}

	First we use the fact $3(L_1+L_2) \le 3$ established in \autoref{sum_L}.
	Moreover for $1 \le k \le N / 2$, ${\cos((k-\frac {1} {2}) \theta_N)} > 0$ thus according to the previous formula
	\[
		cp_1-cp_2 - \frac {s_1}{cc_1} sp_1 + \frac {s_2}{cc_2} sp_2 \le 0.
	\]
	Hence the second entry in this case is not smaller than
	\[
		3 \cdot \left(cp_1-cp_2 - \frac {s_1}{cc_1} sp_1 + \frac {s_2}{cc_2} sp_2\right) +  2 \cdot \left({cc_1} \cdot \frac{{sp_1}}{s_1} - {cc_2} \cdot \frac{{sp_2}}{s_2}\right).
	\]
	Using the trigonometric computations of \autoref{s_cc} and \autoref{eqn:ratio_sine_power_half_angle_difference} we can rewrite this expression as
	\begin{align*}
		 - 12 \cdot \frac {\cos((k-\frac {1} {2}) \theta_N)} {\cos(3\theta_N/2)} \cdot \sin\left(k \theta_N\right) \cdot \sin\left((k-1) \theta_N\right)\\
		 + 4 \sin(k\theta_N) \left(\frac{\cos(\theta_N/2)}{\sin(3\theta_N/2)} - \frac {\cos(3\theta_N/2)} {\sin(3\theta_N/2)} \cos(2k\theta_N) \right).
	\end{align*}
	Factoring out the positive term $4 \cdot \frac{\sin(k\theta_N)} {\cos(3 \theta / 2) \cdot \sin(3 \theta / 2)}$ we are reduced to show positivity for
	\begin{align*}
		- 3 \sin(3 \theta_N /2) \cdot {\cos((k-\frac {1} {2}) \theta_N)} \cdot \sin\left((k-1) \theta_N\right)\\
		+ {\cos(3\theta_N/2)} \cdot \big(\cos(\theta_N/2) - {\cos(3\theta_N/2)} \cos(2k\theta_N) \big) \\
		= - \frac 32 \sin(3 \theta_N /2) \cdot \left(\sin((2k-\frac {3} {2}) \theta_N) - \sin\left(\theta_N/2\right)\right)\\
		+ {\cos(3\theta_N/2)} \cdot \big(\cos(\theta_N/2) - {\cos(3\theta_N/2)} \cos(2k\theta_N) \big)
	\end{align*}
	Hence we are reduced to showing that
	\begin{align*}
		\frac 32 \sin(3 \theta_N /2) \cdot \sin((2k-\frac {3} {2}) \theta_N) + \cos(3\theta_N/2)^2 \cos(2k\theta_N)\\
		\le \frac 32 \sin(3 \theta_N /2) \cdot \sin(\theta_N / 2) + \cos(3\theta_N/2) \cos(\theta_N / 2).
	\end{align*}
	In the following we show that the expression
	\[
		\frac 32 \sin(3 \theta_N /2) \cdot \sin((2k-\frac {3} {2}) \theta_N) + \cos(3\theta_N/2)^2 \cos(2k\theta_N)
	\]
	takes its maximal value at $k=1$.
	\begin{proposition}
		\label{extremal}
		Let $a,b,\epsilon$ be three positive parameters.
		Then there exists a unique $x_M \in [0, \pi)$ such that the map
		$$f(x) = a \sin(x-\epsilon) + b \cos(x)$$
		is extremal at $x_M$.
		Moreover $\tan(x_M) = \frac {a \cos(\epsilon)} {b-a \sin(\epsilon)}$ and $f(x_M-x) = f(x)$.
	\end{proposition}
	\begin{proof}
		We compute the derivative
		$$f'(x) = a \cos(x-\epsilon) - b \sin(x).$$
		We look for zeros of this map,
		\begin{align*}
			f'(x) = 0 \iff a \cos(x-\epsilon) = b \sin(x) \\
			\iff a \cos(x) \cos(\epsilon) = (b - a \sin(\epsilon)) \sin(x)\\
			\iff \tan(x) = \frac {a \cos(\epsilon)} {b-a \sin(\epsilon)}.
		\end{align*}
		The symmetry property comes from the fact that for all $a,b,\epsilon$ there exists $A, D$ such that $f(x) = A \cos(x-D)$.
	\end{proof}

	We apply the proposition in the setting where $a=\frac{3}{2} \sin(3 \theta_N / 2)$, $b = \cos(3 \theta_N / 2)^2$, $\epsilon = 3 \theta_N /2$.
	Then
	\begin{align*}
		\frac {a \cos(\epsilon)} {b-a \sin(\epsilon)} &=   \frac {3 \sin(\epsilon) \cos(\epsilon)} {2 \cos(\epsilon)^2 -3 \sin(\epsilon)^2}.
	\end{align*}
	Notice that $\tan(2\epsilon) = \frac {2\sin(\epsilon)\cos(\epsilon)} {\cos(\epsilon)^2-\sin^2(\epsilon)}$.
	Hence
	\begin{align*}
		\tan(x_M) \le \tan(2 \epsilon) \iff \frac {3 \sin(\epsilon) \cos(\epsilon)} {2 \cos(\epsilon)^2 -3 \sin(\epsilon)^2} \le \frac {2\sin(\epsilon)\cos(\epsilon)} {\cos(\epsilon)^2-\sin^2(\epsilon)} \\
		\iff \frac {3} {2 \cos(\epsilon)^2 -3 \sin(\epsilon)^2} \le \frac {2} {\cos(\epsilon)^2-\sin^2(\epsilon)} \\
		\iff \frac {2 \cos(\epsilon)^2 -3 \sin(\epsilon)^2}{3}  \ge \frac {\cos(\epsilon)^2-\sin^2(\epsilon)}{2}  \\
		\iff \frac {2 \cos(\epsilon)^2} {6}  \ge \frac {\sin^2(\epsilon)}{2}  \iff \tan(\epsilon)^2 \le \frac{2}{3}\\
	\end{align*}
	which is true for $N \ge 7$.
	Hence for any $N > 6$, $x_M < 3 \theta_N$.
	Notice moreover that $f(0) = \frac{3}{2} \sin(\epsilon)^2 + \cos(\epsilon)^2 \ge 0$ hence the map is maximal at $x_M$.
	As we are only considering even multiples $2k \theta_N$, by the symmetry property, the map
	\[
	\frac {3} {2} \sin(3\theta_N /2) \sin((2k-\frac {3} {2})\theta_N) + \cos(3 \theta_N /2)^2 \cos(2k\theta_N)
	\]
	is maximal for $k=1$ with value
	\begin{align*}
		\frac {3} {2} \sin(3\theta_N /2) \sin(\theta_N / 2) + \cos(3 \theta_N /2)^2 \cos(2\theta_N)\\
		\le \frac {3} {2} \sin(3\theta_N /2) \sin(\theta_N / 2) + \cos(3 \theta_N /2) \cos(\theta_N / 2).
	\end{align*}
	For the cases $N= 4,5,6$ one only has to check the inequality for $k=2$.
	This can be done directly.



\subsection{Computing with \texorpdfstring{$v_1$}{v1}}
	\label{ssec:computing_with_v_1}

\subsubsection{Setup}
	\label{sssec:setup_computing_with_v_1}
Recall that the vector $v_1$ is the second one in the cone, so we have
\[
	v_1 =
	\begin{bmatrix}
	  -L_2 cc_1\\
	-\frac{cc_1^2}{s_1}\\
	   L_1cc_2\\
	 \frac{cc_2^2}{c_2}
	\end{bmatrix}
	\text{ so }
	(-R)^k v_1 =
	\begin{bmatrix}
	-L_2\cdot cc_1\cdot cp_1 + \frac{cc_1^2}{s_1}\cdot sp_1\\
	-L_2\cdot cc_1\cdot sp_1 - \frac{cc_1^2}{s_1}\cdot cp_1\\
	 L_1\cdot cc_2\cdot cp_2 - \frac{cc_2^2}{c_2}\cdot sp_2\\
	 L_1\cdot cc_2\cdot sp_2 + \frac{cc_2^2}{c_2}\cdot cp_2
	\end{bmatrix}
\]
Computing now $MA(-R)^k v_1$ we find:
\[
	\begin{bmatrix}
			L_1 L_{2} (cp_2 - {cp_1}) + L_1 cc_1 \frac {sp_1} {s_1} - L_2 cc_2 \frac {sp_2} {s_2} \\
			L_2^2 sp_1 s_1 - L_1^2 sp_2 s_2 + cc_1^2 \frac {sp_1} {s_1} - cc_2^2 \frac {sp_2} {s_2} \\
			L_1 cp_2 - L_2 cp_1 + cc_1 \frac{sp_1}{s_1} - cc_2 \frac{sp_2}{s_2} \\
			-(L_2+1) cc_1 cp_1 + (L_1+1) cc_2 cp_2 - L_2 sp_1 s_1 + L_1 sp_2 s_2 + cc_1^2 \frac {sp_1} {s_1} - cc_2^2 \frac {sp_2} {s_2}
	\end{bmatrix}
\]

\subsubsection{The first entry of $MA(-R)^kv_1$}
	Notice that it is similar to the second entry for $v_0$.
	As in \autoref{sssec:the_second_entry_of_r_kv_0} we split the proof into two cases.
	The proof is very similar but exchanges the two cases.

	\subsubsection{The case $N/2 \le k < N$}
	Notice that as $L_2 > L_1$, $1-L_1 > 1-L_2$ and $-cp_2 \ge -cp_1$ on this range since $cp_2-cp_1 = 2 \sin(k\theta_N) \sin(2k \theta_N) \le 0$.
	Multiplying these three inequalities, we get
	$$-L_2 (1-L_1) cp_2 \ge -L_1 (1-L_2) cp_1$$
	then
	\begin{align*}
		L_1 L_2 cp_2 - L_1 L_2 cp_1 \geq L_2 cp_2 - L_1 cp_1.
	\end{align*}
	Hence the entry is not larger than
	\[
		L_1 \left(cc_1 \frac {sp_1} {s_1} - cp_1\right) - L_2 \left(cc_2 \frac {sp_2} {s_2} - cp_2\right)
	\]
	which is exactly the fourth coordinate of $v_0$, proved to be positive in \autoref{sssec:the_fourth_entry_of_v_0}.

	\subsubsection{The case $1 \le k \le N/2$}
	We compare the first entry with
	$$
		L_1 (cp_1 + cc_1 \frac {sp_1} {s_1}) - L_2 (cp_2 + cc_2 \frac {sp_2} {s_2}).
	$$
	which was proved to be positive in \autoref{sssec:the_case_k_ngeq_n_2_second_entry_v0}.
	Subtracting the above expression from the entry of interest, it suffices to show that
	\[
		L_2(L_1+1)cp_2 - L_1(L_2+1)cp_1 \geq 0.
	\]
	Using identities for $L_i$ from \autoref{eqn:L_i_identities} and factoring out $3(L_1+L_2)$ we reduce to
	\begin{align*}
		L_2 \frac{cp_2}{cc_2} - L_1 \frac{cp_1}{cc_1} \geq 0.
	\end{align*}
	The factor $\frac{cp_2}{cc_1}$ is positive in this range of $k$, so we reduced to showing:
	\[
	L_2 \frac{cc_1}{cc_2} - L_1\frac{cp_1}{cp_2} \geq 0
	\]
	This was done in \autoref{sssec:the_case_ngeq_k_geq_n_2_second_entry_v0}.

\subsubsection{The third entry of $MA(-R)^kv_1$}
	\label{sssec:the_third_entry_of_rkv_1}
	This entry is equal to
	\begin{align*}
		L_1 cp_2 &- L_2 cp_1 + cc_1 \frac{sp_1}{s_1} -cc_2 \frac{sp_2}{s_2} =\\
		&= (L_1+1) cp_2 - (L_2+1) cp_1 + cp_1-cp_2 + cc_1 \frac{sp_1}{s_1} - cc_2 \frac{sp_2}{s_2}\\
		&= 3(L_1+L_2) \left(\frac{cp_2}{cc_2} - \frac{cp_1}{cc_1}\right) + \left(cp_1+cc_1 \frac{sp_1}{s_1}\right) - \left(cp_2+ cc_2 \frac{sp_2}{s_2}\right)\\
		\intertext{Using the formulas from \autoref{sssec:the_case_k_ngeq_n_2_second_entry_v0} we get}
		&= 3(L_1+L_2) \frac{cp_2}{cc_1} \left(\frac{cc_1}{cc_2} - \frac{cp_1}{cp_2}\right) - \frac {\sin\left(3(k-\frac{1}{2})\theta_N\right)} {\sin(3 \theta_N / 2)} + \frac {\sin\left((k-\frac{1}{2})\theta_N\right)} {\sin(\theta_N / 2)} \\
		\intertext{with the angle tripling formula}
		&= 6(L_1+L_2) \frac{cp_2}{cc_1} \left(\cos(2\theta_N) - \cos(2k\theta_N)\right) \\
		& \quad + 2 \cdot \frac {\sin\left((k-\frac{1}{2})\theta_N\right)} {\sin(3 \theta_N / 2)} \left(\cos(\theta_N) - \cos\left((2k-1)\theta_N\right)\right).
	\end{align*}
	Using the difference of cosines formula, and factoring out 4, we get
	\begin{align*}
		3(L_1+L_2) \frac{cp_2}{cc_1} \sin\left((k+1)\theta_N\right) \sin\left((k-1)\theta_N\right) \\
		+ \frac {\sin\left((k-\frac{1}{2})\theta_N\right)} {\sin(3 \theta_N / 2)} \sin\left(k\theta_N\right) \sin\left((k-1)\theta_N\right).
	\end{align*}
	Factoring out $\frac {\sin\left((k-1)\theta_N\right)}{cc_1 \sin(3 \theta_N /2)}$, we get:
	\begin{align*}
		3(L_1+L_2) \cdot cp_2 \cdot \sin\big((k+1)\theta_N\big) \cdot \sin(3 \theta_N / 2)\\
		+ 2 \cdot \sin\left((k-\frac{1}{2})\theta_N\right) \cdot \sin\left(k\theta_N\right) \cdot \cos(3 \theta_N /2)^2 .
	\end{align*}
	For $1 \le k \le N /2$ this expression is clearly positive.
	Let us deal with $ N / 2 < k < N$.
	In this case, the first term of the sum is negative and the second positive.
	Using the fact that $L_1+L_2 < 1$ and $\sin\left((k-\frac{1}{2}) \theta_N\right) > \sin\left((k+1)\theta_N\right)$ on this domain we have the following lower bound
	\begin{align*}
		\sin\big((k+1)\theta_N\big) \left(3 \cdot \cos(k\theta_N) \cdot \sin(3 \theta_N / 2) + 2 \cdot \sin\left(k\theta_N\right) \cdot \cos(3 \theta_N /2)^2\right) .
	\end{align*}
	Factoring out the sine factor and using the fact that for $N\geq 5$, $\cos(3 \theta_N /2) \ge \frac{1}{2}$ we observe that
	\begin{align*}
		3 \cdot \cos(k\theta_N) \cdot \sin(3 \theta_N / 2) + 2 \cdot \sin\left(k\theta_N\right) \cdot \cos(3 \theta_N /2)^2\\
		\ge 3 \cdot \cos(k\theta_N) \cdot \sin(3 \theta_N / 2) + \sin\left(k\theta_N\right) \cdot \cos(3 \theta_N /2) \\
		\ge 3 \sin((k+\frac{3}{2}) \theta_N) - 2 \cos(3 \theta_N / 2) \cdot \sin(k \theta_N) \\
		\ge 3 \sin((k+\frac{3}{2}) \theta_N) - \sin(k \theta_N)\\
		\ge 3 \sin(k \theta_N) - \frac{3}{2} \theta_N - \sin(k \theta_N)\\
		\ge 2 \sin(k \theta_N) - \frac{3}{2} \theta_N.
	\end{align*}
	This last expression is minimal for $k = N-1$.
	We are then reduced to showing that $\sin(\theta_N) \ge \frac{3}{4} \theta_N$.
	This is true for $N > 4$ since $\cos(x) \ge \frac{3}{4}$ for all $x\in [0, \frac{\pi}{5}]$ and can be checked directly for $N=4$.

\subsubsection{The fourth entry of $MA(-R)^kv_1$}
	\label{sssec:the_fourth_entry_of_rkv_1}
	With the identity $$cc_i^2 + s_i^2 = 2 cc_i$$ for $i \in \{0,1\}$, we can rewrite the fourth coordinate as
	\begin{align*}
		-(L_2+1) cc_1 cp_1 + (L_1+1) cc_2 cp_2 - (L_2+1) sp_1 s_1 + (L_1+1) sp_2 s_2 + 2 \left(cc_1 \frac {sp_1} {s_1} - cc_2 \frac {sp_2} {s_2}\right) \\
		= 3(L_1+L_2)\left(-cp_1 + cp_2 - \frac{s_1}{cc_1} sp_1 + \frac{s_2}{cc_2} sp_2\right) + 2 \left(cc_1 \frac {sp_1} {s_1} - cc_2 \frac {sp_2} {s_2}\right).
	\end{align*}

	\subsubsection{The case $1 \le k \le N/2$}
	Notice that the second term of the sum is exactly the fourth coordinate for $v_3$ proved to be positive in \ref{sssec:the_fourth_entry_of_v_3}.

	The first term of the sum can be rewritten, using \autoref{s_cc}, as
	\[
		4 \cdot 3 (L_1+L_2) \frac {\cos\left((k+\frac{1}{2}) \theta_N\right)} {\cos(3\theta_N / 2)} \sin\left((k+1)\theta_N\right) \sin(k\theta_N)
	\]
	which is positive for $1 \le k < N / 2$.

	\subsubsection{The case $N/2 \le k < N$}
	In this case
	\begin{align*}
		3(L_1+L_2)\left(-cp_1 + cp_2 - \frac{s_1}{cc_1} sp_1 + \frac{s_2}{cc_2} sp_2\right) \ge 3\left(-cp_1 + cp_2 - \frac{s_1}{cc_1} sp_1 + \frac{s_2}{cc_2} sp_2\right)
	\end{align*}
	Using the previous computation and \autoref{sssec:the_fourth_entry_of_v_3}, we are reduced to showing positivity of
	\begin{align*}
		4 \cdot 3 \frac {\cos\left((k+\frac{1}{2}) \theta_N\right)} {\cos(3\theta_N / 2)} \sin\left((k+1)\theta_N\right) \sin(k\theta_N)\\
		+ 2 \cdot 2 \sin(k\theta_N) \left(\frac {\cos(\theta_N / 2)} {\sin(3\theta_N / 2)} - \frac {\cos(3 \theta_N / 2)} {\sin(3\theta_N / 2)} \cos(2 k \theta_N)\right).
	\end{align*}
	Factoring $\frac {4 \sin(k\theta_N)} {\sin(3\theta_N /2) \cos(3 \theta_N / 2)}$ out yields:
	\begin{align*}
		3 \sin(3\theta_N /2) \cos\left((k+\frac{1}{2})\theta_N\right) \sin((k+1)\theta_N)\\
		+ {\cos(3 \theta_N / 2)} \left({\cos(\theta_N / 2)} - {\cos(3 \theta_N / 2)} \cos(2 k \theta_N)\right).
	\end{align*}
	Using the relation
	\[
		2 \cos((k+\frac{1}{2})\theta_N) \sin((k+1)\theta_N) = \sin\left((2k+\frac{3}{2})\theta_N\right) + \sin\left(\theta_N / 2\right),
	\]
	we can rewrite
	\begin{align*}
	\frac{3}{2} {\sin(3\theta_N / 2)} \sin\left((2k+\frac{3}{2})\theta_N\right) - {\cos(3 \theta_N / 2)^2}  \cos\left(2 k \theta_N\right) \\
	+ \frac {3} {2} {\sin(3\theta_N / 2)} \sin(\theta_N /2) + {\cos(3 \theta_N / 2)} {\cos(\theta_N / 2)}
	\end{align*}
	which positivity is equivalent to
	\begin{align*}
	\frac{3}{2} {\sin(3\theta_N / 2)} \cos\left((2k+\frac{3}{2})\theta_N\right) - {\cos(3 \theta_N / 2)^2}  \cos\left(2 k \theta_N\right) \\
	\ge -\frac {3} {2} {\sin(3\theta_N / 2)} \sin(\theta_N /2) - {\cos(3 \theta_N / 2)} {\cos(\theta_N / 2)}.
	\end{align*}
	We now need to understand the minimal value of the function
	\begin{align*}
	\frac{3}{2} {\sin(3\theta_N / 2)} \sin\left((2k+\frac{3}{2})\theta_N\right) - {\cos(3 \theta_N / 2)^2}  \cos\left(2 k \theta_N\right) =\\
	a \sin(x-\epsilon) + b \cos(x) = f(x)
	\end{align*}
	where $a = \frac{3}{2} \sin(3\theta_N /2)$, $b = -\cos(3 \theta_N / 2)^2$ and $\epsilon = 3 \theta_N /2$.
	Then according to \autoref{extremal}, if $x_M$ is a point on which the function is extremal, then
	\begin{align*}
		\tan(x_M) &=  \frac {a \cos(\epsilon)}{b-a\sin(\epsilon)}\\
			  &= \frac {3 \sin(\epsilon) \cos(\epsilon)} {-2 \cos(\epsilon)^2 - 3 \sin(\epsilon)^2}.
	\end{align*}
	Hence
	\begin{align*}
		\tan(-x_M) \le \tan(2\epsilon) \iff \frac {3 \sin(\epsilon) \cos(\epsilon)} {2 \cos(\epsilon)^2 + 3\sin(\epsilon)^2} \le \frac {2 \sin(\epsilon) \cos(\epsilon)} {\cos(\epsilon)^2 - \sin(\epsilon)^2} \\
		\iff \frac {3} {2 \cos(\epsilon)^2 + 3\sin(\epsilon)^2} \le \frac {2} {\cos(\epsilon)^2 - \sin(\epsilon)^2} \\
		\iff \frac {2 \cos(\epsilon)^2 + 3\sin(\epsilon)^2}{3}  \ge \frac {\cos(\epsilon)^2 - \sin(\epsilon)^2}{2} \\
		\iff {4 \cos(\epsilon)^2 + 6\sin(\epsilon)^2}\ge {3\cos(\epsilon)^2 - 3\sin(\epsilon)^2}
	\end{align*}
	which is true.
	Hence $x_M \in (\pi-2\epsilon, \pi)$ and moreover as $f(0) =  - \frac{3}{2} \sin(\epsilon)^2 - \cos(\epsilon)^2 < 0$, $x_M$ is the maximum of $f$ and its minimum is at $x_m = x_M + \pi \in (2\pi-2\epsilon,2\pi)$.
	Thus the above expression takes its minimum value for $k = N-1$ and is bounded from below by
	\begin{align*}
		\frac{3}{2} \sin(3\theta_N / 2) \cos\left((2(N-1)+\frac{3}{2})\theta_N\right) - \cos(3 \theta_N / 2)^2  \cos\left(2(N -1) \theta_N\right) \\
		\frac{3}{2} {\sin(3\theta_N / 2)} \cos\left(\theta_N /2\right) - \cos(3 \theta_N / 2)^2  \cos\left(\theta_N\right) \\
		\ge -\frac {3} {2} {\sin(3\theta_N / 2)} \sin(\theta_N /2) - {\cos(3 \theta_N / 2)} {\cos(\theta_N / 2)}.
	\end{align*}
	The last inequality follows from the fact that $\cos(3 \theta_N/2)^2 \cos(\theta_N) \le \cos(3 \theta_N /2) \cos(\theta_N /2)$.

\subsubsection{The second entry of $MA (-R)^k v_1$}
	\label{sssec:the_second_entry_of_v_1}

	\begin{align*}
		L_2^2 sp_1 s_1 - L_1^2 sp_2 s_2 + cc_1^2 \frac {sp_1} {s_1} - cc_2^2 \frac {sp_2} {s_2} \\
		= (L_2^2 s_1 + \frac{cc_1^2}{s_1}) sp_1 - (L_1^2 s_2 + \frac{cc_2^2}{s_2}) sp_2.
	\end{align*}
	Factoring out $\sin(k \theta_N)$ we get
	\begin{align*}
		-(L_2^2 s_1 + \frac{cc_1^2}{s_1}) \left(2\cos(2k\theta_N)+1\right) + (L_1^2 s_2 + \frac{cc_2^2}{s_2}).
	\end{align*}
	As cosine is decreasing on $[0,\pi]$, and increasing on $[\pi, 2\pi]$, we only have to check that
	\begin{align*}
		-(L_2^2 s_1 + \frac{cc_1^2}{s_1}) \left(2\cos(2\theta_N)+1\right) + (L_1^2 s_2 + \frac{cc_2^2}{s_2}) \ge 0\\
		\intertext{and}
		-(L_2^2 s_1 + \frac{cc_1^2}{s_1}) \left(2\cos(2(N-1)\theta_N)+1\right) + (L_1^2 s_2 + \frac{cc_2^2}{s_2}) \ge 0.
	\end{align*}
	These two inequalities are equivalent to
	\begin{align*}
		\frac {L_1^2 s_2 + \frac{cc_2^2}{s_2}} {L_2^2 s_1 + \frac{cc_1^2}{s_1}} \ge 2\cos(2\theta_N)+1\\
		\iff \frac{s_1}{s_2} \cdot \frac {L_1^2 s_2^2 + cc_2^2} {L_2^2 s_1^2 + cc_1^2} \ge 2\cos(2\theta_N)+1\\
		\iff \left(2 \cos(3\theta_N/2) + 1\right) \frac {L_1^2 s_2^2 + cc_2^2} {L_2^2 s_1^2 + cc_1^2} \ge 2\cos(2\theta_N)+1.
	\end{align*}
	Again as cosine is decreasing on $[0,\pi]$, we only need to show
	\begin{align*}
		{L_1^2 s_2^2 + cc_2^2}  \ge {L_2^2 s_1^2 + cc_1^2} \\
		\iff
		(cc_1+cc_2)(cc_2-cc_1) \ge (L_2 s_1 + L_1 s_2)(L_2 s_2 - L_1 s_1)
	\end{align*}
	Notice that $(L_2 s_1 + L_1 s_2)(L_2 s_2 - L_1 s_1) \le s_2$ and $(cc_1+cc_2)(cc_2-cc_1) = 2 (cc_1+cc_2) \sin(\theta_N) \sin(2 \theta_N)$.
	Hence the inequality follows from the fact that $2 (cc_1+cc_2) \ge 1$.



\subsection{Computing with \texorpdfstring{$Bv_1$}{Bv1}}
	\label{ssec:computing_with_bv_1}

\subsubsection{Setup}
	\label{sssec:setup_computing_with_bv_1}
We finally have to handle the vector $Bv_1$, namely
\[
	Bv_1 =
	\begin{bmatrix}
		L_2\cdot cc_1\\
		-\frac{cc_1^2}{s_1}\\
		-L_1\cdot cc_2\\
		\frac{cc_2^2}{s_2}
	\end{bmatrix}
	\text{ so }
	(-R)^k Bv_1 =
	\begin{bmatrix}
		L_2cc_1\cdot cp_1 + \frac{cc_1^2}{s_1} sp_1\\
		L_2 cc_1\cdot sp_1 - \frac{cc_1^2}{s_1}\cdot cp_1\\
		-L_1 cc_2\cdot cp_2 - \frac{cc_2^2}{s_2}\cdot sp_2\\
		-L_1 cc_2\cdot sp_2 + \frac{cc_2^2}{s_2}\cdot cp_2
	\end{bmatrix}
\]
Computing now $MA(-R)^k Bv_1$ we find:
\[
\begin{bmatrix}
			L_1 L_{2} (cp_1 - {cp_2}) + L_1 cc_1 \frac {sp_1} {s_1} - L_2 cc_2 \frac {sp_2} {s_2} \\
			2 L_2 {cc_1} {cp_1} - 2 L_1 {cc_2} {cp_2} + L_1^2 sp_2 s_2 - L_2^2 sp_1 s_1 + cc_1^2 \frac {sp_1} {s_1} - cc_2^2 \frac {sp_2} {s_2} \\
			L_2 cp_1  - L_1 cp_2  + cc_1 \frac{sp_1}{s_1} - cc_2 \frac{sp_2}{s_2}\\
			(L_2-1) cc_1 cp_1 - (L_1-1) cc_2 cp_2 + L_2 sp_1 s_1 - L_1 sp_2 s_2 + cc_1^2 \frac {sp_1} {s_1} - cc_2^2 \frac {sp_2} {s_2}
\end{bmatrix}
\]

\subsubsection{The first entry of $MA(-R)^kBv_1$}
	Notice that it corresponds to the second entry for $v_0$ proved to be positive in \autoref{sssec:the_second_entry_of_r_kv_0}.

\subsubsection{The third entry of $M_{test}(-R)^kBv_1$}
	\label{sssec:the_third_entry_of_rkBv_1}
	This entry is very similar to the one in \autoref{sssec:the_third_entry_of_rkv_1} and we follow the same scheme of proof.
	\begin{align*}
		L_2 cp_1 &- L_1 cp_2 + cc_1 \frac{sp_1}{s_1} -cc_2 \frac{sp_2}{s_2} \\
		&= (L_2+1) cp_1 - (L_1+1) cp_2 + cp_2-cp_1 + cc_1 \frac{sp_1}{s_1} - cc_2 \frac{sp_2}{s_2}\\
		&= 3(L_1+L_2) \left(\frac{cp_1}{cc_1} - \frac{cp_2}{cc_2}\right) - \left(cp_1-cc_1 \frac{sp_1}{s_1}\right) + \left(cp_2- cc_2 \frac{sp_2}{s_2}\right)\\
		\intertext{Using formulas from \autoref{sssec:the_fourth_entry_of_v_0} we get}
		&= 3(L_1+L_2) \frac{cp_2}{cc_1} \left(\frac{cp_1}{cp_2} - \frac{cc_1}{cc_2}\right) - \frac {\sin\left(3(k+\frac{1}{2})\theta_N\right)} {\sin(3 \theta_N / 2)} + \frac {\sin\left((k+\frac{1}{2})\theta_N\right)} {\sin(\theta_N / 2)} \\
		\intertext{with the angle tripling formula}
		&= -6(L_1+L_2) \frac{cp_2}{cc_1} \left(\cos(2\theta_N) - \cos(2k\theta_N)\right) \\
		& \quad + 2 \cdot \frac {\sin\left((k+\frac{1}{2})\theta_N\right)} {\sin(3 \theta_N / 2)} \left(\cos(\theta_N) - \cos\left((2k+1)\theta_N\right)\right).
	\end{align*}
	Using the difference of cosines formula, and factoring out 4, we get
	\begin{align*}
		-3(L_1+L_2) \frac{cp_2}{cc_1} \sin\left((k+1)\theta_N\right) \sin\left((k-1)\theta_N\right) \\
		+ \frac {\sin\left((k+\frac{1}{2})\theta_N\right)} {\sin(3 \theta_N / 2)} \sin\left(k\theta_N\right) \sin\left((k-1)\theta_N\right).
	\end{align*}
	Factoring be $\frac {\sin\left((k-1)\theta_N\right)}{cc_1 \sin(3 \theta_N /2)}$,
	\begin{align*}
		-3(L_1+L_2) \cdot cp_2 \cdot \sin\big((k+1)\theta_N\big) \cdot \sin(3 \theta_N / 2)\\
		+ 2 \cdot \sin\left((k+\frac{1}{2})\theta_N\right) \cdot \sin\left(k\theta_N\right) \cdot \cos(3 \theta_N /2)^2 .
	\end{align*}
	For $N / 2 \le k < N$ this expression is clearly positive.
	Let us deal with $ 1 \le k < N /2$.
	In this case, the first term of the sum is negative and the second positive.
	Using the fact that $L_1+L_2 < 1$ and $\sin\left((k+1) \theta_N\right) < \sin\left((k+\frac{1}{2})\theta_N\right)$ on this domain we have the following lower bound
	\begin{align*}
		\sin\big((k+\frac{1}{2})\theta_N\big) \left(-3 \cdot \cos(k\theta_N) \cdot \sin(3 \theta_N / 2) + 2 \cdot \sin\left(k\theta_N\right) \cdot \cos(3 \theta_N /2)^2\right) .
	\end{align*}
	Factoring out the first and using the fact that for $N \ge 5$, $\cos(3 \theta_N /2) \ge \frac{1}{2}$ we observe that
	\begin{align*}
		-3 \cdot \cos(k\theta_N) \cdot \sin(3 \theta_N / 2) + \sin\left(k\theta_N\right) \cdot \cos(3 \theta_N /2) \\
		\ge 3 \sin((k-\frac{3}{2}) \theta_N) - 2 \cos(3 \theta_N / 2) \cdot \sin(k \theta_N) \\
		\ge 3 \sin((k-\frac{3}{2}) \theta_N) - \sin(k \theta_N)\\
		\ge 3 \sin(k \theta_N) - \frac{3}{2} \theta_N - \sin(k \theta_N)\\
		\ge 2 \sin(k \theta_N) - \frac{3}{2} \theta_N.
	\end{align*}
	This last expression is minimal for $k=1$.
	We are then reduced to showing that $\sin(\theta_N) \ge \frac{3}{4} \theta_N$.
	This was proved to be true for all $N \ge 4$ in \autoref{sssec:the_third_entry_of_rkv_1}.

\subsubsection{The second entry of $MA(-R)^kBv_1$}
	\label{ssec:the_second_entry_of_rkBv_1}
	Let us give an alternative expression for the second entry.
	\begin{align*}
		2 L_2 cc_1 cp_1 - 2 L_1 cc_2 cp_2 + L_1^2 sp_2 s_2 - L_2^2 sp_1 s_1 + cc_1^2 \frac {sp_1} {s_1} - cc_2^2 \frac {sp_2} {s_2} \\
		= 3(L_1+L_2) (cp_1 - cp_2) - 2 cc_1 cp_1 + 2 cc_2 cp_2 \\
		+ (L_1^2+1) sp_2 s_2 - (L_2^2+1) sp_1 s_1 + 2 cc_1 \frac {sp_1} {s_1} - 2 cc_2 \frac {sp_2} {s_2} \\
		= 3(L_1+L_2) (cp_1 - cp_2) \\
		+ (L_1^2+1) sp_2 s_2 - (L_2^2+1) sp_1 s_1 + 2 cc_1 \left(\frac {sp_1} {s_1} - cp_1\right) - 2 cc_2 \left(\frac {sp_2} {s_2} - cp_2\right) \\
	\end{align*}

	\subsubsection{The case $1 \le k \le N/2$}
	\label{sssec:the_second_entry_of_rkBv_1_case_1}
	Notice that on this domain, we have $cp_1 < cp_2$ and $cc_2 cp_2 - cc_1 cp_1 > 0$.
	Thus we are reduced to showing positivity for
	\[
		3(L_1+L_2) (cp_1 - cp_2) + (L_1^2+1) sp_2 s_2 - (L_2^2+1) sp_1 s_1 + 2 cc_1 \frac {sp_1} {s_1} - 2 cc_2 \frac {sp_2} {s_2}.
	\]
	Factoring $\sin(k \theta_N)$ out, we get
	\begin{align*}
		-6 (L_1+L_2) \sin(2k\theta_N) + \left(2 \frac {cc_2} {s_2} - (L_1^2+1) s_2 \right) - \left(2 \frac {cc_1} {s_1} - (L_2^2+1) s_1 \right) (2 \cos(2k\theta_N) +1)  \\
		= \left(2 \frac {cc_2} {s_2} - (L_1^2+1) s_2\right) - \left(2 \frac {cc_1} {s_1} - (L_2^2+1) s_1\right) \\
		- 2 \left(\left(2 \frac {cc_1} {s_1} - (L_2^2+1) s_1 \right) \cos(2k\theta_N) + 3 (L_1+L_2) \sin(2k\theta_N)\right)  \\
	\end{align*}
	Let us determine the maximal value of
	$$ \left(2 \frac {cc_1} {s_1} - (L_2^2+1) s_1 \right) \cos(2k\theta_N) + 3 (L_1+L_2) \sin(2k\theta_N) = f(2k \theta_N).$$
	There exists a unique $x_M \in [0, \pi)$ such that $f(x_M)$ is maximal, since $L_1+L_2 > 0$, and it satisfies
	$$\tan(x_M) = \frac {3(L_1+L_2) s_1} {2 cc_1 - (L_2^2+1) s_1} \le \frac {3 s_1} {2 - 2 s_1^2} = \frac {3} {2 c_1} \tan(\theta_N).$$
	For $N \ge 5$, $\cos(3 \pi / 2N) \ge \frac{1}{2}$, hence
	$$\tan(x_M) \le 3 \tan(\theta_N).$$
	By convexity of the tangent function on $[0,\frac{\pi}{2})$ this is less than $\tan(2\theta_N)$.
	This implies that the maximum of the entry is reached at $k=1$.
	Hence the minimal value of the entry on this domain is also reached at $k=1$ which evaluated on the first expression gives
	\begin{align*}
		3 (L_1+L_2) (c_1 - c_2) - L_1^2 s_2^2 + L_2^2 s_1^2 - cc_1^2 + cc_2^2 \\
		= 3 (L_1 + L_2) (c_1-c_2) + (cc_1+cc_2)(c_2-c_1) + L_2^2 s_1^2 - L_1^2 s_2^2 \\
		\ge (cc_1 + cc_2 - 3) (c_2-c_1).
	\end{align*}
	This last expression is positive, since $\cos(3\theta_N)+\cos(\theta_N) = 2 \cos(\theta_N) (\cos(2 \theta_N) + 2)$ which is increasing with $N$ and equal to $\frac{4}{\sqrt{2}} > 1$ for $N = 4$.

	\subsubsection{The case $N/2 < k < N$}
	We compute the difference with the fourth entry
	\begin{align*}
	-L_2 (L_2+1) sp_1 s_1 + L_1 (L_1 +1) sp_2 s_2 + 3(L_1+L_2)(cp_1-cp_2)
	\end{align*}
	Factoring $\sin(k \theta_N)$ out, we get
	\begin{align*}
		L_2 (L_2+1) s_1 (2 \cos(2k\theta_N)+1) - L_1 (L_1 +1) s_2 - 6(L_1+L_2) \sin(2k\theta_N).
	\end{align*}
	Notice that this expression evaluated at $k = 0$ is
	\begin{align*}
		3 L_2 (L_2+1) s_1  - L_1 (L_1 +1) s_2 \ge 2 L_2 (L_2+1) s_1 > 0.
	\end{align*}
	And at $k = \frac{N}{2} + \frac{1}{2}$,
	\begin{align*}
		L_2 (L_2+1) s_1 (-2 c_2 + 1) - L_1 (L_1 +1) s_2 + 6(L_1+L_2) s_2 \\
		 \ge 6(L_1+L_2) s_2 - L_2 (L_2+1) s_1 - L_1 (L_1 +1) s_2 \\
		 \ge 6(L_1+L_2) s_2 - 2 s_1 (L_1 + L_2).\\
	\end{align*}
	Where we have assumed that $N$ is large enough for $L_1$ to be positive.
	Otherwise we have the lower bound
	\begin{align*}
		 6(L_1+L_2) s_2 - L_2 (L_2+1) s_1 \ge 6 s_2 - 2 s_1.
	\end{align*}
	Hence positivity of the expression follows from the inequality $3 s_2 \ge s_1$ \textit{i.e.} $2 \cos(2 \theta_N) + 1 \le 3$.

\subsubsection{The fourth entry of $MA(-R)^kBv_1$}
	\label{sssec:the_fourth_entry_of_rkBv_1}
	We rewrite it as:
	\begin{align*}
		(L_2-1) cc_1 cp_1 - (L_1-1) cc_2 cp_2 + L_2 sp_1 s_1 - L_1 sp_2 s_2 + cc_1^2 \frac {sp_1} {s_1} - cc_2^2 \frac {sp_2} {s_2} \\
		=
		(1-L_1) cc_2 cp_2 - (1-L_2) cc_1 cp_1 + (1-L_1) sp_2 s_2 - (1-L_2) sp_1 s_1 \\
		+ 2 \left(cc_1 \frac {sp_1} {s_1} - cc_2 \frac {sp_2} {s_2}\right) \\
		=
		3 (L_1+L_2) \left(cp_1-cp_2+sp_1 \frac{s_1}{cc_1} - sp_2 \frac{s_2}{cc_2}\right)\\
		- 2(cc_1 cp_1 - cc_2 cp_2 + sp_1 s_1 - sp_2 s_2)
		+ 2\left(cc_1 \frac{sp_1}{s_1} - cc_2 \frac{sp_2}{s_2}\right)
	\end{align*}

	\subsubsection{The case $1 \le k \le N/2$}
	We have proved in \autoref{sssec:the_second_entry_of_rkBv_1_case_1} positivity of
	\[
		3(L_1+L_2) (cp_1 - cp_2) + (L_1^2+1) sp_2 s_2 - (L_2^2+1) sp_1 s_1 + 2 cc_1 \frac {sp_1} {s_1} - 2 cc_2 \frac {sp_2} {s_2}.
	\]
	We compute the difference of the fourth entry with this expression
	\begin{align*}
		2 (cc_2 cp_2 - cc_1 cp_1) + (L_2^2 + L_2+1) sp_1 s_1 - (L_1^2 + L_1 +1) sp_2 s_2.
	\end{align*}
	As $cc_2 cp_2 > cc_1 cp_1$ on the domain, we are reduced to showing
	\begin{align*}
		(L_2^2 + L_2+1) sp_1 s_1 \ge (L_1^2 + L_1 +1) sp_2 s_2
	\end{align*}
	which is true since $s_1 > s_2 > 0$, $sp_1 = (2 \cos(2 k \theta_N) + 1) sp_2 \ge sp_2 \ge 0$ on the domain, $L_1 < L_2$ and $L_1^2 + L_1 + 1 \ge 0$.
	This last inequality follows from the fact that $L_2 > \frac{1}{2}$, $3(L_1+L_2) = (L_2+1) cc_2 > 0$ and finaly $L_1+\frac{1}{2} > 0$.

	\subsubsection{The case $N/2 < k < N$}
	First notice that, by \autoref{s_cc},
	\begin{align*}
		cp_1-cp_2+sp_1 \frac{s_1}{cc_1} - sp_2 \frac{s_2}{cc_2} = - 4 \cdot \frac {\cos((k+\frac{1}{2})\theta_N)} {\cos(3\theta_N /2)} \cdot \sin(k \theta_N) \cdot \sin((k+1)\theta_N)
	\end{align*}
	which is non negative on the domain.
	Thus we are reduced to showing positivity of
	\begin{align*}
		\frac{cc_1}{s_1} sp_1 - cc_1 cp_1 - sp_1 s_1
		-\frac{cc_2}{s_2} sp_2 + cc_2 cp_2 + sp_2 s_2.
	\end{align*}
	Now, notice that
	\begin{align*}
		\frac{cc_1}{s_1} sp_1 - cc_1 cp_1 - sp_1 s_1
		= \frac {\cos(3\theta_N /2)} {\sin(3 \theta_N /2)} \left(-\sin(3k\theta_N)\right) - \left(1 + \cos(3\theta_N) \right) cp_1 - sp_1 s_1 \\
		= \frac{1}{\sin(3 \theta_N /2)} \big( \cos(3\theta_N /2) (-\sin(3 k \theta_N)) - \cos(3 k \theta_N) \sin(3 \theta_N /2) \big) \\
		- \cos(3 \theta_N) \cos(3 k \theta_N) + \sin(3k \theta_N) \sin(3 \theta_N)\\
		= \frac {- \sin(3(k+\frac{1}{2}) \theta_N)} {\sin(3\theta_N /2)} - \cos(3(k+1) \theta_N).
	\end{align*}
	Similarly, we have
	\begin{align*}
		\frac{cc_2}{s_2} sp_2 - cc_2 cp_2 - sp_2 s_2
		= \frac {- \sin((k+\frac{1}{2}) \theta_N)} {\sin(\theta_N /2)} - \cos((k+1) \theta_N).
	\end{align*}
	Subtracting the two terms, we get
	\begin{align*}
		\frac {- \sin(3(k+\frac{1}{2}) \theta_N)} {\sin(3\theta_N /2} + \frac {\sin((k+\frac{1}{2}) \theta_N)} {\sin(\theta_N /2)} - \cos(3(k+1) \theta_N) + \cos((k+1) \theta_N)\\
		= \frac {\sin((k+\frac{1}{2})\theta_N)} {\sin(3 \theta_N /2)} \left(2 \cos(\theta_N) - 2 \cos((2k+1)\theta_N)\right) - \cos(3(k+1) \theta_N) + \cos((k+1) \theta_N)\\
		\intertext{ Using the formula for sum of cosines,}
		= 4 \cdot \frac {\sin((k+\frac{1}{2})\theta_N)} {\sin(3 \theta_N /2)} \sin(k \theta_N) \sin((k+1)\theta_N) + 2 \cdot \sin((k+1)\theta_N) \sin(2 (k+1)\theta_N)\\
		= 4 \cdot \frac {\sin((k+\frac{1}{2})\theta_N)} {\sin(3 \theta_N /2)} \sin(k \theta_N) \sin((k+1)\theta_N) + 4 \cdot \sin((k+1)\theta_N)^2 \cos((k+1)\theta_N)\\
		= 4 \cdot \sin((k+1)\theta_N) \left(\frac {\sin((k+\frac{1}{2})\theta_N)} {\sin(3 \theta_N /2)} \sin(k \theta_N) + \sin((k+1)\theta_N) \cos((k+1)\theta_N)\right).
	\end{align*}
	The sine function is decreasing on this domain, hence we bound from below by
	\begin{align*}
		4 \sin((k+1)\theta_N)^2 \left(\frac {\sin(k \theta_N)} {\sin(3 \theta_N /2)}  + \cos((k+1)\theta_N)\right).
	\end{align*}
	We are thus reduced to showing positivity for
	\begin{align*}
		\sin(k \theta_N) + \sin(3 \theta_N /2) \cos((k+1)\theta_N)
	\end{align*}
	in the range $k > \frac{N}{2}$ and $k+1 < N$.
	As this last expression is decreasing on the domain, we only have to check it for $k = N-2$, \textit{i.e.}
	\begin{align*}
		\sin(2 \theta_N) - \sin(3 \theta_N /2) \cos(\theta_N) \ge 0.
	\end{align*}
	This is implied by the fact that $\sin(2\theta_N) \ge \sin(3 \theta_N /2)$.




\section{Symplectic Geometry and Causality}
	\label{sec:symplectic_geometry_and_causality}

\paragraph{Outline of section}
Introduced by Drumm in \cite{Drumm1992_Fundamental-polyhedra-for-Margulis-space-times}, crooked surfaces are used in Lorenzian geometry to produce fundamental domains for group actions, see e.g. \cite{DancigerGueritaudKassel2016_Geometry-and-topology-of-complete-Lorentz-spacetimes-of-constant-curvature}.
The basic definitions and constructions are introduced in \autoref{ssec:crooked_surfaces}.
In \autoref{ssec:disjointness_of_crooked_surfaces} we will reinterpret some of the disjointness criteria for crooked surfaces from \cite{BurelleCharetteFrancoeur2021_Einstein-tori-and-crooked-surfaces} using cones.
We extend their analysis to situations when crooked surfaces can touch, a geometric situation that occurs in our case.
Finally, in \autoref{ssec:domain_of_discontinuity}, we will construct a domain of discontinuity for the action of $\Gamma_N$ on $\LGr(V)$.
The construction will be in two stages: first an open set $\Omega^{\circ}$ built directly from the definition of crooked surfaces, then a larger domain $\Omega$ where we have added some sets where the crooked surfaces ``touch'', but on which the action is nonetheless properly discontinuous.

\subsubsection*{General conventions}
To lighten the notation, when it is clear from the context a nonzero element in a vector space and the induced line in the projectivization will carry the same notation.
Given their structure, it seems natural to us to call the objects in this section ``winged surfaces'' instead of ``crooked surfaces''.
We will continue to use the term ``crooked surface'' but denote them by $\cW\cS$, to denote that they consist of a wing and a stem.


\subsection{Crooked surfaces}
	\label{ssec:crooked_surfaces}

\subsubsection{Symplectic conventions}
	\label{sssec:symplectic_conventions}
Let $V$ be a real $4$-dimensional symplectic vector space.
Fix a basis $e_1,f_1,e_2,f_2$ such that the symplectic pairing denoted by $I$ satisfies
\[
	I(e_1,f_1)=I(e_2,f_2)=1
\]
Fix also an anti-symplectic involution $A$ given by the formula:
\[
	Ae_i = - e_i \quad Af_i = -f_i
\]
Let us note for convenience of reference that our basis is related to the one used in \cite[\S5]{BurelleCharetteFrancoeur2021_Einstein-tori-and-crooked-surfaces} by:
\begin{align}
	\label{eqn:uv_pm_e_i_f_i_notation}
	\begin{split}
	e_1 = u_+\\
	e_2 = v_+
	\end{split}
	\begin{split}
	f_1 &= -u_-\\
	f_2 &= v_-
	\end{split}
\end{align}
Note in particular the minus sign in front of $u_-$, which we hope minimizes the number of further negative signs later.

Given two vectors $v,v'\in V$, not proportional, we will denote by $L_{vv'}$ their $2$-dimensional span or its projectivization.
Typically we will consider the case when this is a Lagrangian.

\subsubsection{The cone}
	\label{sssec:the_cone_symplectic_simplicial}
In analogy with our constructions in previous sections, we will consider the cone
\[
	\cC := \bR_{\geq 0}e_1 + \bR_{\geq 0}e_2
	+\bR_{\geq 0}f_1 + \bR_{\geq 0}f_2
\]
We regard this as a projective cone $\cC\subset \bP(V_\bR)$, and later will denote by $\overset{\circ}{\cC}$ the interior of the cone, also in projective space.
The projective cone is a tetrahedron with four of the edges contained in the Lagrangian planes (projective lines):
\begin{align*}
	\begin{split}
	L_{e_1e_2} := \text{span}(e_1,e_2)\\
	L_{f_1f_2} := \text{span}(f_1,f_2)
	\end{split}
	\begin{split}
	L_{e_1f_2} := \text{span}(e_1,f_2)\\
	L_{e_2f_1} := \text{span}(e_2,f_1)
	\end{split}
\end{align*}
The remaining two edges are contained in the projectivization of subspaces orthogonal for the symplectic form, and on which the symplectic form is non-degenerate:
\begin{align*}
	S_1 = \text{span}(e_1,f_1)\quad
	S_2 = \text{span}(e_2,f_2)\
\end{align*}

\subsubsection{Indefinite inner product conventions}
	\label{sssec:indefinite_inner_product_conventions}
Let now $W:=\Lambda^2_0V$ denote the subspace of the second exterior power which wedges to zero against $e_1\wedge f_1 + e_2\wedge f_2$ (or equivalently is in the kernel of the symplectic form).
Then $W$ is equipped with a nondegenerate quadratic form of signature $(2,3)$ given by taking the wedge product of elements and using the trivialization of $\Lambda^4 V$ by the volume form induced from the symplectic pairing.
Given these sign conventions, we will say that a subspace is time-like if it is positive definite, and space-like if it is negative definite.

We will use the explicit basis of $W$ given by
\[
	e_1\wedge e_2 \quad
	e_1\wedge f_2 \quad
	e_1\wedge f_1 - e_2\wedge f_2
	\quad
	f_1\wedge e_2
	\quad f_1\wedge f_2
\]
In this basis the induced involution (which actually preserves the inner product) has eigenvalue $+1$ on $e_1\wedge e_2$ and $f_1\wedge f_2$, and eigenvalue $-1$ on the remaining three basis vectors.

\subsubsection{The Lagrangian Grassmannian}
	\label{sssec:the_lagrangian_grassmannian}
Recall next that the Lagrangian Grassmannian $\LGr(V)$ is equal to the quadric of null vectors in $\bP(W)$:
\[
	\LGr(V)=\{[w]\in \bP(W)\colon w^2=0\}\subset \bP(W).
\]
It is equipped with a conformal class of Lorenzian metrics of signature $(1,2)$ and is frequently also called an Einstein universe and denoted $\Ein^{1,2}$.

\subsubsection{Photons}
	\label{sssec:photons_def}
Associated to a nonzero vector $v\in V$ (rather, the corresponding point in $\bP(V)$) there is a ``photon'' of Lagrangians:
\[
	\phi(v):=\{L\in \LGr(V)\colon v\in L\} \subset \LGr(V)
\]
The photon can also be identified as
\[
	\phi(v) = \bP\left(v\wedge v^{\perp}\right)\isom \bP\left(v^{\perp}/v\right)
\]
where $v^{\perp}\subset V$ denotes the symplectic-orthogonal to $v$.

\subsubsection{The wings of the crooked surface}
	\label{sssec:the_wings_of_the_crooked_surface}
We can now define the crooked surface.
It consists of two ``wings'' and a ``stem'' (the stem also decomposes into two pieces, see \autoref{sssec:the_stem_of_the_crooked_surface} below).

Consider the ``interval of lines''
\begin{align}
	\label{eqn:edge_of_tetrahedron_def}
	[e_1,e_2]:=\{se_1+te_2 \colon s,t\geq 0, s+t = 1\}\subset \bP(V)
\end{align}
which is one of the boundary edges of the cone $\cC$.
Then the $e$-wing is defined as:
\[
	\cW_e:=\phi([e_1,e_2]) = \bigcup_{l\in [e_1,e_2]}\phi(l).
\]
Analogously define the $f$-wing:
\[
	\cW_f:= \phi([f_1,f_2]).
\]

\subsubsection{The stem of the crooked surface}
	\label{sssec:the_stem_of_the_crooked_surface}
Consider the ``Einstein torus'' consisting of Lagrangians spanned by one vector in each of $S_1,S_2$:
\begin{align}
	\label{eqn:Einstein_torus_definition}
	\Ein^{1,1} (S_1,S_2) = \{L = w_1\wedge w_2 \colon w_i \in S_i\}
\end{align}
Then the stem is defined as:
\[
	\cS:=\{L\in \Ein^{1,1}(S_1,S_2)\colon |\text{Maslov}(L_{e_1f_2},L,L_{e_2f_1})| = 2\}
\]
Note that we can further decompose the stem as $\cS= \cS^+\coprod \cS^-$ according to the sign of the Maslov index (see \autoref{sssec:making_the_stem_explicit} below for the definition of the Maslov index).

Set now the crooked surface to be
\begin{align}
	\label{eqn:crooked_surface_def}
	\cW \cS := \cW_e \coprod \cS \coprod \cW_f
\end{align}
Observe that according to the definitions, the wings are relatively closed subsets, while the stem is a relatively open set in $\cW\cS$.

\subsubsection{Making the stem explicit}
	\label{sssec:making_the_stem_explicit}
Recall that the Maslov index of a Lagrangian is defined as the index of the quadratic form obtained from the symplectic form, using the direct sum decomposition provided by two transverse Lagrangians.
In the case at hand $V = L_{e_1f_2}\oplus L_{e_2f_1}$ and the quadratic form, denoted $Q_{12}$, comes out to be
\begin{align*}
	Q_{12}(\alpha_1 e_1 + \alpha_2 e_2 + \beta_1f_1 + \beta_2f_2) & :=
	I(\alpha_1e_1 + \beta_2f_2, \alpha_2 e_2 + \beta_1 f_1)\\
	& = \alpha_1 \cdot \beta_1 - \alpha_2\cdot \beta_2
\end{align*}
If the Lagrangian $L$ is spanned by $w_i\in S_i$ with coordinates
\[
	w_1 = \alpha_1 e_1 + \beta_1 f_1 \quad
	w_2 = \alpha_2 e_2 + \beta_2 f_2
\]
then we observe that $w_1$ and $w_2$ are orthogonal with respect to $Q_{12}$ and $Q_{12}(w_i)=(-1)^{i+1}\alpha_i\cdot \beta_i$.
So for the Lagrangian to belong to the stem both products have to be of opposite sign.
We thus have:
\begin{align*}
	\begin{split}
	L &= w_1\wedge w_2 \text{ belongs to:}\\
	w_1& = \alpha_1 e_1 + \beta_1 f_1\\
	w_2& = \alpha_2 e_2 + \beta_2 f_2
	\end{split}
	&
	\begin{split}
	\cS^+&\\
	\alpha_1\cdot \beta_1& >0\\
	\alpha_2\cdot \beta_2& <0
	\end{split}
	\begin{split}
	\cS^-&\\
	\alpha_1\cdot \beta_1& <0\\
	\alpha_2\cdot \beta_2& >0
	\end{split}
\end{align*}

\subsubsection{Action of reflection}
	\label{sssec:action_of_reflection}
The anti-symplectic involution $A$ preserves the crooked surface as a set.
Furthermore it exchanges the two components of the stem: $A\cS^{\pm}=\cS^\mp$ and fixes as a set each photon on the wings.
On individual photons on the wings, it fixes two points and exchanges the two complementary regions.
Explicitly, on the photon $\phi(s\cdot e_1 + t\cdot e_2)$ belonging to the $e$-wing, the fixed points are the two Lagrangians $e_1\wedge e_2$ and $(se_1+te_2)\wedge(-tf_1+sf_2)$.
The formula for photons on the $f$-wing is analogous.



\subsection{Disjointness of crooked surfaces}
	\label{ssec:disjointness_of_crooked_surfaces}

In this section, we proceed to study the geometric configurations that crooked surfaces, photons, and Lagrangians, can be in.
First, some of the results from \cite{BurelleCharetteFrancoeur2021_Einstein-tori-and-crooked-surfaces} can be reinterpreted using the cones that we introduced earlier.
This gives transparent geometric conditions for when photons, or crooked surfaces, are disjoint.
We then further refine our analysis to situations when crooked surfaces can ``touch'', an inevitable situation when facing groups with unipotent elements.

\begin{proposition}[Disjointness of photon from crooked surface]
	\label{prop:disjointness_of_photon_from_crooked_surface}
	Consider a vector $v\in V$ with coordinates
	\[
		v = \alpha_1 e_1 + \alpha_2 e_2 + \beta_1 f_1 + \beta_2 f_2.
	\]
	The following are equivalent:
	\begin{enumerate}
		\item The photon $\phi(v)\in \LGr(V)$ is disjoint from the crooked surface $\cW\cS$.
		\item The following inequalities hold:
			\[
				\alpha_1\cdot\alpha_2>0
				\quad \text{ and }\quad
				\beta_1\cdot\beta_2>0.
			\]
		\item The (projectivized) vector $v$ is either in the interior of the cone $\overset{\circ}{\cC}$ or in the interior of the reflected cone $A\overset{\circ}{\cC}$.
	\end{enumerate}
\end{proposition}
\begin{proof}
	The equivalence of (i) and (ii) is simply a restatement of \cite[Lemma 9]{BurelleCharetteFrancoeur2021_Einstein-tori-and-crooked-surfaces}.
	The geometric interpretation of (ii) with cones in (iii) follows directly.
	Indeed the cone $\cC$ corresponds to vectors with all coordinates of the same sign, while $A\cC$ to those vectors where the pair $(\alpha_1:\alpha_2)$ has the same sign, and so does $(\beta_1:\beta_2)$, but the signs of the two pairs are opposite.
\end{proof}

\subsubsection{Position of a Lagrangian}
	\label{sssec:position_of_a_lagrangian}
Let us also list the possibilities for the position of a Lagrangian relative to the crooked surface, when viewing the picture in $\bP(V)$.
Regard the Lagrangian as projectivized in $\bP(V)$, thus yielding a line.
If the Lagrangian intersects the interior of either $\cC$ or $A\cC$ then it is clearly in the corresponding component in $\LGr(V)$, since it lies on a photon entirely contained in such a component.
But the Lagrangian could intersect also just the boundary, say the boundary of $\cC$ for simplicity.
If it intersects a vertex, or more generally one of the edges $[e_1,e_2]$ or $[f_1,f_2]$ then it clearly lies on the respective wing.
If it intersects one of the edges $[e_1,f_1]$ or $[e_2,f_2]$, the assumption that it doesn't go through the interior of $\cC$ implies that $L$ must belong to one of pieces of the stem $\cS^{\pm}$.
Finally, suppose $L$ intersects the edge $[e_1,f_2]$ (for $[e_2,f_1]$	the analysis is similar), say in $v=\alpha_1e_1+\beta_2f_2$ with $\alpha_1,\beta_2>0$.
Then its orthogonal complement is spanned by $e_1,f_2,\beta_2f_1 + \alpha_1 f_1$, and unless $L$ is the span of $e_1,f_2$, it is immediate that some linear combination of vectors in $L$ lies in the interior of $\cC$, placing $L$ in the interior of the respective component.

Next, we have the following criterion:
\begin{proposition}[Disjointness of crooked surfaces]
	\label{prop:disjointness_of_crooked_surfaces}
	Let $\cW\cS,\cW\cS$ be two crooked surfaces, with corresponding vectors $e_i,f_i, e_i',f_i'$, cones $\cC,\cC'$ and anti-symplectic involutions $A,A'$.
	The following are equivalent:
	\begin{enumerate}
		\item The crooked surfaces $\cW\cS$ and $\cW\cS'$ are disjoint.
		\item The photons $\phi(e_1),\phi(e_2),\phi(f_1),\phi(f_2)$ are disjoint from $\cW\cS'$ and also
		the photons $\phi(e_1'),\phi(e_2'),\phi(f_1'),\phi(f_2')$ are disjoint from $\cW\cS$.
		\item The following vectors are contained in the interiors of the cones:
		\[
			e_1,e_2,f_1,f_2 \in \overset{\circ}{\cC'} \cup A'\cdot \overset{\circ}{\cC'}
			\text{ and }
			e_1',e_2',f_1',f_2' \in \overset{\circ}{\cC} \cup A\cdot \overset{\circ}{\cC}.
		\]
	\end{enumerate}
\end{proposition}
\begin{proof}
	Again, the equivalence of (i) and (ii) is the content of \cite[Thm. 10]{BurelleCharetteFrancoeur2021_Einstein-tori-and-crooked-surfaces}.
	The geometric interpretation in (iii) follows from \autoref{prop:disjointness_of_photon_from_crooked_surface} applied to each photon individually.
\end{proof}

Let us also recall the basic facts on the topology of the crooked surface:
\begin{proposition}[Connectivity and topology of crooked surface]
	\label{prop:connectivity_and_topology_of_crooked_surface}
	Given a crooked surface $\cW\cS$:
	\begin{enumerate}
		\item It is homeomorphic to a Klein bottle: $\cW\cS \homeo \bK^2$.
		\item Its complement $\LGr(V)\setminus \cW\cS$ has two connected components.
		The components can be labeled according to the cones $\cC$ and $A\cC$, corresponding to the photons which are contained in one component or the other.
		The components are exchanged by the anti-symplectic involution $A$.
	\end{enumerate}
\end{proposition}
\begin{proof}
	That crooked surfaces are homeomorphic to Klein bottles is \cite[Thm.~8.3.1]{BarbotCharetteDrumm2008_A-primer-on-the-21-Einstein-universe}.
	That a crooked surface disconnects $\LGr(V)$ is proved in \cite[Thm 3.16]{CharetteFrancoeurLareau-Dussault2014_Fundamental-domains-in-the-Einstein-universe}, and that the anti-symplectic involution exchanges the two components follows immediately as well.
\end{proof}

In our geometric applications, a ``touching'' of crooked surfaces occurs, because the cones can intersect along edges or faces.
This situation is handled in \autoref{prop:tangency_of_crooked_surfaces} below, and we need some preliminaries on Einstein tori.

\subsubsection{Einstein tori}
	\label{sssec:einstein_tori}
Recall that associated to a symplectic-orthogonal splitting $V=S_1\oplus S_2$ with symplectically non-degenerate summands, we defined an Einstein torus $\Ein^{1,1}(S_1,S_2)$ in \autoref{eqn:Einstein_torus_definition}.
As a real projective algebraic manifold it is naturally isomorphic to a product of two projective lines $\bP(S_1)\times \bP(S_2)$, since it is also a quadric in the projectivization of a space of signature $(2,2)$ (see also \cite[\S5.3]{BarbotCharetteDrumm2008_A-primer-on-the-21-Einstein-universe}).
Note that the Einstein torus embedded in $\bP(W)$ is given as the intersection of the orthogonal complement of a negative-definite vector with the null quadric (i.e. $\LGr(V)$).
Furthermore the torus is equipped with a natural conformal class of Lorenz metric, for which the light rays are fibers of the projection to one coordinate $\bP^1$-factor.

Suppose given now two Einstein tori $E,E'\subset \LGr(V)$.
Then the intersection $E\cap E'$ viewed as a subset of $E=\bP^1\times \bP^1$ is a $(1,1)$-curve, i.e. cut out by a homogeneous equation of bi-degree $(1,1)$ in each of the homogeneous coordinates on $\bP^1\times \bP^1$.
Three possibilities can occur for $E\cap E'$ inside $E$ (see \cite[\S3]{BurelleCharetteFrancoeur2021_Einstein-tori-and-crooked-surfaces}): it can be a timelike curve, it can be a spacelike curve, or it can be the union of two intersecting light rays.
In the first two cases the intersection projects isomorphically to any of the $\bP^1$-factors, in the last case each light ray projects isomorphically to a corresponding $\bP^1$-factor.

Below is a criterion for when crooked surfaces can touch.
Recall that if $p,q$ are nonzero vectors then $[p,q]$ denotes the closed projective interval of their positive linear combinations  (see \autoref{eqn:edge_of_tetrahedron_def}) and we will denote by $(p,q)$ the open interval where both coefficients are strictly positive.

\begin{proposition}[Tangency of crooked surfaces]
	\label{prop:tangency_of_crooked_surfaces}
	Let $\cC,\cC'$ be cones determining crooked surfaces $\cW\cS,\cW\cS'$, with notation as in \autoref{prop:disjointness_of_crooked_surfaces}.
	\begin{enumerate}
		\item Suppose that $[e_1,e_2]=[e_1',e_2']$ and $f_1',f_2'$ belong to the interior $\cC$.
		Then
		\[
			\cW\cS\cap \cW\cS' = \cW_e \text{ which also equals }\cW_e'
		\]
		i.e. the surfaces intersect along the $e$-wing but nowhere else.
		\item Suppose that we have:
		\begin{align*}
			\begin{split}
			f_1' & = f_1\\
			e_2' & = e_2 + f_1
			\end{split}
			\begin{split}
			f_2' & = f_2+e_2+\tfrac12 f_1\\
			e_1' & = e_1 + f_2 + \tfrac12 e_2 + \tfrac16 f_1
			\end{split}
		\end{align*}
		Then:
		\[
			\cW\cS\cap \cW\cS' = \phi(f_1)\text{ which also equals }\phi(f_1')
		\]
		i.e. the surfaces intersect along a photon on their $f$-wings but nowhere else.
	\end{enumerate}
\end{proposition}
\noindent The formulas in case (ii) arise when the cone $\cC'$ is the image of $\cC$ under a maximally unipotent symplectic matrix, which preserves the flag $f_1\subseteq L_{f_1e_2}\subseteq f_1^{\perp}\subseteq V$.
Case (i) arises when $\cC'$ is the image of $\cC$ under a rank $1$ symplectic unipotent matrix.
\begin{proof}
	For both cases, it is immediate that the stated sets are in the intersection.
	We must check that no intersections occur elsewhere.

	Consider case (i).
	First, let us see that $\cW_f$ is disjoint from $\cW\cS'$, and similarly for $\cW_f'$ and $\cW\cS$.
	Indeed the vectors in $V$ spanning the photons are assumed in the interior of a cone, and so are the segments connecting them, so the disjointness of an $f$-wing from the (other) crooked surface follows by \autoref{prop:disjointness_of_photon_from_crooked_surface}.

	To see that the stems $\cS,\cS'$ don't intersect either, let $E,E'$ be the Einstein tori containing them.
	By the discussion in \autoref{sssec:einstein_tori} we see that $E\cap E'=\phi(e_1)\cup \phi(e_2)$.
	Indeed the two photons are clearly in the intersection (they give the joining places of the $e$-wings to the stem), and since the tori are distinct they account for all the intersection points.
	Since the stems are in the complement of these ``joining'' photons, their disjointness follows.

	Consider now case (ii).
	The photons in the wing $\cW_e'$ don't intersect $\cW\cS$ in the range $[e_1',e_2')$ by the criterion of \autoref{prop:disjointness_of_photon_from_crooked_surface}.
	The photon $\phi(e_2')$ intersects $\cW\cS$ at the Lagrangian $L_{e_2f_1}$ and note except for this one point of intersection, this photon lies in the component of $\LGr(V)\setminus \cW\cS$ corresponding to $\cC$ (by an arbitrarily small perturbation it can be ``pushed'' to be entirely in the interior of that components).

	The photons in the wing $\cW_f'$ intersect $\cW\cS$ at the Lagrangian $L_{f_1'f_2'}=L_{f_1f_2'}\in \phi(f_1)\subset \cW_f$, but nowhere else and lie, except for this one point of intersection, in the component corresponding to $\cC$ in $\LGr(V)\setminus \cW\cS$ (again, a small push takes them to the interior).

	Finally, the Einstein tori $E,E'$ containing the stems intersect in two photons $E\cap E'=\phi(f_1)\cup \phi(v)$ where $v=f_2+\tfrac12 e_2$ is the point of intersection between the line $f_2'e_2'$ and the segment $[f_2,e_2]$.
	Let us check that the piece of the photon $\phi(v)$ that belongs to the stem $\cS'$ is not in the stem $\cS$.
	Indeed, that piece consists of Lagrangians of the form $L_{vv'}$ where $v'=\alpha_1'e_1'+\beta_1'f_1'$ with $\alpha_1'\cdot\beta_1'>0$, since $v$ lies outside the segment $[f_2',e_2']$.
	This Lagrangian will intersect the subspace $S_1$ spanned by $e_1,f_1$ at the point $\alpha_1'(e_1+\tfrac16 f_1) + \beta_1'f_1$, which can be checked directly from the formulas for $v,e_1',f_1'$.
	But this intersection point has both coordinates positive with respect to $e_1,f_1$, and since $v$ also has both coordinates positive with respect to $e_2,f_2$, it follows that this Lagrangian is not in $\cS$ (see \autoref{sssec:making_the_stem_explicit}).
\end{proof}

\begin{corollary}[Cutting along crooked surfaces]
	\label{cor:cutting_along_crooked_surfaces}
	Suppose that two crooked surfaces $\cW\cS,\cW\cS'$ are either in the configuration of \autoref{prop:disjointness_of_crooked_surfaces}, i.e. disjoint, or in one of the configurations in \autoref{prop:tangency_of_crooked_surfaces}.

	Then $\cW\cS'$ is entirely contained in one component of $\LGr(V)\setminus \cW\cS$ in the first case, or contained in a component except for a set of photons along which in intersects $\cW\cS$ in the second case.
\end{corollary}



\subsection{Domain of discontinuity}
	\label{ssec:domain_of_discontinuity}

\subsubsection{Setup}
	\label{sssec:setup_domain_of_discontinuity}
We now apply the preceding formalism of crooked surfaces to analyze the domains of discontinuity for the groups $\tilde{\Gamma}_N$ from \autoref{thm:ping_pong_property_of_cones}, using the cones constructed in \autoref{sec:cones_and_ping_pong}.

Let $I:=\{0,0',\ldots, (N-1),(N-1)'\}$ denote the indexing set for the cones.
Given $i\in I$, we have a cone $\cC_i\subset \bP(V)$, and an associated reflection $A_i\in \GSp(V)$.
In the Lagrangian Grassmannian we then obtain a crooked surface $\cW\cS_i\subset \LGr$, and its complement decomposes into two open sets
\[
	\LGr(V)\setminus \cW\cS_i = \cL_{i,s}\coprod \cL_{i,b}
\]
where the ``small'' open set $\cL_{i,s}$ is associated to the component determined by the cone $\cC_i$, and the ``big'' open set $\cL_{i,b}$ is associated to the component determined by the cone $A_i\cdot \cC_i$.
The reflection $A_i$ preserves $\cW\cS_i$ as a set, and exchanges the two components $\cL_{i,\bullet}$.
We will denote by $\ov{\cL}$ the closure of a component (so $\ov{\cL}=\cL\coprod \cW\cS$).

\subsubsection{Left and right adjacency}
	\label{sssec:left_and_right_adjacency}
Recall that the indexing set $I$ for the cones is cyclically ordered.
Let then $r(i)$, resp. $l(i)$, denote the right, resp. left, neighbors of the element $i$.
We will also use the composition of reflections:
\[
	T_{i,r}:= A_i  A_{r(i)} \quad T_{i,l} := A_i  A_{l(i)}
\]
which for adjacent vertices satisfy $T_{i,r}\cdot T_{r(i),l}=1$ and $T_{i,l}\cdot T_{l(i),r}=1$.
Note that each $T_{i,l/r}$ is a unipotent transformation taking the cone $\cC_i$ to itself, and one of the matrices is a rank $1$ unipotent while the other is maximally unipotent.

\subsubsection{Finite approximations to limit set and domain of discontinuity}
	\label{sssec:finite_approximations_to_limit_set_and_domain_of_discontinuity}
We can now combine the calculations with containments of cones from \autoref{thm:ping_pong_property_of_cones} with the disjointness/touching criteria from \autoref{prop:disjointness_of_crooked_surfaces} and \autoref{prop:tangency_of_crooked_surfaces}.
It follows that when $i\neq j$, we have that $\cW\cS_{j}\subset \ov{\cL_{i,b}}$ and more generally $\ov{\cL_{j,s}}\subset \ov{{\cL_{i,b}}}$.

Let us define
\[
	\Lambda_{1}:=\bigcup_{i\in I}\ov{\cL_{i,s}}
	\quad
	\text{ and }
	\quad
	\Omega_{1}:= \bigcap_{i\in I}\cL_{i,b} = \LGr(V)\setminus \Lambda_1
\]
These provide a first approximation to the domain of discontinuity $\Omega$ and limit set $\Lambda$.
We can define $\Lambda_n$ and $\Omega_n:=\LGr(V)\setminus \Lambda_n$ recursively, or in a more direct manner:
\begin{align}
	\label{eqn:Lambda_n_Omega_n_def}
	\begin{split}
	\Lambda_n & := \{x\in \LGr(V) \colon \exists i_1,\ldots, i_n \in I \text{ s.t. } i_{l}\neq i_{l+1}\\
	& \qquad \text{ and } x_{i_1}\in \ov{\cL_{i_1,s}} \text{ s.t. }
	x = A_{i_n}\cdots A_{i_2}\cdot x_{i_1} \}\\
	\Omega_n & := \{x\in \LGr(V) \colon \forall i_1,\ldots, i_n \in I \text{ s.t. } i_{l}\neq i_{l+1}\\
	& \qquad \text{ we have that } A_{i_2}\cdots A_{i_n}\cdot x \notin  \ov{\cL_{i_1,s}} \}
	\end{split}
\end{align}
It is immediate from the definitions that $\Omega_n=\LGr(V)\setminus \Lambda_n$, and that $\Lambda_n$ is closed (resp. $\Omega_n$ is open).
Let us point out that the sequence $i_1,\ldots,i_n$ which certifies that $x\in\Lambda_n$ need not be uniquely associated to $x$.

We also have that $\Lambda_{n+1}\subset \Lambda_n$ since if $x\in \Lambda_{n+1}$ with $x = A_{i_{n+1}}\cdots A_{i_2}\cdot x_{i_1}$ then we can also use $x_{i_2}:=A_{i_2}x_{i_1}$ and the last $n$ terms of the sequence, to see that $x\in \Lambda_n$, since $A_{i_2}\cdot \ov{\cL_{i_1,s}}\subset \ov{\cL_{i_2,s}}$.
Similarly note that $\Omega_{n+1}\supset \Omega_n$, since if $x\notin \Omega_{n+1}$ then there exists a sequence $i_1,\ldots, i_{n+1}$ with $A_{i_2}A_{i_3}\cdots A_{i_{n+1}}x\in \ov{\cL_{i_1,s}}$, but then $A_{i_3}\cdots A_{i_{n+1}}x\in A_{i_2}\ov{\cL_{i_1,s}}\subset \ov{\cL_{i_2,s}}$ showing that $x\notin \Omega_n$ either.

\subsubsection{A preliminary domain of discontinuity}
	\label{sssec:a_preliminary_domain_of_discontinuity}
We can define now the sets
\[
	\Lambda^{\circ}:= \bigcap_{n\geq 1} \Lambda_n \text{ and its complement }\Omega^{\circ}:=\bigcup_{n\geq 1}\Omega_n.
\]
By construction $\Lambda^{\circ}$ is closed and $\Omega^{\circ}$ is open, and both sets are $\tilde{\Gamma}_N$-invariant.
We will see in \autoref{sssec:enlarging_the_domain_of_discontinuity} below that $\Omega^{\circ}$ can be slightly enlarged to a bigger $\tilde{\Gamma}_N$-invariant set, while its complement $\Lambda^{\circ}$ can be slightly enlarged.

\subsubsection{Boundary of the fundamental domain}
	\label{sssec:boundary_of_the_fundamental_domain}
If we denote by $\ov{\Omega}_1^{rel}\subset \Omega^{\circ}$ the relative closure of the first domain $\Omega_1$, then it is immediate to check from the properties of the action, and the definitions, that the $\tilde{\Gamma}_N$-orbit of any point intersects $\ov{\Omega}_1^{rel}$.

Let us further analyze the boundary of this fundamental domain.
We have the following set-theoretic calculations:
\begin{align*}
	\Omega_2 \setminus \left(\Omega_1 \cup \bigcup_{i\in I}A_i \cdot \Omega_1\right) & = \Omega_1 \cap \left(\Lambda_1 \cap \bigcap_{i\in I}A_i\cdot \Lambda_1  \right)\\
	& = \Omega_2 \cap \bigcap_{i\in I}\cW\cS_i.
\end{align*}
Consider now a point $x\in \Omega_{2}\cap \cW\cS_{i_0}$.
The assumption $x\in \Omega_2$ is equivalent to the statement that for any $i_1\neq i_2$ we have that $A_{i_2}x\notin \ov{\cL_{i_1,s}}$.
Observe that if $i_2$ is not adjacent to $i_0$ then this is automatic since the reflection $A_{i_2}$ will map $\cL_{i_0,s}$ strictly inside $\cL_{i_2,s}$, and so the same will remain true of the boundary $\cW\cS_{i_0}$.

So we have to consider the cases $i_2\in\{r(i_0),l(i_0)\}$.
By an analogous reasoning, if $i_1\neq i_0$ then we have $A_{i_2}\cW\cS_{i_0}\cap \ov{\cL_{i_1,s}}=\emptyset$, so we have to consider only the case $i_1=i_0$.
A point $x\in \Omega_2\cap \cW\cS_{i_0}$ is characterized by
\[
	A_{r(i_0)}x\notin \ov{\cL_{i_0,s}} \text{ and } A_{l(i_0)}x\notin \ov{\cL_{i_0,s}}
\]
which, by applying $A_{i_0}$ to both sides, and using the notation from \autoref{sssec:left_and_right_adjacency}, is equivalent to
\[
	T_{i_0,r}x\notin \ov{\cL_{i_0,b}} \text{ and }
	T_{i_0,l}x\notin \ov{\cL_{i_0,b}}
\]
In other words we have
\[
	\Omega_2\cap \cW\cS_{i_0} = \cW\cS_{i_0}\setminus \left(T_{i_0,r}\cdot \cW\cS_{i_0}\cup T_{i_0,l}\cdot \cW\cS_{i_0}\right)
\]
In other words, we must eliminate the intersections of the original crooked surface with its translates by two unipotent transformations.
These are precisely the sets described in \autoref{prop:tangency_of_crooked_surfaces}: one intersection is along a full wing of the surfaces, while another is along a single photon.

\subsubsection{Action of reflections on a wing}
	\label{sssec:action_of_reflections_on_a_wing}
It follows from the previous analysis that the sets $\Lambda_n$ will contain full wings of adjacent crooked surfaces, along which the sources ``touch'' in the sense of \autoref{prop:tangency_of_crooked_surfaces}.
Let us analyze now the dynamics of the two reflections which fix, as a set, the particular wing.
Up to conjugacy, the model is that of the group generated by the matrices $A,B$ from \autoref{sssec:the_reflection_matrices}.
Their product $AB=(A-B)B + 1$ is a rank $1$ unipotent matrix, since $A-B$ is visibly a rank $1$ matrix.

Recall also that the cone is given in \autoref{eqn:fundamental_cone_definition} and its column vectors are (up to scaling) what we called $e_1,f_1,e_2,f_2$.
Then the geometry is as follows.
The group generated by $A,B$ preserves as a set each photon $\phi(v)$ for $v=\alpha_1e_1 + \alpha_2 e_2$.
The two reflections fix the Lagrangian $L_{e_1e_2}$, which lies on each of the photons in question.
Identifying $\phi(v)=\bP(v^{\perp}/v)\isom \bP^1(\bR)$, and removing ``the point at infinity'' $L_{e_1e_2}$, the action of $A,B$ then becomes that of two Euclidean reflections on $\bR$.
Under increasingly longer words in $A,B$ the orbit of a point approaches the point at infinity $L_{e_1e_2}$.

The above description holds except for a photon $\phi(v_0)$, where $v_0$ is the image of $AB-\id$.
The action of the group generated by $A,B$ is trivial on this photon, and the vector $v_0$ is the ``attractor'' for the projective action of large powers of $AB$.
Let us call $\phi(v_0)\subset \cW_e$ the ``attractor photon'' on the corresponding $e$-wing.

\subsubsection{Enlarging the domain of discontinuity}
	\label{sssec:enlarging_the_domain_of_discontinuity}
We can now enlarge our open set $\Omega^{\circ}$ to a larger domain of discontinuity, as follows.
Enlarge $\Omega_2$ by adding, for each index $i\in I$, the complement in the $e$-wing of the ``attractor photon''.
Then, take the image of $\Omega_2$ under the group $\tilde{\Gamma}_N$ and call the resulting set $\Omega$, with complement $\Lambda$.

This construction is equivalent to removing from the limit set $\Lambda_2$ the $e$-wings, except for the attractor photons, and then taking successive images and intersecting as in \autoref{sssec:finite_approximations_to_limit_set_and_domain_of_discontinuity}.

Let us finally remark that the limit set $\Lambda$ intersects each crooked surface $\cW\cS_i$ in two photons only, namely the $f$-vertex photon which is the attractor for the maximally unipotent matrix, and another photon attractor for the rank $1$ unipotent transformation on the $e$-wing.

\begin{theorem}[Proper discontinuity]
	\label{thm:proper_discontinuity}
	The action of $\tilde{\Gamma}_N$ on the open sets $\Omega^{\circ}$ and $\Omega$ in $\LGr(V)$ is properly discontinuous.
\end{theorem}
\begin{proof}
	It suffices to restrict our attention to the finite index subgroup of $\tilde{\Gamma}_N$ generated by the reflections $A_i$ for $i\in I$.
	The proper discontinuity for the action on $\Omega^{\circ}$ follows from its construction in \autoref{sssec:finite_approximations_to_limit_set_and_domain_of_discontinuity} and the mapping properties of the reflections $A_i$ for the surfaces $\cW\cS_i$ and their configurations.
	The orbit of a point $x\in \Omega^{\circ}$, and a sufficiently small open set containing it, can be traced combinatorially through the open sets $\Omega_n\subset \Omega^{\circ}$ just like the corresponding orbit of a point in the hyperbolic plane for the corresponding action of the reflection group there.

	The only new points in $\Omega$ are those that belong to initial $e$-wings, plus their images.
	As explained in \autoref{sssec:action_of_reflections_on_a_wing}, the action of the dihedral group preserving that wing is properly discontinuous on the complement of the fixed photon, and if we apply a reflection not in the dihedral group, the point goes to the interior of domains $\cL_{i,s}$ and its subsequent orbit does not return to the initial neighborhood, again by the mapping properties of the cones and corresponding regions $\cL_{i,s/b}$ and reflections $A_i$.
\end{proof}



\bibliographystyle{sfilip_bibstyle}
\bibliography{cyclotomic_cones}

\end{document}